\documentclass[11pt]{amsart}

\usepackage{amsmath,amsthm,verbatim,amssymb,amsfonts,amscd, graphicx, enumerate, enumitem,url,hyperref}
\usepackage{graphics}
\usepackage{tikz-cd}
\usepackage{pinlabel}
\usepackage{xcolor}
\usepackage{stackrel}
\usepackage{mathtools,calc}
\usepackage[margin=1.5in]{geometry}%\theoremstyle{plain}

\newtheorem{theorem}{Theorem}[section]
\newtheorem{corollary}[theorem]{Corollary}
\newtheorem{lemma}[theorem]{Lemma}
\newtheorem{proposition}[theorem]{Proposition}

\newtheorem*{lemstong1}{Lemma~\ref{lem:stong1}}
 
\newtheorem{question}[theorem]{Question} 
 
\theoremstyle{definition}
\newtheorem{definition}[theorem]{Definition}
\newtheorem{remark}[theorem]{Remark}
\newtheorem{observation}[theorem]{Observation}

\newtheorem*{mainthm1}{Theorem~\ref{thm:main1}}

\newtheorem*{mainthm3}{Theorem~\ref{thm:main3}}
\newtheorem*{finitecor}{Corollary~\ref{cor:finite}}
\newtheorem*{slicelink}{Theorem~\ref{thm:slicelink}}

\theoremstyle{definition}

\newtheorem{example}[theorem]{Example}

\newtheorem{conclusion}[theorem]{Conclusion}

\newcommand{\nocontentsline}[3]{}
\newcommand{\tocless}[2]{\bgroup\let\addcontentsline=\nocontentsline#1{#2}\egroup}

\newcommand{\leftrarrows}{\mathrel{\raise.75ex\hbox{\oalign{%
  $\scriptstyle\leftarrow$\cr
  \vrule width0pt height.5ex$\hfil\scriptstyle\relbar$\cr}}}}
\newcommand{\lrightarrows}{\mathrel{\raise.75ex\hbox{\oalign{%
  $\scriptstyle\relbar$\hfil\cr
  $\scriptstyle\vrule width0pt height.5ex\smash\rightarrow$\cr}}}}
\newcommand{\Rrelbar}{\mathrel{\raise.75ex\hbox{\oalign{%
  $\scriptstyle\relbar$\cr
  \vrule width0pt height.5ex$\scriptstyle\relbar$}}}}

\usepackage[utf8]{inputenc}

\makeatletter
\def\leftrightarrowsfill@{\arrowfill@\leftrarrows\Rrelbar\lrightarrows}
\newcommand{\xleftrightarrows}[2][]{\ext@arrow 3399\leftrightarrowsfill@{#1}{#2}}
\makeatother

\definecolor{violet}{rgb}{.6,0,.6}
\definecolor{green}{rgb}{.0,.69,0}
\definecolor{brown}{rgb}{.63,.35,.17}
\definecolor{darkgreen}{rgb}{.0,.4,0}
\definecolor{darkred}{rgb}{.7,0,0}
\definecolor{darkerred}{rgb}{.4,0,0}
\definecolor{darkyellow}{rgb}{.4,.4,0}

\newcommand{\Self}{\text{Self}}

\let\int\relax
\newcommand{\int}{\mathring}

\newcommand{\boundary}{\partial}
\newcommand{\into}{\hookrightarrow}

\DeclareMathOperator{\lk}{{lk}}

\DeclareMathOperator{\pt}{{pt}}
\DeclareMathOperator{\fq}{{fq}}
\DeclareMathOperator{\km}{{stong}}
\DeclareMathOperator{\tw}{{tw}}
\DeclareMathOperator{\SO}{{SO}}

\DeclareMathOperator{\Concordance}{{Concordance}}

\newcommand*{\vtilde}[2][0pt]{% \vtilde[<lift>]{<stuff>}
  \setbox0=\hbox{$#2$}%
  \widetilde{\mathrlap{\phantom{\rule{\wd0}{\ht0+{#1}}}}\smash{#2}}%
}

\let\Im\relax
\DeclareMathOperator{\Im}{{Im}}

\title{Concordance of spheres in 4-manifolds with an immersed dual sphere}
\author{Michael Klug}
\address{Department of Mathematics\\University of Chicago\\Chicago, IL 60637, USA}
\email{michaelklug@uchicago.edu}
\author{Maggie Miller}
\address{Department of Mathematics\\Stanford University\\Stanford, CA 94301, USA}
\email{maggie.miller.math@gmail.com}

\thanks{MK was supported by MPIM Bonn and later by NSF grant DMS-2203023 (at Chicago). MM was supported by NSF grants DGE-1656466 (at Princeton) and DMS-2001675 (at MIT) and a research fellowship from the Clay Mathematics Institute (at Stanford).}
\begin{document}

\maketitle

\begin{abstract}
    Let $S_0$ and $S_1$ be two homotopic, oriented 2-spheres embedded in an orientable 4-manifold $X^4$.  After discussing several operations for modifying an immersion of a 3-manifold into a 5-manifold, we discuss two concordance obstructions $\fq(S_0,S_1)$ and $\km(S_0,S_1)$ which, when defined, are defined in terms of the self-intersection set of a regular homotopy from $S_0$ to $S_1$.  When $S_0$ has an immersed dual sphere, we see that under some mild topological conditions on $X$, the invariants $\fq$ and $\km$ are a complete set of concordance obstructions.  This work is an adaption of the methods of Richard Stong to the context of concordances of 2-spheres.  
\end{abstract}

\addtocontents{toc}{\protect\setcounter{tocdepth}{0}}

\section{Introduction}

We study the question of when two homotopic, oriented 2-spheres $S_0, S_1$ embedded in a 4-manifold $X^4$ are concordant. In the case that one of the 2-spheres admits an immersed dual (that is, in the case that there is an immersed 2-sphere intersecting $S_0$ in a single point) and the 4-manifold $X$ satisfies some mild conditions then there is a complete pair of obstructions to the existence of a concordance.

These invariants are the {\emph{Freedman--Quinn}} invariant $\fq$ (see~\cite{schneiderman_teichner, fq, km} and Section~\ref{sec:removing_type_II}) and the {\emph{Stong}} invariant $\km$ of Stong~\cite{stong} following Freedman and Quinn (see~\cite{stong},~\cite[Chapter 10]{fq}, and Section~\ref{sec:km}). If we were to closely follow the language of Stong, we would probably call the Stong invariant the Kervaire--Milnor invariant km -- but this is somewhat misleading as it has very little to do with the work of Kervaire or Milnor, so we instead refer to it as the Stong invariant.\footnote{This also neatly avoids any reader thinking we intended km to stand for Klug--Miller.} The invariant $\km$ is secondary to $\fq$, in the sense that $\km(S_0, S_1)$ cannot be defined if $\fq(S_0,S_1)$ is non-vanishing. %(and furthermore, we need $S_0$ and $S_1$ are $s$-characteristic, as discussed below).  
If the spheres $S_0$ and $S_1$ are concordant, then both $\fq(S_0,S_1)$ and (when it is defined) $\km(S_0,S_1)$ must vanish.  When one of $S_i$ admits an immersed dual and $X$ satisfies some mild conditions, the 2-spheres $S_0, S_1$  are concordant if and only if $\fq$ and $\km$ vanish. 
This follows the usual trend that stabilization (here in the form of a dual sphere, which can be introduced by various stabilization operations) greatly simplifies 4-dimensional topology. %Note that we are not placing any restriction on the normal bundle of the dual sphere, although its euler number will be relevant in precise theorem statements.% (\red{cite some other stuff}).  

\subsection*{The Freedman--Quinn and Stong invariants}

%Let $S_0$ and $S_1$ be the two homotopic embedded based spheres in the 4-manifold $X^4$. 
Briefly, the Freedman--Quinn invariant $\fq(S_0,S_1)$ is defined by %c%onsidering an immersion $H: S^2 \times I \to X \times I$ with $H(S^2 \times \{0\}) = S_0 \times \{0\}$ and $H(S^2 \times \{1\}) = S_1 \times \{1\}$ and
considering the 1-dimensional link that is the preimage of the self-intersection set of a singular concordance between $S_0$ and $S_1$. The following definitions are of critical importance to this paper. %In the definition of $\fq$, the immersion $H$ must also be based, so that $H(p\times I)=z\times I$ for a basepoint $p\in S^2$ and $z\in X$. 

\begin{definition}
Let $S_0$ and $S_1$ be embedded, oriented spheres in a 4-manifold $X$.  A regular immersion $H : S^2 \times I \to X^4 \times I$ with $H(S^2 \times \{0\}) = S_0$ and $H(S^2 \times \{1\}) = S_1$, with $S^2\times I$ oriented so that $\partial H(S^2\times I)$ induces the correct orientation on $S_1$ and the opposite orientation on $S_0$, is called a \emph{singular concordance} from $S_0$ to $S_1$.  If $S_0$ and $S_1$ share a basepoint $z \in X$, and $H$ is a singular concordance from $S_0$ to $S_1$ where there exists some $p \in S^2$ with $H(p,t) = (z,t)$ for all $t \in I$, then we call $H$ a \emph{based singular concordance} from $S_0$ and $S_1$.  
\end{definition}

\begin{remark}
If $S_0$ and $S_1$ are (based) homotopic, then the trace of a (based) regular homotopy from $S_0$ to $S_1$ yields a (based) singular concordance from $S_0$ to $S_1$. (Recall that any two homotopic, embedded 2-spheres in a 4-manifold are also regularly homotopic by work of Hirsch~\cite{hirsch} and Smale~\cite{smale} in the smooth category; see~\cite[Chapter 1]{fq} in the topological category.) %projection from $X\times I\to I$ may restrict to a Morse function with many critical points -- that is, w
However, we do not generally expect a singular concordance $H$ to be the trace of a regular homotopy.
\end{remark}

The following weakening of the usual condition that a surface in a 4-manifold is characteristic will be essential in what follows.

\begin{definition}
A surface $F$ in a 4-manifold $X$ is {\emph{$s$-characteristic}} (or {\emph{spherically-characteristic}}) if for every 2-sphere $S$ immersed in $X$, we have \[[F] \cdot [S] \equiv [S] \cdot [S]\pmod{2}.\] Similarly, a 3-manifold $M$ properly embedded in a 5-manifold $W$ is $s$-characteristic if for every 2-sphere $S$ embedded in $W$, we have \[[M]\cdot[S]\equiv[S]\cdot[S]\pmod{2}.\]
\end{definition}

%From now on, we will often refer to a \emph{based} immersion $H : S^2 \times I \to X^4 \times I$ with $H(S^2 \times \{0\}) = S_0$ and $H(S^2 \times \{1\}) = S_1$, where by based, we mean that there are basepoints $p\in S^2$ and $z\in X$ with $H(p\times I)=z\times I$.

\begin{observation} \label{lem:4to5}
If $H$ is a singular concordance from $S_0$ to $S_1$, then $S_0$ (and hence $S_1$) is $s$-characteristic in $X$ if and only if $H(S^2\times I)$ is $s$-characteristic in $X\times I$.

This is clear because any 2-sphere $G$ in $X\times I$ can be projected to an immersed 2-sphere in $X$, and similarly any 2-sphere immersed in $X\times 0$ can be pushed into $X\times I$ and isotoped to be embedded. 

In particular, if there is a 2-sphere $G$ in $X\times I$ with trivial normal bundle that intersects $H(S^2\times I)$ transversely once, then $S_0$ (and hence also $S_1$) is not $s$-characteristic.
%Let $p : X \times I \to X \times \{0\} = X$ be the projection map.  $S_0$ is $s$-characteristic in $X$ if and only if, for every 2-sphere $S$ embedded in $X \times I$, we have \[|S \cap H(S^2 \times I)| =\begin{cases}0\pmod{2}&\text{if $S$ has trivial normal bundle,}\\1\pmod{2}&\text{if $S$ has non-trivial normal bundle.}\end{cases}\]
%Note in particular that if there exists a 2-sphere $G$ in $X\times I$ with trivial normal bundle that intersects $H(S^2\times I)$ transversely once, then $S_0$ (and hence also $S_1)$ is not $s$-characteristic.
\end{observation}

Given a singular concordance $H:S^2\times I\to X\times I$ from $S_0$ to $S_1$, let $L$ be the link in $S^2\times I$ that is the preimage of the self-intersection set of $H$. %of a singular concordance $H:S^2\times I\to X\times I$ between $S_0$ and $S_1$.
We call $L$ the {\emph{singular link}} of $H$. We can associate an element of $\pi_1 X$ to every component of $L$ via a sheet-changing based loop. The invariant $\fq(S_0,S_1)$ counts the fundamental group elements corresponding to circles that double-cover their images under $H$ as an element in a certain quotient of  $\mathbb{Z}[\pi_1 X]/\langle g+g^{-1},1\rangle$; %$L$ of self-intersection circles in $S^2 \times I$.  In particular, $\fq(S_0,S_1)$ counts the fundamental group elements (given by sheet-changing loops) circles that double cover their images in $L$.   A detailed description of the Freedman--Quinn invariant is deferred to 
see Section~\ref{sec:removing_type_II} for a detailed description.  As we discuss in Section~\ref{sec:removing_type_II}, in general $H$ must be a based singular concordance, although in the presence of dual spheres, this assumption is not necessary (see also~\cite{schneiderman_teichner} and~\cite{km})

The Stong $\km(S_0, S_1)$ invariant is more complicated but essentially counts the group elements associated to $L$ (viewed as elements in %the different fundamental group elements  of different singular circles (again given by sheet-changing loops but now only considered as elements in
$H_1(X; \mathbb{Z}/2\mathbb{Z})$) weighted by linking numbers of the various components of $L$. A relatively complete description of an analogous invariant in the case of a 3-sphere in a 5-manifold can be found in~\cite{stong}; we have further exposition of this in progress and will release it in a forthcoming paper~\cite{km2}.   %between different components of $L$.  Fro a complete discussion of the Stong invariant we refer to \red{part 2} --
Here is a sufficient description for a first-pass at the our main results is as follows (for a complete discussion, see Section~\ref{sec:km}).

\begin{enumerate}[itemsep=1ex]
    \item The invariant $\km(S_0,S_1)$ takes values in a quotient of $H_1(X; \mathbb{Z}/2\mathbb{Z})$.  In certain cases (for example, when $\pi_3(X)=0$), we quotient by nothing and in fact $\km(S_0,S_1) \in H_1(X; \mathbb{Z}/2\mathbb{Z})$. We discuss this further in Section~\ref{sec:km}.
    
    \item If $\km(S_0, S_1) \neq 0$, then $S_0$ and $S_1$ are not concordant.
    
    \item If $H$ is as above, with $S_0$ and $S_1$ $s$-characteristic, $\fq(S_0, S_1) = 0$ and $L$ is a Hopf link in a ball in $S^2 \times I$ with the image of both knots equal, then $\km(S_0, S_1)$ is the sheet-changing group element associated to $L$ (considered in a certain quotient of $H_1(X; \mathbb{Z}/2\mathbb{Z})$). 
\end{enumerate}

%As with the Freedman--Quinn invariant, the presence of the dual sphere $G$ allows us to drop consideration of basepoints.  

%Note, in our setup, since $S_0$ and $S_1$ are homologous, $S_0$ is $s$-characteristic if and only if $S_1$ is $s$-characteristic.  

\subsection*{Main Theorems} We now state the main results of the paper, delaying a more detailed discussion of the invariants until Section~\ref{sec:fqkm}.

\begin{theorem}[Case where spheres are not $s$-characteristic]\label{thm:main1}
Suppose that $S_0$ and $S_1$ are embedded, oriented, homotopic 2-spheres in an orientable 4-manifold $X$ such that $S_0$ has an immersed dual sphere $G$ in $X$ (i.e., $G$ and $S_0$ intersect in a single point) and $S_0$ is not $s$-characteristic. Then $S_0$ and $S_1$ are concordant if and only if $\fq(S_0,S_1) = 0$.
\end{theorem}

This in particular generalizes~\cite[Theorem 1.4]{km}.  In the case where $S_0$ is $s$-characteristic, we will observe a secondary obstruction $\km(S_0,S_1)$.  Furthermore, in this case, some mild restrictions must be placed on $X$ in order for the vanishing of $\fq(S_0,S_1)$ and $\km(S_0,S_1)$ to imply that $S_0$ and $S_1$ are concordant.  The need for such conditions arises since both $\fq(S_1,S_0)$ and $\km(S_0,S_1)$ live in quotients and the first step in attempting to find a concordance between $S_0$ and $S_1$ is to start with a singular concordance and modify it so as to make the pre-quotiented versions of $\fq$ and $\km$ simultaneously vanish for this singular concordance -- which may not be possible even when $\fq(S_1,S_0)$ and $\km(S_0,S_1)$ both vanish, without any additional hypotheses.  %Additional hypotheses are used to achieve this for both invariants simultaneously.   

In what follows, the map $\mu$ is the self-intersection number of a 3-sphere in a 6-manifold (where here the 6-manifold is $X \times I \times I$) -- see~\cite{schneiderman_teichner},~\cite{km}, or Section~\ref{sec:fqkm} for additional discussion.  

\begin{theorem}[Case where spheres are $s$-characteristic]\label{thm:main2}
Suppose that $S_0$ and $S_1$ are embedded, oriented, homotopic 2-spheres in an orientable 4-manifold $X$ such that $S_0$ has an immersed dual sphere $G$ in $X$ and $S_0$ is $s$-characteristic. Assume that $\mu(\pi_3(X))=0$, e.g., by assuming either $\pi_1 X$ has no 2-torsion or $\pi_3(X)=0$. 
    Then $S_0$ and $S_1$ are concordant if and only if $\fq(S_0,S_1) = 0$ and $\km(S_0, S_1) = 0$.
\end{theorem}

\begin{remark}
In fact, since the map $\mu$ factors through the Hurewicz homomorphism $\pi_3(X)\twoheadrightarrow H_3(\vtilde[0pt]{X};\mathbb{Z})\cong H_3(X;\mathbb{Z}\pi_1X)$ (see~\cite[Lemma 4.2]{schneiderman_teichner}), it suffices to assume that $H_3(X;\mathbb{Z}\pi_1X)$ is trivial for the conclusion of Theorem~\ref{thm:main2} to hold.  
\end{remark}

We actually prove Theorem~\ref{thm:main2} by proving the following theorem, which is currently harder to parse.

\begin{theorem}\label{thm:main3}
Suppose that $S_0$ and $S_1$ are embedded, oriented, homotopic 2-spheres in an orientable 4-manifold $X$ such that $S_0$ has an immersed dual sphere $G$ in $X$ and $S_0$ is $s$-characteristic. Then $S_0$ and $S_1$ are concordant if and only if there exists a singular concordance $H$ from $S_0$ to $S_1$ so that $\mu(H)=0$ and $\Delta(H)=0$. %Then  Assume that $\mu(\pi_3(X))=0$, e.g., by assuming either $\pi_1 X$ has no 2-torsion or $\pi_3(X)=0$. 
 %   Then $S_0$ and $S_1$ are concordant if and only if $\fq(S_0,S_1) = 0$ and $\km(S_0, S_1) = 0$.
\end{theorem}

We give the following applications in Section~\ref{sec:example}. The first suggests the importance of constructing spheres realizing $\km$.

\begin{theorem}\label{thm:slicelink}
%Let $S_0$ be the unknotted 2-sphere in $B^3\times S^1$.
If there are 2-spheres $S_0,S_1$ in $B^3\times S^1$ with $\km(S_0,S_1)=1$ in $H_1(B^3\times S^1;\mathbb{Z}/2\mathbb{Z})\cong\mathbb{Z}/2\mathbb{Z}$ then there is a 2-component link of spheres in $S^4$ that is not concordant to the unlink.
\end{theorem}

We give examples of pairs of homotopic 2-spheres with specified Stong invariant in some other 4-manifolds in Example \ref{km_example} (actually summarizing our previous work \cite{km} and adding some more discussion), but not in $B^3\times S^1$.

\begin{corollary}[Corollary of Theorems~\ref{thm:main1} and~\ref{thm:main2}]\label{cor:finite}
Let $S_0$ be an oriented embedded 2-sphere in an orientable 4-manifold $X$ with an immersed dual sphere $G$.  Let $\Concordance(S_0)$ be the set of concordance classes of embedded spheres in $X$ that are homotopic to $S_0$. Suppose that $\pi_1 X$ has a finite number of 2-torsion elements.
\begin{enumerate}
    \item Suppose $S_0$ is not $s$-characteristic.  Then $\Concordance(S_0)$ is finite of size at most $2^{|T_X|}$.
    \item Suppose $S_0$ is $s$-characteristic and that $\mu(\pi_3 X) = 0$.  Then $\Concordance(S_0)$ is finite of size at most $2^{|T_X|} \cdot |H_1(X; \mathbb{Z}/2\mathbb{Z})|$.
\end{enumerate}
\end{corollary}

Theorems~\ref{thm:main1},~\ref{thm:main2},~\ref{thm:main3} are greatly indebted to Richard Stong and closely follow the arguments given by Stong for the case of 3-spheres in a 5-manifold~\cite{stong}, and the non-$s$-characteristic cases follow from Freedman and Quinn~\cite{fq}.  Stong's work draws many ideas from and corrects an omission in the work of Freedman and Quinn.  Stong and Freedman--Quinn develop several manipulations of 3-manifolds immersed in 5-manifolds (in our setting, this will be the image of $H$ inside of $X \times I$) that allow for the modification of a singular link. Our construction of the $\km$ invariant closely follows Stong's, including making use of his many explicit operations for modifying the immersion so as to modify the singular link in a prescribed manner.  In particular, all of our moves described in Section~\ref{sec:moves} are contained in Stong's work~\cite{stong}.  We provide many details, add further exposition of his work, and correct some minor errors.  In addition, Propositions~\ref{fq_typeII},~\ref{prop:LHopflink},~\ref{prop:trivialremove}, and Lemmas~\ref{lem:notschar},~\ref{lem:Z/2_abelianization},~\ref{lem:addHopfnullhomologous}, as well as the definition of $\km$ are all due to Stong~\cite{stong} (often drawing from~\cite{fq}) in the closely related context of 3-spheres in 5-manifolds.  

%\begin{definition}
%Given connected manifolds $Y$ and $W$, a map $f : V \to Z$ is \emph{$\pi_1$-trivial} if the resulting map on $\pi_1$ is trivial.  A connected submanifold $Y$ of $W$ is \emph{$\pi_1$-trivial} if the inclusion map $i : Y \to W$ is $\pi_1$-trivial.  
%\end{definition}

%Given two homotopic genus-$g$ closed $\pi_1$-trivial surfaces $F_0$ and $F_1$ embedded in a 4-manifold $X$ where there is an immersed dual sphere $G$ intersecting $F_0$ at a single point, then $\fq(F_0, F_1)$ is defined analogously as for spheres (see~\cite[Section 6]{km} and Section \label{sec:removing_type_II} for details).  We generalize the main result of~\cite{km} as follows.

\subsection*{Historical overview}We now give a short overview of the previous results leading up to this work, together with an overview of some of the relevant definitions and an outline of the present work.  All results hold in both the smooth and topological category. Any obstructions to concordance obstruct even locally flat concordances, and if $S_0$ and $S_1$ are smoothly embedded 2-spheres then any constructed concordance may be taken to be smooth (while if $S_0$ and $S_1$ are locally flat, then the constructed concordance can be taken only to be locally flat). This holds true for the remainder of the paper; the reader may choose to live in either the topological or smooth category.

Kervaire~\cite{kervaire} proved that, in sharp contrast to the analogous situation in dimension 3, all embedded 2-spheres in $S^4$ are concordant. This was later generalized to pairs of homologous genus-$g$ surfaces in simply-connected 4-manifolds by Sunukjian~\cite{sunukjian}, to the case of homotopic 2-spheres in non-simply-connected 4-manifolds with the additional hypotheses that $\fq(S_0,S_1)=0$ and $S_0$ has a framed dual by Freedman--Quinn~\cite{fq} (see also the present authors' previous work~\cite{km} for alternate proof), and to the case that $S_0, S_1$ are $\pi_1$-trivial homotopic positive-genus surfaces related by a homotopy $H$ with $\mu(H)\in\mu(\pi_3(X))$%\footnote{See Section~\ref{sec:removing_type_II} for a definition of the self-intersection invariant $\mu$. When $S_0$ and $S_1$ are 2-spheres, $\mu(H)\in\mu(\pi_3(X))$ implies $\fq(S_0,S_1)=0$. When $S_0, S_1$ are positive-genus, their Freedman--Quinn invariant is not defined so we write this instead in terms of a homotopy. See Section~\ref{sec:problems} for more discussion on positive genus.}
(which holds automatically when e.g.,\ either $\pi_1(X)$ has no 2-torsion or $H_3(X;\mathbb{Z}\pi_1X)=0$) and $S_0$ has a framed dual sphere by the present authors~\cite{km}. % for the simply connected case and the authors' previous paper~\cite{km} for a more general case, in which $\fq=0$ and the dual is framed.
In this paper, we will recover a more general version of the main theorem of~\cite{km} in the case of 2-spheres.

%However, concordance questions in general can be very subtle without the presence of a dual sphere or when the fundamental group of a surface maps nontrivially into the ambient manifold.

%The existence of a dual sphere is a somewhat popular hypothesis in the study of concordance.
%Gabai's 4-dimensional lightbulb theorem~\cite{gabai} was extended by Scheiderman and Teichner~\cite{schneiderman_teichner} to address issues of 2-torsion by considering the Freedman--Quinn invariant.  In section 7 of their paper, Schneiderman and Teichner give another proof of their generalization of Gabai's lightbulb theorem using the work of Freedman--Quinn and Stong.  In particular, in they
Schneiderman and Teichner~\cite{schneiderman_teichner} give an alternate proof of Gabai's lightbulb theorem~\cite{gabai} by way of starting with a concordance between two homotopic 2-spheres $S_0,S_1$ in a 4-manifold $X$ and modifying the concordance to produce an isotopy. They assume that there is another 3-sphere $G$ embedded in $X$ with trivial normal bundle such that each $S_i$ intersects $G$ transversely once, and that $\fq(S_0,S_1)=0$. They invoke work of Stong~\cite{stong} (here making use of the Freedman--Quinn invariant vanishing) to obtain a concordance from $S_0$ to $S_1$ in $X\times I$ such that in each $X\times\pt$ the concordance intersects $G\times\pt$ transversely once. Using these embedded dual spheres at every level, Schneiderman and Teichner are able to convert the concordance into an isotopy and thus re-prove the lightbulb theorem (which says that $S_0$ and $S_1$ are isotopic).  %We will work in the setting where there is still a dual sphere, however it is not necessarily embedded and the concordance that we produce will not necessarily be able to be converted into isotopies.  

Our work here is an exposition of the techniques in Stong~\cite{stong} and Freedman--Quinn~\cite{fq} (in particular Section 10.7) together with clarification of some of the arguments and discussion of how these arguments apply to questions regarding concordance of 2-spheres in 4-manifolds.

\subsection*{Outline}
In Section~\ref{sec:moves}, we describe many moves for modifying a $\pi_1$-trivial immersion $H : Y^3 \to W^5$. In Section~\ref{sec:fqkm} we discuss the Freedman--Quinn and Stong invariants. In Section~\ref{sec:reduce_circles}, we move to the setting of singular concordances with dual spheres and prove Theorems~\ref{thm:main1},~\ref{thm:main2}, and~\ref{thm:main3}. %we assume that $H$ is a singular concordance between 2-spheres $S_0,S)1\subset X$.  Further, we assume $H$ has an immersed dual sphere.  After reviewing the Freedman--Quinn invariant, we modify $H$ using the aforementioned moves to simplify its singular link $L$.  In particular, we remove all of the singular circles and obtain a concordance when $\fq(S_0,S_1) = 0$ and $S_0$ is not $s$-characteristic.  When $\fq(S_0,S_1) = 0$ and $S_0$ is $s$-characteristic, there is an additional obstruction to concordance, namely, $\km(S_0,S_1)$.  When $\km(S_0,S_1) = 0$, we show how to further modify $H$ to obtain a concordance, thus showing $\fq$ and $\km$ together are complete invariants for concordance of 2-spheres in 4-manifolds in the presence of a dual sphere.  We also give
In Section~\ref{sec:example} we give some examples, discuss the relationship to concordance of links of 2-spheres, and prove the finiteness (with explicit upper bounds) of the set of concordance classes of embedded 2-spheres in a given homotopy class under some conditions.

%as originally studied by Stong~\cite{stong}. %These include Whitney moves \S\ref{sec:whitney}, clean Whitney moves \S\ref{sec:cleanwhitney}, Whitney moves along glued disks \S\ref{sec:gluingwhitneydisks}, introducing a new component to $L$ that double covers its image under $H$ \S\ref{sec:maketrivialtypeII}, effecting clasp moves on $L$ \S\ref{sec:clasp}, finger moves \S\ref{sec:finger_around_an_element} and ambient Dehn surgery \S\ref{sec:ambient_dehn}.

\subsection*{Acknowledgements}
This work would not have been possible without Rob Schneiderman's introduction to and overview of Richard Stong's work over Zoom in Fall 2020.  In addition, we would like to thank Rob Schneiderman and Peter Teichner for their encouragement and helpful technical comments that shaped these results and Mark Powell for his remarks that resulted in Section~\ref{sec:problems}, in which we describe an error in~\cite[Lemma 6.1]{km} (and hence~\cite[Proposition 6.2]{km}). See Section~\ref{sec:problems} for details.

We would like to further emphasize that many of the constructions in this paper follow the work of Stong in~\cite{stong}. While we often give more details (see for example Sections~\ref{sec:whitney},~\ref{sec:cleanwhitney}, and~\ref{sec:gluingwhitneydisks} on Whitney disks) or expand upon Stong's constructive techniques while adapting them to the setting of concordance, the importance of Stong's work in this paper cannot be overstated.

\tableofcontents

\addtocontents{toc}{\protect\setcounter{tocdepth}{2}}

\section{Moves for modifying the immersion \texorpdfstring{$H$}{H}} 

\label{sec:moves}

Given connected manifolds $Y$ and $W$, a map $H : 
Y \to W$ is \emph{$\pi_1$-trivial} if the resulting map on $\pi_1$ is trivial.  A connected submanifold $Y$ of $W$ is \emph{$\pi_1$-trivial} if the inclusion map $i : Y \to W$ is $\pi_1$-trivial.  

Let $Y^3$ be an oriented, compact 3-manfold, $W^5$ an orientable 5-manifold, and $H : Y \to W$ a $\pi_1$-trivial, regular immersion that restricts to an embedding $\partial Y\to\partial W$. The preimage of the self-intersection set of $H$ is a link $L$ in $Y$. We refer to $L$ as the \emph{singular link} of $Y$ and the components of $L$ as \emph{singular circles}.  In this section, we give several techniques for modifying $H$ so as to modify its singular link.  In order to assign group elements to the singular circles we will need to choose a basepoint $z$ for $Y$ and specify $H(z)$ to be the basepoint of $W$. We take $z$ to be in the boundary of $Y$, which we keep constant in all moves on $H$ discussed in this section. We may write $z$ to denote $H(z)$; the context should make it clear whether we mean the basepoint of $Y$ or $W$.

%For some of the operations we introduce, it will be necessary for the image of $H$ to have a dual sphere and the framing of the dual sphere will affect the discussion.

%\begin{remark} \label{rmk:homologous}
%\red{Note that if $H'$ is obtained from $H$ by performing any of the surgery operations in the section, then $[H] = [H']$ in $H_3(W; \mathbb{Z})$ since the surgery carried out by attaching the handles defines an (immersed) manifold in $W \times I$ cobounding the image of $H$ considered in $W \times \{0\}$ and the image of $H'$ considered in $W \times \{1\}$.  Note that since $H$ was $\pi_1$-trivial, then so is $H'$, since every element in $\pi_1 Y$ can be made disjoint from the location of the surgery.  In addition, if $W$ has non-empty boundary then $[H] = [H']$ in $H_3(W, \partial W; \mathbb{Z})$.}
%\end{remark}

\begin{remark}\label{rem:remzpi1class}
The moves we discuss in this section change the immersion $H$ while fixing $H_{\partial S^2\times I}$ and preserving regularity and $\pi_1$-triviality of $H$.

By fixing a lift of the basepoint $z$ to the universal cover $\smash{\vtilde[0pt]{W}}$, we may consider $[H]\in H_3(W,H(\partial Y);\mathbb{Z}\pi_1 W)=H_3(\vtilde[0pt]{W},\smash{\vtilde[-.3mm]{H(\partial Y)}};\mathbb{Z})$, where $\smash{\widetilde{W}}$ is the universal cover of $W$ and $\smash{\vtilde[-.3mm]{H(\partial Y)}}$ consists of the lifts of $H(\partial Y)$ to $\smash{\widetilde{W}}$. The moves performed to $H$ discussed in this section also preserve this homology class.

In fact, all of the moves simply change $H$ by regular homotopy except for the ambient Dehn surgery operation in Section~\ref{sec:ambientdehn}, in which case we surger the image of $H$ along a 4-dimensional 2-handle and thus still preserve $\pi_1$-triviality and the $H_3(W,H(\partial Y);\mathbb{Z}\pi_1 W)$ homology class of the underlying immersion.

This will be important in later sections of the paper, when we will modify a singular concordance $H:S^2\times I\to X\times I$ while needing to preserve $[H]\in H_3(X\times I,X\times\{0,1\};\mathbb{Z}\pi_1X)$.
\end{remark}

To give a 4-dimensional motivation for this discussion, consider the following situation.  Let $S_0$ and $S_1$ be homotopic, oriented, embedded 2-spheres in a 4-manifold $X$. Let $H$ be a singular concordance from $S_0$ to $S_1$.  For example, $H$ can be the track of a regular homotopy from $S_0$ to $S_1$.  We would like to know when we can modify $H$ in the interior of $S^2\times I$ to obtain an embedding and thus obtain a concordance between $S_0$ and $S_1$.  In this case, we have $Y = S^2 \times I$ and $W = X \times I$.

\subsection{Preliminaries: the singular link}

%Generically, the self-intersection set of the image of $H$ is a 1-manifold of double points. %has a 1% Denote the preimage of this 1-manifold by $L \subset Y$ which we call the \emph{singular link} of $H$. (We call the components of $L$ {\emph{singular circles}}.) 
A singular circle $A$ in the singular link $L$ of $H$ can be one of two types.% come in two distinct types.
  \begin{enumerate}[label=Type \Roman*:,leftmargin=1in]
      \item We say $A$ is {\emph{type I}} if $H|_A$ is a homeomorphism. In this case, there is another singular circle $A'$ with $H(A)=H(A')$. We call $A'$ the {\emph{dual}} of $A$, and refer to $A$ and $A'$ together as a {\emph{dual pair}}.
      \item We say $A$ is {\emph{type II}} if $A$ double covers its image under $H$.
  \end{enumerate}%Each component $A$ of $L$ either double covers its image or its image has a disconnected preimage consisting of two singular circles, $A$ and another circle which we will denote by $A'$ and call the \emph{dual} of $A$, where $A$ and $A'$ both trivially cover their image.  We refer to $A$ as either a \emph{type I} or \emph{type II} circle depending on whether it single or double covers its image, respectively.  Additionally, we refer to components of the self-intersection set $H(L)$ as being type I or type II depending on if their preimage singular circles are type I or type II, respectively. 

%There is one move (ambient Dehn surgery, \S\ref{sec:ambient_dehn}) for which we will assume that $S_0$ has an immersed dual sphere $G$, but otherwise we make no such assumption (i.e., $G$ is immersed in $X$ and intersects $S_0$ in a single point).  

%Let $\ast$  Fix a basepoint $\ast \in S^2 \times I$ with $H(\ast)=z\times 0$. %arrange that $\ast \notin L$; take $H(\ast)=z\times 0$.

There is a natural (up to taking inverses) way of assigning an element of $\pi_1 W$ to every singular circle in $L$.

\begin{definition}

Let $A$ be in $L$. If $A$ is type I, then let $A'$ be its dual circle; if $A$ is type II then let $A'=A$.  Fix $x \in A, x' \in A'$ with $H(x) = H(x')$.

Let $\eta$ and $\eta'$ be directed arcs from the basepoint $z\in Y$ to $x$ and $x'$ (respectively), with the interiors of $\eta$ and $\eta'$ disjoint from $L$. Then $H(\eta)H(\overline{\eta'})$ is a closed loop in $X$ representing an element $a\in\pi_1 W$. We call $a$ the {\emph{group element associated to $A$ and $A'$}}.

Note that the choice of $\eta$ and $\eta'$ do not affect the class of $a$ in $\pi_1 W$, since $H$ is $\pi_1$-trivial. However, exchanging the roles of $x$ and $x'$ will replace $a$ by $a^{-1}$. In order to make sure $a$ is well defined, when $A$ is type I we fix an ordering of $A$ and $A'$, calling $A$ {\emph{active}} and $A'$ {\emph{inactive}}. Thus, when $A,A'$ are a dual type I pair, we construct a loop representing $a$ by traveling from the $z$ through $H(Y)$ to $H(A)=H(A')$ in the sheet containing $A$, and then leaving through the sheet containing $A'$ and returning to the $z$.

\end{definition}

From now on, we will label a dual pair of circles by a capital letter and the same capital letter prime, with the unprimed letter denoting the active circle and the primed denoting the inactive circle.

\begin{remark}
If $A$ is type II, then we may choose $\eta,\eta'$ so that %then when constructing $a$ we may take $x=x'$ and $\eta=\eta'$. Then 
$H(\eta)H(\overline{\eta'})$ is parallel to a copy of $H(A)$ (with a whisker chosen to the basepoint). Since $H$ is $\pi_1$-trivial and $A$ double covers $H(A)$ under $H$, we conclude that $a^2=1$. This fact, will be important in Section~\ref{sec:reduce_circles}.
\end{remark}

\subsection{Whitney moves} \label{sec:whitney}

We begin with the most involved move changing $H$, in which we homotope $H$ via a {\emph{Whitney disk}}. To improve readability, we divide this subsection into several pieces: finding a disk (\S\ref{sec:finddisks}), orientation restrictions on the framed boundary of a Whitney disk (\S\ref{sec:whitneyorient}), the $\omega_2$ obstruction to a framing of a circle extending over a bounded disk (\S\ref{sec:w2obstruction}), and some final remarks summarizing this section (\S\ref{sec:whitneyfinalremarks}). For now, we do not assume that the image of $H$ has an immersed dual -- we will add this assumption in Section~\ref{sec:cleanwhitney} and see how it simplifies the situation.

\subsubsection{Finding a disk}\label{sec:finddisks}

Let $A$ and $B$ be singular circles in $L$ with corresponding fundamental group elements $a=b$. Let $A',B'$ be the duals of $A$ and $B$ (if $A$ or $B$ is type II, then $A'$ or $B'$ (respectively) is simply equal to $A$ or $B$).

Let $x\in A$, $x'\in A'$, $y\in B$, $y'\in B'$ with $H(x)=H(x')$ and $H(y)=H(y')$. Let $\gamma_1$ be an arc in $Y$ from $x$ to $y$ that is disjoint in its interior from $L$ and let $\gamma_2$ be an arc in $Y$ from $y'$ to $x'$ that is disjoint in its interior from $L$. Let $\eta$ be an arc from the basepoint of $Y$ to $x$. Then $H(\eta)H(\gamma_1)H(\gamma_2)H(\overline{\eta})$ is a based loop in $X\times I$ representing $ab^{-1}=1\in\pi_1 W$ (see Figure~\ref{fig:gxgy}). We conclude that the unbased loop $H(\gamma_1)\cup H(\gamma_2)$ is nullhomotopic.

%Let $g_x \in \pi_1(X)$ be the element obtained by picking arcs from $\ast$ to $x$ and from $\ast$ to $x'$, orienting the first arc from $\ast$ to $x$ and the second from $x'$ to $\ast$, and let $g_x$ be the image of these arcs as an element of $\pi_1(X)$.  Similarly, define $g_y$.  Since $S^2 \times I$ is simply-connected, the choice of arcs from/to $\ast$ does not affect the resulting elements of $\pi_1(X)$.  By inspecting Figure~\ref{fig:gxgy} we see that $g_x = g_y$ if and only if $H(\gamma_1 \cup \gamma_2)$ is null-homotopic. 

Similarly, if $a=b^{-1}$ and $\gamma_1$ runs from $x$ to $y'$ while $\gamma_2$ runs from $y$ to $x'$, then $H(\gamma_1)\cup H(\gamma_2)$ is a nullhomotopic loop in $X\times I$.

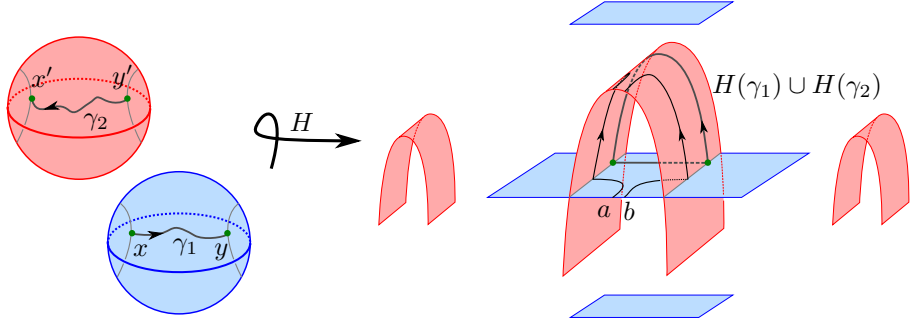
\begin{figure}
    \centering
    \labellist
\small\hair 2pt
\pinlabel $x'$ at 20 135
\pinlabel $y'$ at 65 135
\pinlabel $x$ at 75 38
\pinlabel $y$ at 120 37
\pinlabel $\gamma_2$ at 50 110
\pinlabel $\gamma_1$ at 100 40
\pinlabel $H$ at 165 110
\pinlabel $a$ at 334 60
\pinlabel $b$ at 348 60
\pinlabel $H(\gamma_1)\cup H(\gamma_2)$ at 442 130
\endlabellist
    \includegraphics[width=120mm]{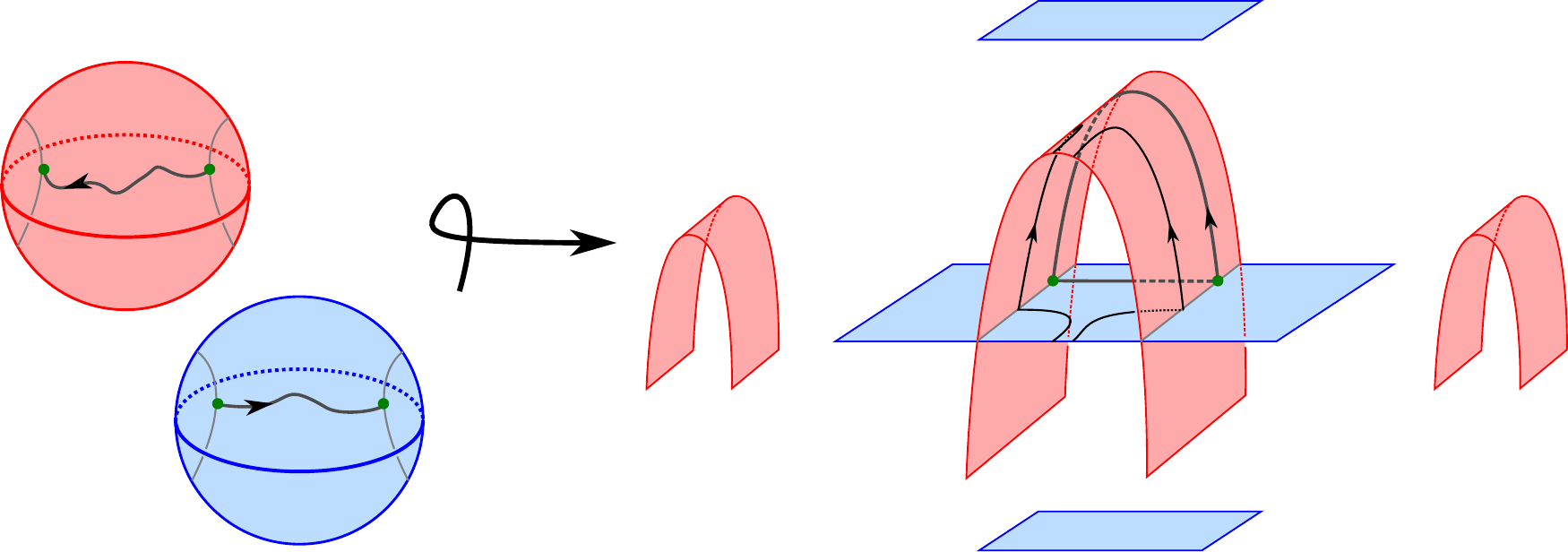}
    \caption{On the left, we draw two balls in $Y$ that contain points $x,y$ and $x',y'$ with $H(x)=H(x')$ and $H(y)=H(y')$. We draw arc $\gamma_1$ from $x$ to $y$ and $\gamma_2$ from $y'$ to $x'$. On the right, we see curves representing $a$ and $b$ that agree outside of the pictured region. We see that the closed curve $H(\gamma_1)\cup H(\gamma_2)$ is freely homotopic to a based curve representing $ab^{-1}$. Thus, if $a=b$, then $H(\gamma_1)\cup H(\gamma_2)$ is nullhomotopic.}
    \label{fig:gxgy}
\end{figure}
%Note that it can be the case that $x,x'$ or $y,y'$ are on type I or type II singular circles.  Note that changing the role of $x$ and $x'$ changes the element $g_x$ to the element $g_x^{-1}$.  
In either case, when $H(\gamma_1)\cup H(\gamma_2)$ is nullhomotopic it follows that there is an embedded disk $D\subset W$ with $\partial D=H(\gamma_1)\cup H(\gamma_2)$. Note that the interior of $D$ may not be disjoint from $H(Y)$.

%Now we assume that such points have been chosen so that $g_x = g_y$.  Since we are in 5-dimensions and $H(\gamma_1 \cup \gamma_2)$ is null-homotopic, there exists an embedded disk $D \subset X \times I$ with $\partial D = H(\gamma_1 \cup \gamma_2)$.  

We would like to have a local model as in Figure~\ref{fig:uncleanwhitney} so that we can perform the Whitney move across $D$. When this model exists, we call $D$ a {\emph{Whitney disk.}}

\begin{figure}
    \centering
    \includegraphics[width=130mm]{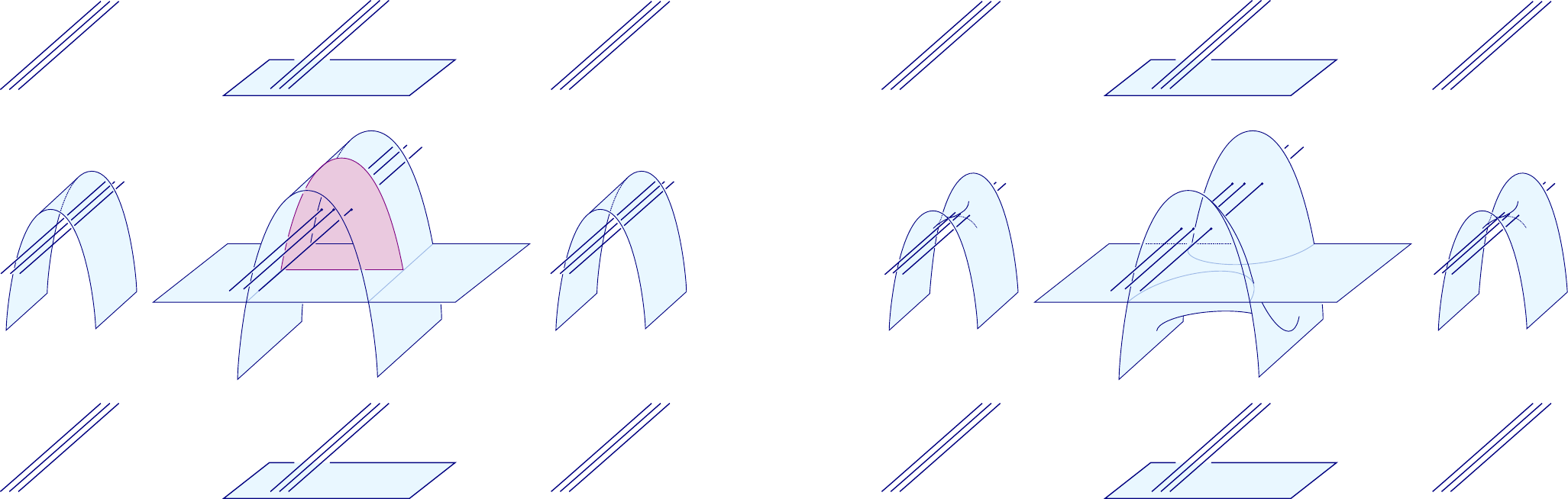}
    \caption{Left: we say a disk in $X$ with boundary in $H(Y)$ is a {\emph{Whitney disk}} if it admits a neighborhood as pictured. Right: We may perform a {\emph{Whitney move}} along a Whitney disk to change $H$ into a different regular immersion of $Y$ with a different self-intersection set.}
    \label{fig:uncleanwhitney}
\end{figure}

After performing the Whitney move along $D$, the effect to $L$ is the following (illustrated in Figure~\ref{fig:toothpickeffect}).
\begin{enumerate}
    \item We perform band surgery along $b_1$ and $b_2$ (these bands come from thickening the arcs comprising $H^{-1}(\partial D)$ via a framing of $D$; see Figure~\ref{fig:b1b2}).
    \item For each intersection of $H(Y)$ with the interior of $D$, we add a new pair of type I circles with trivial group element. One of these circles is a belt around $b_1$ while the other is unlinked from all components of $L$. The pairs of circles arising from multiple intersections of $H(Y)$ with $D$ are parallel, as in Figure~\ref{fig:toothpickeffect}.
\end{enumerate}

\begin{figure}
   \labellist
\small\hair 2pt
\pinlabel $\textcolor{blue}{b_2}$ at 10 25
\pinlabel $\textcolor{darkred}{b_1}$ at 65 25
\pinlabel {Whitney} at 110 27
\pinlabel {\tiny{$A_1$}} at 165 30
\pinlabel {\tiny{$A_2$}} at 165 23
\pinlabel {\tiny{$A_n$}} at 165 12
\pinlabel {\tiny{$A'_1$}} at 247.5 30
\pinlabel {\tiny{$A'_2$}} at 247.5 23
\pinlabel {\tiny{$A'_n$}} at 247.5 12
\endlabellist
    \centering
    \includegraphics[width=100mm]{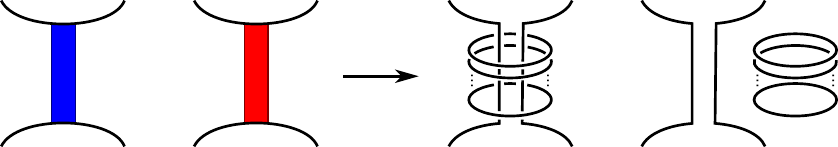}
    \caption{We illustrate the effect that performing a Whitney move has on the singular link $L$. The pictured bands $b_1$ and $b_2$ have image under $H$ as in Figure~\ref{fig:b1b2}, relative to the Whitney disk. The circles $A_1,\ldots, A_n$ are in correspondence with intersections of $H(Y)$ with the interior of the Whitney disk.}
    \label{fig:toothpickeffect}
\end{figure}
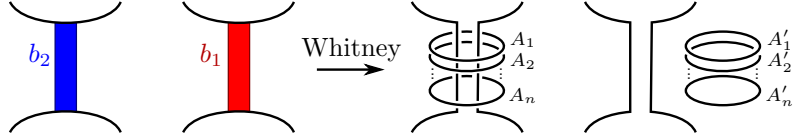

Because our goal is to modify $L$ in a prescribed way, we need to know when we can preordain the bands $b_1$ and $b_2$ in $Y$ with respective cores $\gamma_1$ and $\gamma_2$ so that $H(b_1)$ and $H(b_2)$ that appear in Figure~\ref{fig:b1b2}. That is, we need to understand what framings of $\gamma_1,\gamma_2$ give such bands.

\begin{figure}
    \centering
     \labellist
\small\hair 2pt
\pinlabel \textcolor{darkred}{$H(b_1)$} at 63 34
\pinlabel \rotatebox{-55}{\textcolor{blue}{$H(b_2)$}} at 91 77
\pinlabel $D$ at 67 60
\endlabellist
    \includegraphics[width=60mm]{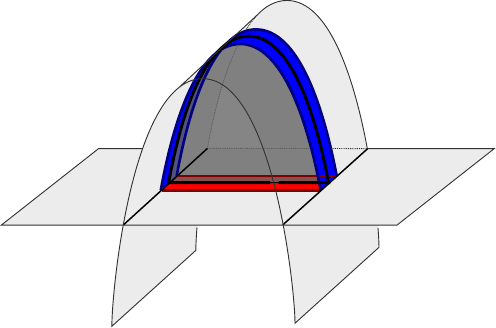}
    \caption{We need to know when bands $b_1$ and $b_2$ can be chosen so that $H(b_1)$ and $H(b_2)$ are as pictured.}
    \label{fig:b1b2}
\end{figure}

Note that when $D$ has the local model of Figure~\ref{fig:uncleanwhitney} or~\ref{fig:b1b2}, the union $H(b_1 \cup b_2)$ is an annulus (as opposed to a M\"obius band) -- this is the first restriction. In addition, let $l_1$ and $l_2$ be 1-dimensional subbundles of the normal bundle of $\partial D$ in $W$ so that along the core of $b_i$, the fiber of $l_i$ is tangent to $H(Y)$ and along the core of $b_{2-i}$, $l_i$ is perpendicular to $H(Y)$. When $D$ is as in Figure~\ref{fig:uncleanwhitney}, both of these line bundles $l_1, l_2$ are trivial -- this is the second restriction.

Further, $H(b_1 \cup b_2)$, $l_1$, and $l_2$ form the normal bundle of $D$ restricted to its boundary $H(\gamma_1) \cup H(\gamma_2)$ which we denote by $N_D|_{\partial D}$.  Any choice of trivialization of $H(b_1\cup b_2)$, $l_1$, and $l_2$ (in that order) yields a trivialization of $N_D|_{\partial D}$. This trivialization must extend to a trivialization of the normal bundle of $D$ -- this is the third restriction.

\begin{definition}\label{def:framedbands}Given bands $b_1$ and $b_2$ thickening $H^{-1}(\partial D)$ for $D$ a disk in $W$, we say that $b_1$ and $b_2$ are {\emph{framed bands for $D$}} when all three of the following are true.
\begin{enumerate}
    \item $H(b_1\cup b_2)$ is an annulus.
    \item The line bundles $l_1, l_2$ are trivial.
    \item A choice of trivialization of each of $H(b_1\cup b_2),l_1,l_2$ as line bundles over $\partial D$ extends to a trivialization of the normal bundle of $D$ in $W$.
\end{enumerate}
By the above discussion, if a Whitney move along $D$ using bands $b_1, b_2$ is possible, then $b_1, b_2$ must be framed bands for $D$. Moreover, so long as $b_1, b_2$ are framed bands for $D$, we may use the obtained trivialization of the normal neighborhood of $D$ to obtain the local model of Figure~\ref{fig:uncleanwhitney} and perform a Whitney move along $D$.
%, we have that $H(b_1)$ and $H(b_2)$ are as in Figure~\ref{fig:b1b2} exactly when $H(b_1 \cup b_2)$, $L_1$, and $L_2$ are trivial, and after choosing a trivialization of all three of them, the resulting trivialization of $N_D|_{\partial D}$ extends to a trivialization of the normal bundle of $D$.  We call such bands $b_1$ and $b_2$ \emph{framed} bands for the disk $D$ and we will say that we have framed bands connecting the respective components of $L$.  
\end{definition}

In Figure~\ref{fig:lattice}, we show the effect of changing either $b_1$ or $b_2$ by adding a half-twist on each of the conditions of Definition~\ref{def:framedbands}.
\begin{itemize}
    \item 
%band twist is shown in Figure~\ref{fig:lattice} which we now justify. 
Adding a half-twist to either $b_1$ or $b_2$ changes the orientability of $H(b_1 \cup b_2)$.% changes with any half-twist in either $b_1$ or $b_2$.
\item Adding a half-twist to $b_i$ changes orientability of $l_i$ while preserving orientability of $l_j$ for $\{i,j\}=\{1,2\}$. 
\item When $H(b_1\cup b_2), l_1, l_2$ are all orientable, adding a full twist to one of $b_1, b_2$ changes whether an induced trivialization of $N_D|_{\partial D}$ extends over all of $D$. We discuss this further in Section~\ref{sec:w2obstruction}.
\end{itemize}% similarly adding a half-In addition, by changing the band $b_1$ by a half-twist, the orientability of $L_1$ changes while the orientability of $L_2$ remains unchanged.  Similarly, if the band $b_2$ is changed by a half-twist then the orientability of $L_2$ changes while the orientability of $L_1$ remains unchanged.  Thus if $H(b_1 \cup b_2)$ is orientable but $L_1$ and $L_2$ are not, then by adding half-twists to both $b_1$ and $b_2$, $L_1$ and $L_2$ become orientable.

The remainder of Section~\ref{sec:whitney} is a study in how to know when the three conditions of Definition~\ref{def:framedbands} hold given bands $b_1,b_2$ and a Whitney disk $D$; i.e., how to determine when $b_1,b_2$ are framed with respect to $D$. A casual reader should feel more than welcome to skip ahead to Section~\ref{sec:cleanwhitney}.

\subsubsection{Choosing the bands to satisfy orientability conditions}\label{sec:whitneyorient}
%\vspace*{.1in}
%\paragraph{{\emph{$H(b_1\cup b_2)$ and $L_1,L_2$}}}

For now, we provide some remarks on how to understand the restrictions that $H(b_1\cup b_2)$ be an annulus and $L_1,L_2$ be orientable from the perspective of working within $Y$.

\begin{remark}\label{rem:checkannulus}
We can check whether $H(b_1\cup b_2)$ is an annulus by orienting the curves in $L$. Orient each active and type II circle; we obtain induced orientations on inactive circles by requiring that dual pairs induce the same orientations on their image.  Then $H(b_1\cup b_2)$ is an annulus if $b_1$ and $b_2$ are either both orientation-preserving or both orientation-breaking as bands (i.e., the orientations on the singular circles can be extended across both $b_1$ and $b_2$, or neither $b_1$ nor $b_2$.) %the boundaries of the bands as in \red{Figure - just some orientation-preserving and orientation-reversing bands, maybe in context so type I circles with an annulus/Moebius band and we similarly show type II circles - maybe overkill, but probably good to just get everyone on board}).
In contrast, $H(b_1\cup b_2)$ is a M\"{o}bius band when one of $b_1,b_2$ is orientation-preserving while the other is orientation-breaking. %In the case where one (or both) of the singular circles involved is of type II, the orientation of the singular circle at $x$ also induces an orientation at $x'$ and this entire discussion applies also in that case
\end{remark}

When checking whether $H(b_1\cup b_2)$ is an annulus as in Remark~\ref{rem:checkannulus}, we can choose the orientations of the singular circles containing $x$ and $y$ independently (assuming these are not the same curve or dual curves, in which case one choice determines the other). %Understanding whether $l_1$ and $l_2$ are orientable requires knowledge of the signs of the self-intersections of $H$, which we can use to relate these two orientations.%, which we will use to induce an orientation on the circle containing $y$ after orientating the circle containing $x$.%, which induce preferred orientations on $l$ (rather than random ones).

\begin{remark}\label{rem:l1l2}
In Figure~\ref{fig:l1l2orientablesign}, we illustrate how, given that $H(b_1\cup b_2)$ is an annulus, the question of whether $l_1$ and $l_2$ are orientable is equivalent to whether two intersections have the same sign. More precisely, fix positive bases $(v_1,v_2,v_3)$ and $(v_3,v_4,v_5)$ of the tangent space of $Y$ at a point near an end of $b_1$ and the corresponding end of $b_2$. %(where we have fixed once and for all an orientation of $S^2 \times I$. Further, a
Choose these bases so that the image of $v_3$ near $b_1$ and the image of $v_3$ near $b_2$ agree under $H$ as in Figure~\ref{fig:l1l2orientablesign} (hence the naming convention). Now use $b_1$ and $b_2$ to transport these bases to the opposite ends of the bands, obtaining positive bases $(w_1,w_2,w_3)$ and $(w_3,w_4,w_5)$ respectively. (Note that implicitly we use the fact that $H(b_1 \cup b_2)$ is an annulus for this to make sense -- this property tells us that the two transports of $v_3$ will have the same image $w_3$ under $H$.) In Figure~\ref{fig:l1l2orientablesign}, we see that  $l_1$ and $l_2$ are orientable when the signs of the bases $(v_1,v_2,v_3,v_4,v_5)$ and $(w_1,w_2,w_3,w_4,w_5)$ of for the tangent space of $W$ disagree (left of Figure~\ref{fig:l1l2orientablesign}) and that $l_1$ and $l_2$ are non-orientable when the signs of these bases agree (right of Figure~\ref{fig:l1l2orientablesign}). %Adding a half twist to both of $b_1,b_2$ causes the signs of these bases to agree and We thus conclude that $L_1$ and $L_2$ are both orientable if and only if $(v_1,v_2,v_3,v_4,v_5)$ and $(w_1,w_2,w_3,w_4,w_5)$ have opposite sign. 
Or more simply, following Remark~\ref{rem:checkannulus}: orient the singular circle containing $x$ near $x$; this induces an orientation of the circle containing $x'$ near $x'$. These orientations induce a sign $s\in\{+,-\}$ on the self-intersection of $H(Y)$ near $H(x)=H(x')$. Now we compatibly orient the singular circles near $y$ and $y'$ specifically so that the induced sign on the self-intersection of $H(Y)$ near $H(y)=H(-y)$ is $-s$. Then $l_1$ and $l_2$ are both orientable exactly when $b_1$ and $b_2$ are both orientation preserving with respect to these local orientations.
\end{remark}

Remark~\ref{rem:l1l2} may seem unwieldy, but essentially is an analogue of the fact that in 4-dimensional topology, a framed Whitney disk must run between self-intersection points of a surface that are of opposite sign. This condition is slightly more difficult to state in this dimension, since our immersed 3-manifold $H(Y)$ intersects itself in circles rather than isolated points and hence the self-intersections do not come with inherent signs. Nevertheless, given a choice of sign on one circle in $H(Y)$ and bands connecting that circle to another, we can induce a sign on the other using the bands (which is well defined exactly when the bands union is an annulus rather than a M\"{o}bius band). We then note that the bundles $l_1,l_2$ are orientable exactly when these signs are opposite (which is exactly the case in the 4-dimensional setting, as well).

\begin{figure}
    \centering
       \labellist
       \small\hair 2pt
       \pinlabel {$v_1$} at 140 200
        \pinlabel {$w_1$} at 200 200 
     \pinlabel {$v_2$} at 102 245 
     \pinlabel {$w_2$} at 185 245 
     \pinlabel {$v_3$} at 90 177
     \pinlabel {$w_3$} at 157 177 
     \pinlabel {$v_4$} at 290 200
        \pinlabel {$w_4$} at 350 200 
     \pinlabel {$v_5$} at 255 245 
     \pinlabel {$w_5$} at 338 245 
     \pinlabel {$v_3$} at 240 177
     \pinlabel {$w_3$} at 307 177 
     
      \pinlabel {$v_4$} at 60 34
        \pinlabel {$w_4$} at 110 34
     \pinlabel {$v_1$} at 35 60
     \pinlabel {$w_1$} at 93 18
     \pinlabel {$v_3$} at 27 22
     \pinlabel {$w_3$} at 74 22
     \pinlabel {$v_2$} at 173 39
        \pinlabel {$w_2$} at 205 39
     \pinlabel {$v_5$} at 70 138 
     \pinlabel {$w_5$} at 95 138 
     
         \pinlabel {$v_4$} at 304.5 46.5
        \pinlabel {$w_4$} at 351 34
     \pinlabel {$v_1$} at 276 60
     \pinlabel {$w_1$} at 334 18
     \pinlabel {$v_3$} at 268 22
     \pinlabel {$w_3$} at 355 57
     \pinlabel {$v_2$} at 415 39
        \pinlabel {$w_2$} at 447 39
     \pinlabel {$v_5$} at 307 138 
     \pinlabel {$w_5$} at 332 138 
     
	\endlabellist
    \includegraphics[width=135mm]{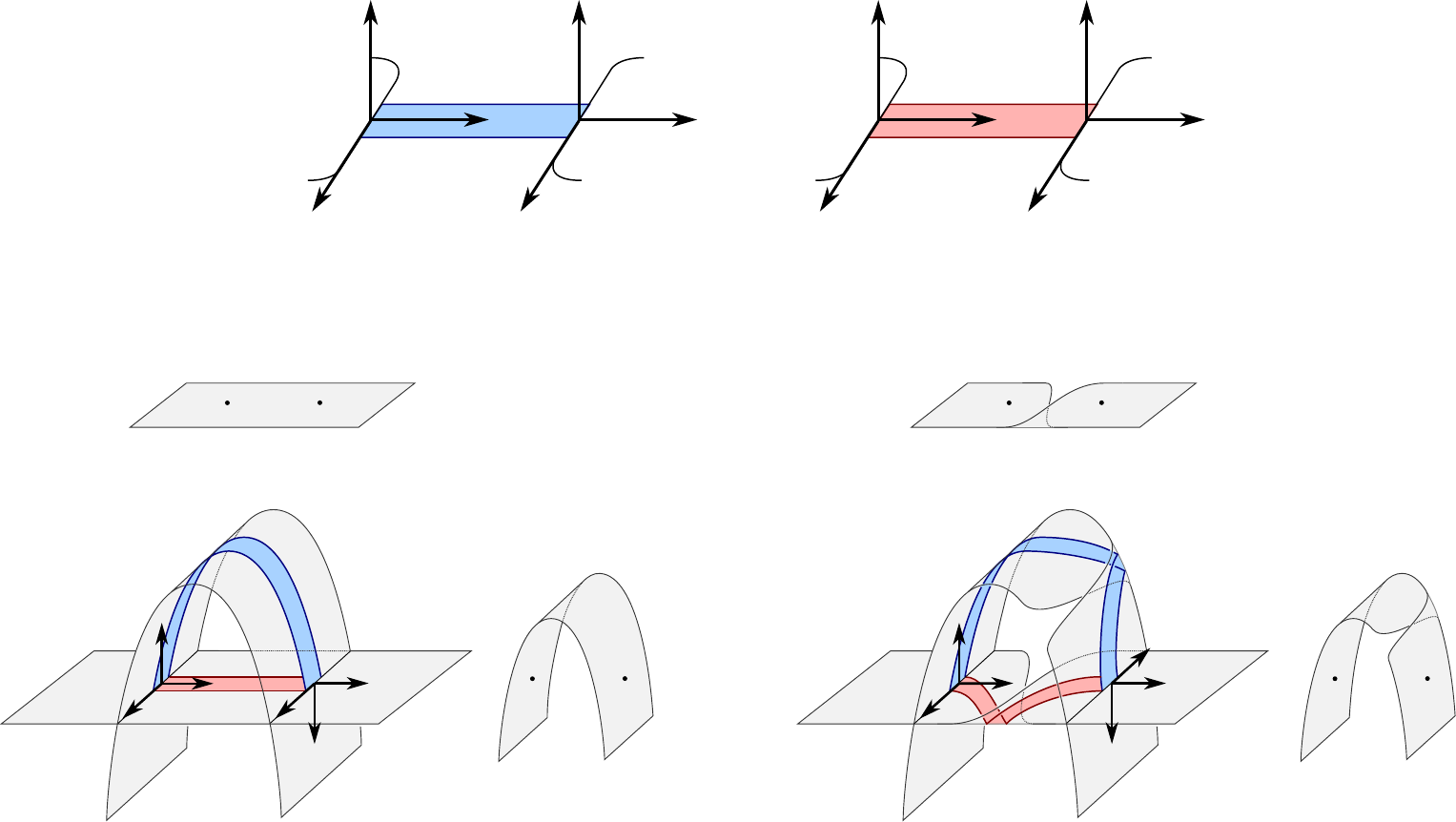}
    \caption{When $H(b_1\cup b_2)$ is an annulus, the bundles $l_1$ and $l_2$ are orientable if and only if the choice of $b_1,b_2$ induce opposite signs on the self-intersections of $H$ near the ends of $b_1,b_2$. {\bf{Top: Schematics of positive bases at the ends of the preimages of $b_1, b_2$ in $Y$.}} {\bf{Bottom left:}} when $l_1$ and $l_2$ are orientable, the induced bases of the points in the tangent space of $W$ at the ends of $H(b_1), H(b_2)$ are of opposite sign. {\bf{Bottom right:}} When $l_1$ and $l_2$ are not orientable, the induced bases at the ends of $H(b_1), H(b_2)$ are of the same sign.}\label{fig:l1l2orientablesign}
\end{figure}
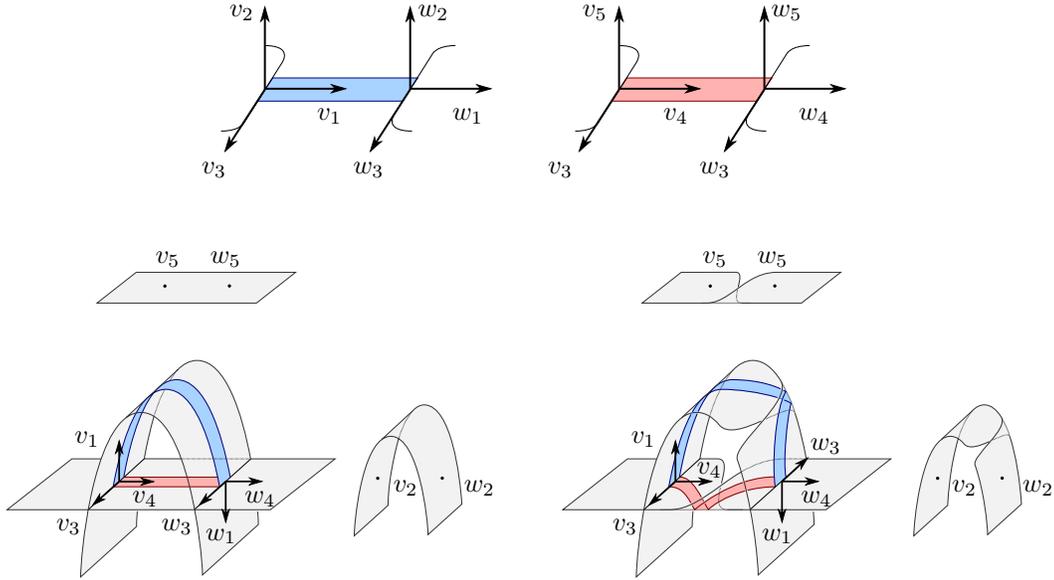

\vspace*{.1in}
\subsubsection{The $\omega_2$ obstruction}\label{sec:w2obstruction}

Now we study the last condition of Definition~\ref{def:framedbands}, which is slightly different from the first two in that it is not a condition on orientability of a bundle. This is the condition that a certain trivialization of $N_D|_{\partial D}$ extends over $D$, and is analogous to the usual condition in dimension 4 of a framing on the boundary of a Whitney disk extending over the interior of the disk.

Assume that $H(b_1 \cup b_2)$ is an annulus and that $l_1$ and $l_2$ are orientable, as in Section~\ref{sec:whitneyorient}. Trivialize all three bundles to obtain a trivialization of $N_D|_{\partial D}=H(b_1\cup b_2)\oplus l_1\oplus l_2$. This trivialization gives an element of $\pi_1(\SO(3))=\mathbb{Z}/2\mathbb{Z}$ which is trivial if the trivialization extends over all of $D$ and nontrivial otherwise. %If this trivialization does not extend over $D$ (namely the resulting element of $\pi_1(\SO(3))$ is nontrivial) then by
If we alter $b_1$ (or similarly $b_2$) by adding a full twist, the effect on the induced element of $\pi_1(SO(3))$ is to add a full twist, changing whether the trivialization of $N_D|_{\partial D}$ extends over $D$. %$l_1$ and the trivialization of $H(b_1 \cup b_2)$ each change while $L_2$ is fixed, in such a way that the resulting new trivialization of $N_D|_{\partial D}$ corresponds to the trivial element of $\pi_1(\SO(3))$ and does extend over $D$.

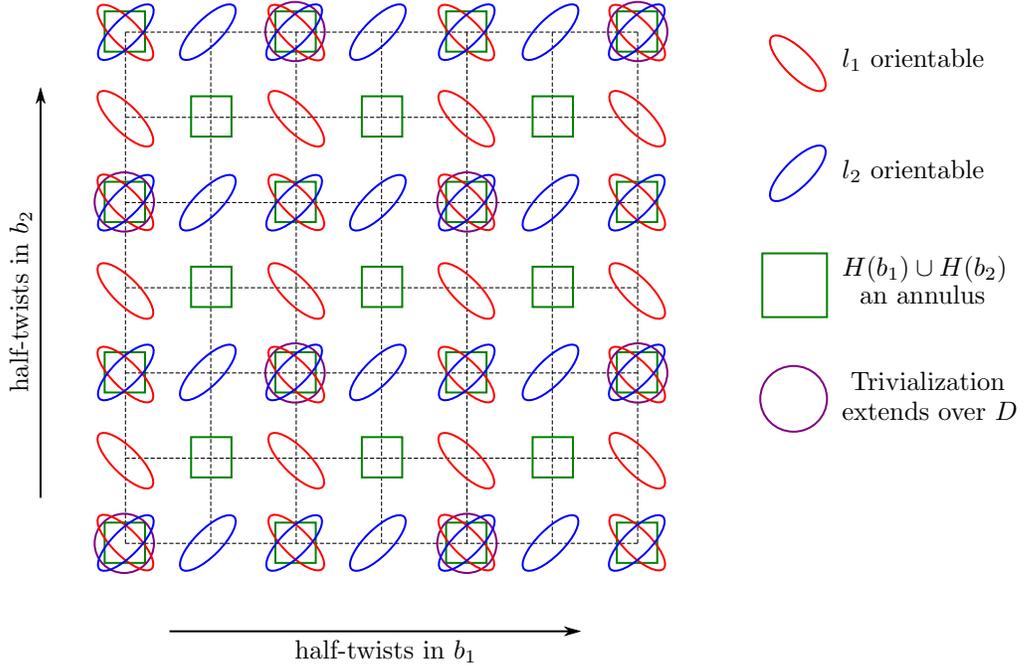
\begin{figure}
      \labellist
      \small\hair 2pt
       \pinlabel {half-twists in $b_1$} at 120 -5
        \pinlabel \rotatebox{90}{half-twists in $b_2$} at -5 115
               \pinlabel {$l_1$ orientable} [tl] at 275 203
               \pinlabel {$l_2$ orientable} [tl] at 275 165
               \pinlabel {$H(b_1)\cup H(b_2)$} [tl] at 275 132
               \pinlabel { \hspace{0in} an annulus} [tl] at 275 122
               \pinlabel { Trivialization} [tl] at 275 92
               \pinlabel {extends over $D$} [tl] at 275 82
	\endlabellist
    \centering
    \includegraphics[width=5.5in]{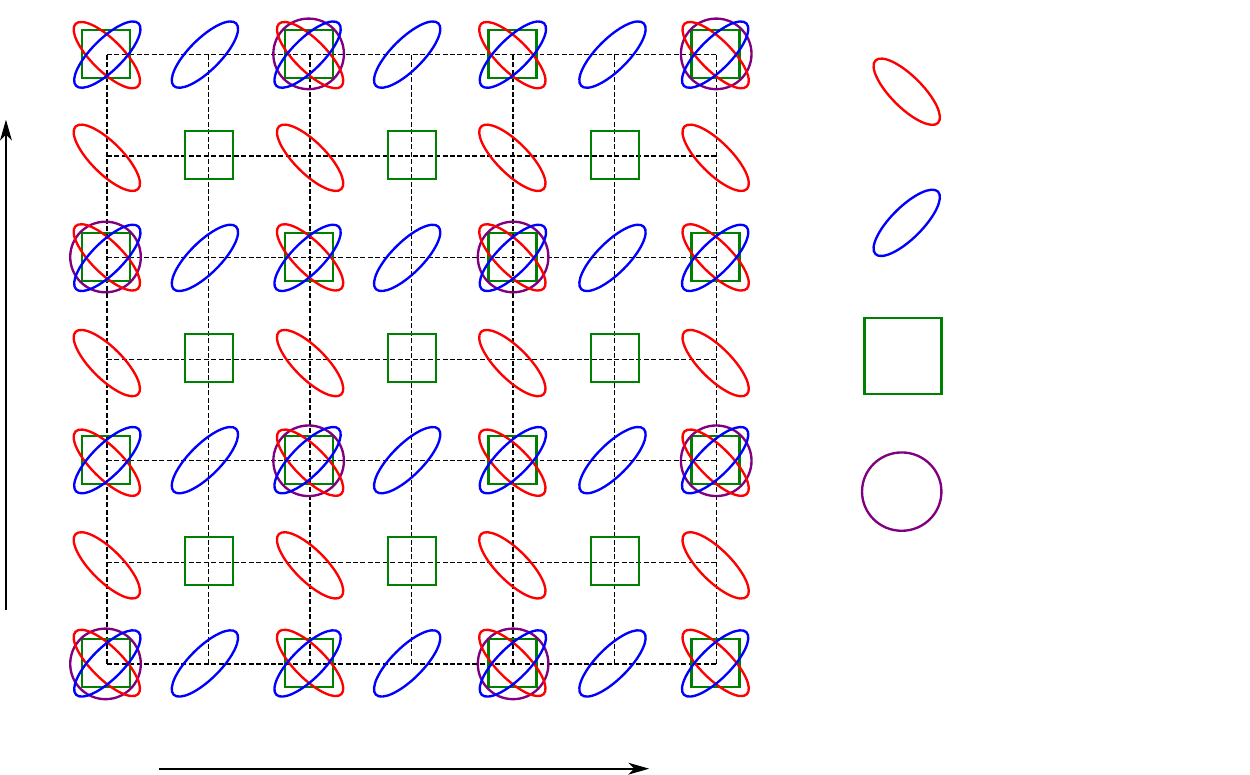}
    \vspace{.1in}
    \caption{We consider how adding half twists to $b_1$ and $b_2$ affects the orientability of the line bundles $H(b_1\cup b_2)$, $l_1$, and $l_2$. Note that we describe this as an integral lattice in $\mathbb{E}^2$ rather than $\mathbb{R}^2$, as there is no canonical zero.  Whenever $h(b_1\cup b_2)$, $l_1$, and $l_2$ are all orientable, then we may consider whether an induced trivialization ${N_D}|_{\partial D}$ extends over all of $D$. The points corresponding to trivializations that extend define a sublattice whose fundamental domain is a parallelogram with side lengths $2$ and $4$.}
    \label{fig:lattice}
\end{figure}

\begin{remark}\label{rmk:same_homology_class}
Assume we are in the case where the bands $b_1$ and $b_2$ are such that $H(b_1 \cup b_2)$, $l_1$, and $l_2$ are all orientable, but an induced trivialization of $N_D|_{\partial D}$ does not extend over the disk $D$. Then the bands $b_1,b_2$ are not framed bands for $D$. If there is a 2-sphere $S$ embedded in $W$ with $w_2(S) = 1$, %that is disjoint from $D$,
then we may obtain a new disk $D'$ by tubing $D$ to $S$. The same trivialization of $N|_{\partial D'}$ does extend over $D'$, so we conclude that $b_1$ and $b_2$ are framed bands for $D'$. This is why we write the disk with respect to which a pair of bands are framed, rather than just saying ``$b_1$ and $b_2$ are framed," without a specification of disk.
\end{remark}

\subsubsection{Final remarks on choosing bands}\label{sec:whitneyfinalremarks}

Recall from Section~\ref{sec:finddisks} that we may only perform Whitney moves along disks meeting intersection circles corresponding to the same or inverse group element. 

\begin{remark}
Performing a Whitney move preserves the group element associated to each singular circle. That is, if a singular circle $A$ with corresponding group element $a$ is involved in a Whitney move, then the resulting singular link has one or two components that include arcs originally in $A$. These components have group element $a$ or $a^{-1}$.

To see this, choose arcs $\eta$ from each singular circle in $L$ to the basepoint of $Y$. Then the (appropriately oriented) images of $\eta$ arcs under $H$ can be concatenated to form loops determining the group elements for each self-intersection of $H(Y)$. The images of endpoints of the arcs in $\eta$ can be taken to be disjoint from the Whitney disk. Moreover, the Whitney disk meets $H(Y)$ in arcs, so generically the boundary of the Whitney disk is disjoint from all of $\eta$. Thus, the same loops determine the group elements of singular circles for $H$ after performing the Whitney move.
\end{remark}
%\begin{definition}
%An immersed 3-manifold $M$ in a 5-manifold is called \emph{$s$-characteristic} (\emph{spherically %characteristic}) if for every 2-sphere $S$ in $W$ we have $w_2(S) = [M] \cdot [S] \pmod{2}$.  
%\end{definition}

%The primary example we will be interested in in this paper is the case where $M = H(S^2 \times I)$ where $H : S^2 \times I \to X \times I$  is a singular concordance between homotopic embedded spheres $S_0$ and $S_1$ in a 4-manifold $X$ -- see Observation~\ref{lem:4to5}.

\begin{conclusion}
When there are (not necessarily distinct) singular circles in $L$ with the same or inverse group elements $a$ and $b$, we may choose arcs $\gamma_1,\gamma_2$ between them and use these arcs to perform a Whitney move (along a disk bounded by the images of the arcs).
\begin{itemize}
    \item If the singular circles have inverse group element and are type I, then the arcs $\gamma_1,\gamma_2$ should connect the active circle in one pair to the inactive circle in the other pair. (This includes the case when we consider one type I pair $A,A'$ and let $\gamma_1$ and $\gamma_2$ both have ends at $A$ and $A'$.)
    \item If the singular circles have equal group elements and are type I, then the arcs should connect the active circle in one pair to the active circle in the other, and the inactive circle in one pair to the inactive circle in the other. (This includes the case when we consider one type I pair $A,A'$ and let $\gamma_1$ have both ends at $A$ and $\gamma_2$ have both ends at $A'$.)
    \item Once the arcs and Whitney disk are specified, the arcs cannot be arbitrarily thickened into bands. Instead, the twisting of the two bands is determined up to a criterion in $\mathbb{Z}/2\mathbb{Z} \oplus \mathbb{Z}/4\mathbb{Z}$. If $H(Y)$ is $s$-characteristic, then given two Whitney disks $D_1, D_2$ with the same boundary and with \[|D_1\cap H(Y)|\equiv |D_2\cap H(Y)|\pmod{2},\] these obstructions agree: bands $b_1,b_2$ are framed with respect to $D_1$ if and only if they are framed with respect to $D_2$. %{\blue{If $H(Y)$ is not $s$-characteristic, then }
    %\item Generally, we cannot hope to find a Whitney disk whose interior is disjoint from the image of $H(Y)$.
    %If there is a 2-sphere in $X\times I$ with nontrivial normal bundle that is disjoint from the image of $H$, then this obstruction is instead contained in $\mathbb{Z}/2\oplus\mathbb{Z}/2$. However, for some framing of $\gamma_1,\gamma_2$ to obtain bands, a Whitney move along the disk is possible. 
    \item A Whitney move along a disk $D$ bounded by $H(\gamma_1)\cup H(\gamma_2)$ changes $L$ by band surgery along bands whose cores are the framed arcs $\gamma_1,\gamma_2$, and also adds a pair of dual type I circles for each intersection of $D$ with $H(Y)$. The group elements of the singular circles resulting from the band surgery are still $a$ or $a^{-1}$. %If $a=b$, then the group elements associated to the resulting singular circles are still $a$. If $a=b^{-1}$, then we must choose which of the resulting circles are active or inactive; depending on this choice all relevant associated group elements are $a$ or $a^{-1}$.
\end{itemize}
\end{conclusion}

\setcounter{secnumdepth}{3}

\subsection{Clean Whitney moves}\label{sec:cleanwhitney}
In Section~\ref{sec:whitney}, we saw that we cannot expect the interior of a Whitney disk to be disjoint from $H(Y)$ and understood how these intersections informed the effect on the singular link $L$ of performing the Whitney move.% on the singular link $L$. %In this section, we will show that when $H(Y)$ admits a dual sphere (framed or unframed), then we may obtain a Whitney disk with interior disjoint from 

Let $\gamma_1,\gamma_2$ be arcs between singular circles in $Y$ so that $H(\gamma_1)\cup H(\gamma_2)$ bounds a disk $D$. % bounding a Whitney disk.

\begin{proposition}\label{prop:getridofwhitneydiskints}
Assume that there is a 2-sphere $G$ embedded in $W$ that intersects $H(Y)$ transversely in a single point. Then there is a Whitney disk $D'$ bounded by $H(\gamma_1)\cup H(\gamma_2)$ so that $D'$ does not intersect $H(Y)$.
\end{proposition}

\begin{proof}
The disk $D$ intersects $H(Y)$ in $n$ points for some $n$. Obtain $D'$ from $D$ by tubing $D$ to $n$ copies of $G$ near these intersections.
\end{proof}

\begin{definition}
We call a Whitney disk bounding the arcs $\gamma_1$ and $\gamma_2$ that is disjoint from $H(Y)$ \emph{clean}.  
\end{definition}

\begin{remark}\label{rmk:disjoint disk}
In Proposition~\ref{prop:getridofwhitneydiskints}, if $w_2(G) = 0$,  then a trivialization of $N_D|_{\partial D}$ that extends over $D$ also extends over $D'$. If $w_2(G) = 1$, then this depends on parity of $n$: If $n$ is even then the trivialization extends over $D'$ but if $n$ is odd then the trivialization does not extend over $D'$.
\end{remark}

%Until now, we have had to specify a Whitney disk in order to discuss whether a framing of $\gamma_1,\gamma_2$ yields framed bands.
When $Y$ has a dual sphere, we can say exactly when a framing of $\gamma_1, \gamma_2$ extends over a clean Whitney disk, as we explain in the following proposition. %This criterion is independent of the choice of clean Whitney disk.
%In Section~\ref{sec:whitneyfinalremarks}, we saw that the $\Z/2\oplus \Z/4$ criterion for framings of $\gamma_1,\gamma_2$ to induce framed bands is well-defined for Whitney disks intersecting $H(Y)$ in a fixed parity. %In particular, when 

%, this is not well-defined across any Whitney disk bounded by $H(\gamma_1)\cup H(\gamma_2)$.  This obstruction is not actually valued in the group $\Z/2\oplus \Z/4$, but we will still use the suggestive terminology.  However, in the situation that $H(Y)$ has a dual sphere, we can always give a criterion for a framing of $\gamma_1$ and $\gamma_2$ to extend over some clean Whitney disk.

\begin{proposition}\label{prop:criterionextend}\leavevmode
\begin{enumerate}
    \item\label{pointone} If $H(Y)$ has a dual sphere $G$ with $w_2(G) = 0$, then there is a $\mathbb{Z}/2\mathbb{Z}\oplus\mathbb{Z}/2\mathbb{Z}$ criterion to a framing of $\gamma_1$ and $\gamma_2$ extending to some clean Whitney disk.
    \item\label{pointtwo} If $H(Y)$ has a dual sphere $G$ with $w_2(G) = 1$ and if $H(Y)$ is {\emph{not}} $s$-characteristic, then there is a $\mathbb{Z}/2\mathbb{Z}\oplus\mathbb{Z}/2\mathbb{Z}$ criterion to a framing of $\gamma_1$ and $\gamma_2$ extending to some clean Whitney disk.
\item\label{pointthree} If $H(Y)$ has a dual sphere $G$ with $w_2(G) = 1$ and is $H(Y)$ is $s$-characteristic, then there is a $\mathbb{Z}/2\mathbb{Z}\oplus\mathbb{Z}/4\mathbb{Z}$ criterion to a framing of $\gamma_1$ and $\gamma_2$ extending to some clean Whitney disk (and this same criterion holds for every clean Whitney disk).
\end{enumerate}
\end{proposition}

\begin{proof}
%Whether framings of $\gamma_1,\gamma_2$ cause $L_1,L_2,$ and $H(b_1)\cup H(b_2)$ to be orientable is independent of the choice of disk. Arranging for these to all be orientable is a $\Z/2\oplus\Z/2$ condition; see Figure~\ref{fig:lattice}.

In all of the stated cases, since $H(Y)$ has a dual sphere we know by Proposition~\ref{prop:getridofwhitneydiskints} that there is a clean Whitney disk $D$ with boundary $H(\gamma_1)\cup H(\gamma_2)$. Let $b_1,b_2$ be band thickenings of $\gamma_1,\gamma_2$ that are framed with respect to $D$. For any other choice of Whitney disk, whether $b_1,b_2$ are framed with respect to $D$ is determined by the $\omega_2$ obstruction of Section~\ref{sec:w2obstruction}.

If $\omega_2(G)=0$, then we may twist $D$ once near its boundary to add a full twist to either $b_1$ or $b_2$. This introduces an intersection of $D$ with $H(Y)$ which we can remove by tubing $D$ to $G$. This preserves whether the induced trivialization of $N_D|_{\partial D}$ extends over $D$, settling Case \eqref{pointone}.

If $\omega(G)=1$ and $H(Y)$ is not $s$-characteristic, then let $F$ be an embedded 2-sphere in $X\times I$ with $F\cdot H(Y)\not\equiv F\cdot F\pmod{2}$. Tube $D$ to $F$ to obtain a different Whitney disk $D'$, and then tube $D'$ to a copy of $G$ for each intersection of $D'$ with $H(Y)$ to obtain a clean Whitney disk $D''$. In all, we have obtained $D''$ from $D$ by tubing on an odd number of 2-spheres with $\omega_2=1$ (and possibly one $\omega_2=0$ 2-sphere). We conclude that the bands $b_1$ and $b_2$ are  {\emph{not}} framed for $D''$, thus settling Case \eqref{pointtwo}.

Finally, if $G$ is unframed and $H(Y)$ is $s$-characteristic, then let $D'$ be any other clean Whitney disk bounded by $H(\gamma_1)\cup H(\gamma_2)$. The 2-sphere $D\cup D'$ can be isotoped off $H(Y)$. Because $H(Y)$ is $s$-characteristic, then the 2-sphere $D\cup D'$ must have trivial normal bundle. Therefore, a framing of $N_D|_{\partial D}$ extends over $D$ if and only if it extends over $D'$, settling Case \eqref{pointthree}.
\end{proof}

We can use the following lemma to understand when bands $b_1,b_2$ are framed with respect to some clean Whitney disk.%bound a disk satisfying the $\omega_2$ condition.

\begin{lemma} \label{lem:stong1}
%Let $S_0, S_1$ be embedded 2-spheres in a 4-manifold $X^4$ and let $H : S^2 \times I \to X^4 \times I$ be an immersion with $H(S^2 \times \{0\}) = S_0$ and $H(S^2 \times \{1\}) = S_1$.
Assume that $H$ is $s$-characteristic in $X$ with a dual sphere $G$ and let $A$ and $A'$ be a dual pair of type I circles in $L$. Assume both $A$ and $A'$ are null-homotopic in $Y$. Then $$\lk(A, L - A) = \lk(A', L - A') \pmod{2}.$$%{\blue{MM: note: do we ever need type II case?}}
\end{lemma}

We delay the proof of Lemma~\ref{lem:stong1} to the end of Section~\ref{sec:ambientdehn}. We make note of it here to emphasize that this condition makes it quite easy to check whether bands $b_1, b_2$ are framed for some clean Whitney disk when $H$ is $s$-characteristic with a dual. First of all, check that $b_1,b_2$ respect the orientation conditions of Section~\ref{sec:whitneyorient} (this is particularly easy in the case that each of $b_1, b_2$ meet only one circle, in which case we must only check that both bands are orientation-preserving). Then, check whether performing band surgery along $b_1, b_2$ would yield a singular link that satisfies Lemma~\ref{lem:stong1}. If not, add a full twist to one of $b_1$ or $b_2$ -- this preserves the orientation properties, but now surgering along $b_1$ and $b_2$ yields a singular link that does satisfy Lemma~\ref{lem:stong1}. The bands $b_1, b_2$ are now framed for some clean Whitney disk by Proposition~\ref{prop:criterionextend}.

%\begin{remark}
%From now on, given bands $b_1$ and $b_2$, we will say that ``$b_1$ and $b_2$ are framed" (without speicification of a disk) to mean that $b_1$ and $b_2$ are framed with respect to some 
%\end{remark}

\subsection{Gluing Whitney disks}\label{sec:gluingwhitneydisks}

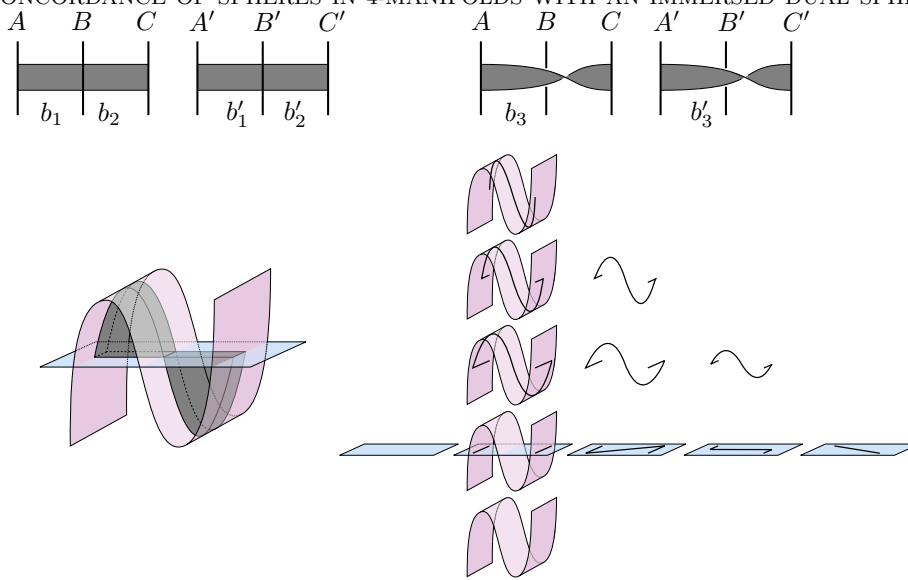
\begin{figure}
     \centering
       \labellist
       \small\hair 2pt
       \pinlabel {$A$} at 0 300
        \pinlabel {$B$} at 35 300
     \pinlabel {$C$} at 70 300
     \pinlabel {$A'$} at 100 300
     \pinlabel {$B'$} at 135 300
     \pinlabel {$C'$} at 170 300
     \pinlabel {$b_1$} at 19 250
        \pinlabel {$b_1'$} at 119 250
     \pinlabel {$b_2$} at 50 250
     \pinlabel {$b_2'$} at 150 250
       \pinlabel {$A$} at 250 300
        \pinlabel {$B$} at 285 300
     \pinlabel {$C$} at 320 300
     \pinlabel {$A'$} at 350 300
     \pinlabel {$B'$} at 385 300
     \pinlabel {$C'$} at 420 300
     \pinlabel {$b_3$} at 269 250
     \pinlabel {$b_3'$} at 369 250
	\endlabellist
	
    \includegraphics[width=120mm]{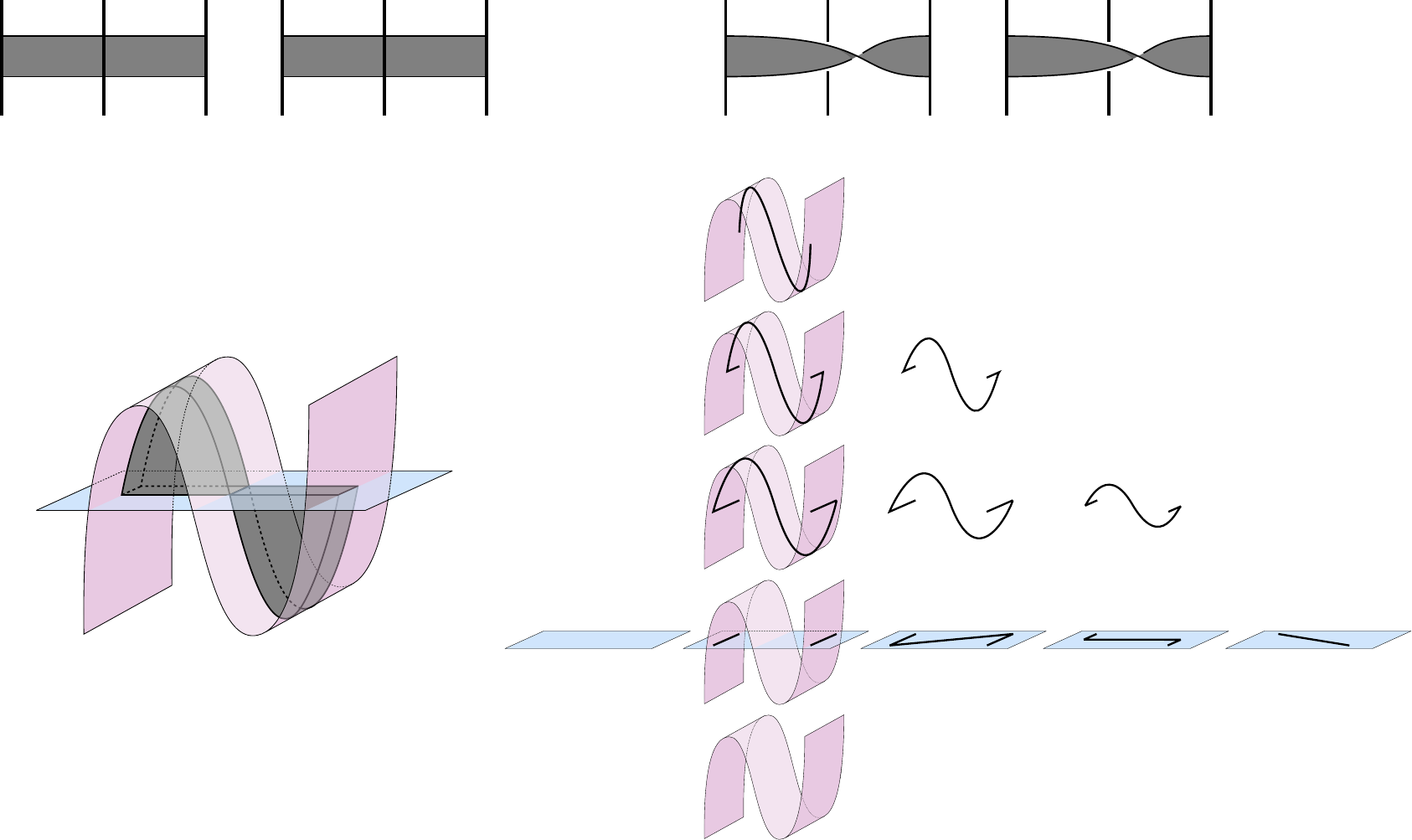}
    \caption{Left: Whitney disks $D_1$ and $D_2$ meeting at $H(B)$. %Middle: We glue $b_1$ and $b_2$ and add a half-twist to form a band $b_3$.
    Right: We push $b_1\cup b_2$ and $b'_1\cup b'_2$ off of $B$ and $B'$ (respectively) and add a half twist to form $b_3$ and $b'_3$. The sign of the half-twist in $b'_3$ is determined by the sign of the half-twist in $b_3$ (and the directions we pushed off $B$ and $B'$), given that we want to find a disk $D_3$ agreeing with $D_1$ and $D_2$ away from $B$ that yields the bands $b_3$ and $b'_3$.} %{\blue{MM: Need to draw this with the vectors added in so we can see the sign information}}}
    \label{fig:gluewhitneydisks}
\end{figure}

In this subsection, we describe a procedure for gluing two Whitney disks which share a common corner to obtain a new Whitney disk. %We examine how bands corresponding to the original Whitney disks yield bands for the new one.

Let $A,B,C\subset L$ be (not necessarily distinct) singular circles in $Y$. If $A$ is type I then set $A'$ to be is its dual; if $A$ is type II then set $A'$ to be $A$. Similarly choose singular circles $B'$ and $C'$.

Let $D_1$ be a Whitney disk meeting $H(A)$ and $H(B)$ and let $D_2$ be a Whitney disk meeting $H(B)$ and $H(C)$. Assume that $D_1$ and $D_2$ meet $H(B)$ in the same point, but that otherwise their boundaries are disjoint (see left of Figure~\ref{fig:gluewhitneydisks}). Then we may glue $D_1$ and $D_2$ together to obtain a new Whitney disk $D_3$ meeting $H(A)$ and $H(C)$ as follows.

Let $\eta_1$ and $\eta'_1$ be the cores of the bands $b_1$ and $b'_1$ in $Y$ whose image bounds the framed disk $D_1$. Let $\eta_2$ and $\eta'_2$ be the cores of the bands $b_2$ and $b'_2$ in $Y$ whose image bounds the framed disk $D_2$. 
Isotope $D_2$ near $H(B)$ so $H^{-1}(\partial D_1\cup\partial D_2)$ form two arcs $\eta_3=\eta_1\cup\eta_2$ and $\eta_3'=\eta'_1\cup\eta'_2$. In Figure~\ref{fig:gluewhitneydisks} (left) we give a prototypical schematic.

Note that if $A,B,C$ are oriented so that $b_1$ and $b_2$ are both orientation-preserving, then $b_3=b_1\cup b_2$ viewed as a band connecting $A$ and $C$ is {\emph{not}} orientation-preserving. We push $b_3$ off $B$ and add a half-twist so that $b_3$ is orientation-preserving. See the right of~\ref{fig:gluewhitneydisks}. %In order to push $b_3$ off $B$, we choose a basis $\langle v_1,v_2\rangle$ of $N_{Y}(B)$ at a point in $\partial b_1\cap\partial b_2$ with $v_2$ pointing along $b_2$ toward $C$. We push $b_3$ in the direction $v_1$.

%Let $\langle v_1',v_2'\rangle$ be a basis of $N_{Y}(B')$ at a point $p'$ with $H(p)=H(p')$ with $v_2'$ pointing along $b_2'$ toward $C$, and with $v_1'$ chosen so that the bases $\langle v_1,v_2\rangle$ and $\langle v_1',v_2'\rangle$ have the same sign.

Now push $b'_3= b'_1\cup b'_2$ slightly off $B$ as well. %in the direction $v_1'$.
If we add a half-twist to $b'_3$, then $H(b_3\cup b'_3)$ will be an annulus and both $b_3$ and $b'_3$ will be orientation-preserving. Let $D_3$ be a (possibly unframed) Whitney disk for $b_3, b_3'$ obtained by plumbing together $D_1, D_2$ as in the right of Figure~\ref{fig:gluewhitneydisks}.

By the analysis of Section~\ref{sec:whitney}, $b_3$ and $b'_3$ are framed for $D_3$ up to adding a whole twist to $b'_3$. That is, for one choice of the sign of half-twist added to $b'_3$, the bands $b_3$ and $b'_3$ will be framed. (We can state this sign explicitly in terms of certain bases of points in the tangent space of $H(Y)$ and $N_{H(Y)}$, but we do not think this would be very instructive.) %In Figure~\ref{fig:gluewhitneydisks} (right), we illustrate that when the signs of the half-twists added to $b_3$ and $b_3'$ agree, then we can construct the Whitney disk $D_3$. (This picture illustrates the only mysterious part of the construction, i.e. the part of $D_3$ close to $B$, in detail. In general, $D_3$ agrees with $D_1$ or with $D_2$ away from a neighborhood of $B$.)

\subsection{Introducing a type II singular circle}\label{sec:maketrivialtypeII}

\begin{figure}
{\centering
    \labellist
    \small\hair 2pt
     \pinlabel {$D$} at 20 234 
     \pinlabel {$D'$} at 185 234
    \pinlabel {\textcolor{violet}{$W$}} at 202 180
		\pinlabel {$B^4\times 0$} at 20 -10
		\pinlabel {$B^4\times \frac{1}{6}$} at 75 -10
		\pinlabel {$B^4\times \frac{1}{3}$} at 130 -10
		\pinlabel {$B^4\times \frac{1}{2}$} at 185 -10
	\endlabellist
	\vspace{.2in}
	\includegraphics[width=80mm]{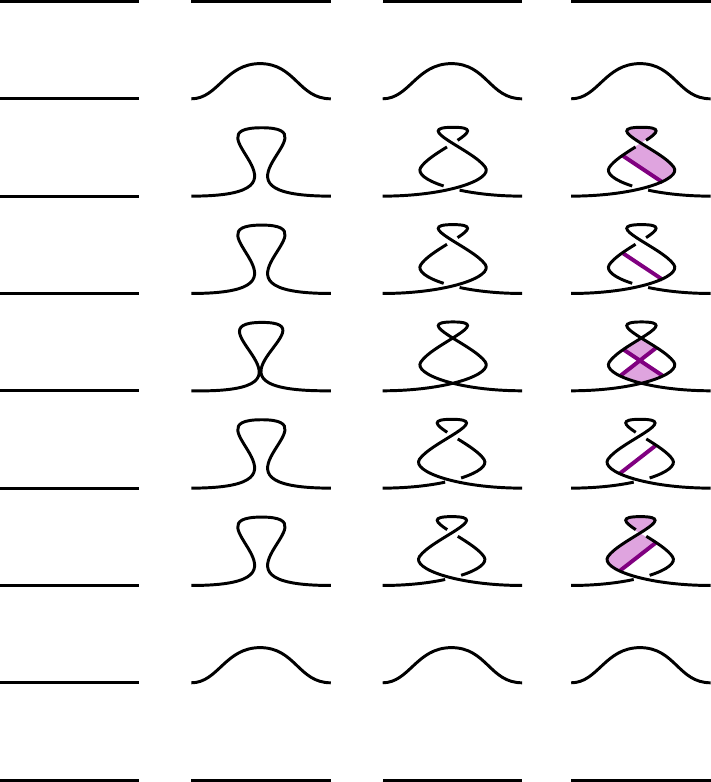}
	\vspace{.3in}
	\caption{An illustration of $B\cong D^2\times I$ immersed in $B^4\times I$; here we picture only $B^4\times[0,1/2]$. %Here, we draw in black the portion of $B$ contained in $B^4\times[0,1/2]$.
	The intersection of $B$ with $B^4\times 0$ is an unknotted disk $D$. In $B^4\times t$ from $t=0$ to $t=1/2$, $B$ consists of a trace of a trivial finger move from $D$, so that $B\cap(B^4\times 1/2)$ is a disk $D'$ with two points of self-intersection in its interior. In $B^4\times 1/2$, we draw a Whitney disk $W$ for $D'$. In $B^4\times t$ from $t=1/2$ to $t=1$ (pictured in Figure~\ref{fig:gettypeiiboundary}), $B$ consists of a trace of a Whitney move from $D'$ along $W$. In Figure~\ref{fig:gettypeiiboundary}, we see that $B\cap(B^4\times 1)$ is again an unknotted disk (and along the way demonstrate that $W$ is framed).}
	\label{fig:gettypeiicircle}
}
\end{figure}
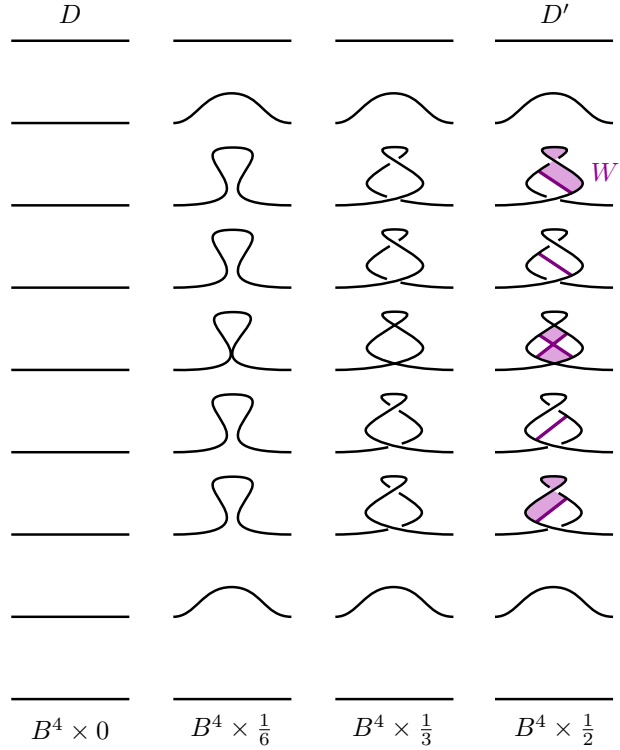

In the definition of the Freedman--Quinn invariant, one does not count the singular circles of an immersed $Y\into X\times I$ that correspond to the trivial group element. In Figure~\ref{fig:gettypeiicircle}, we produce $B\cong D^2\times I$ immersed in $B^4\times I$ whose singular link consists of one type II singular circle. Moreover, the boundary of $B$ is $(D\times 0) \cup (\boundary D\times[0,1])\cup (D\times 1)$, for $D$ the unknotted disk in $B^4$ (see Figure~\ref{fig:gettypeiiboundary}). Using $B$, we may locally modify $H$: consider an immersion $H':Y\to X\times I$ whose image is obtained from that of $H$ by deleting a small standard $D^2\times I \subset B^4\times I$ and replacing it with a copy of $B$. Then $H'$ has a singular link obtained from that of $H$ with an additional type II circle corresponding to the trivial group element. 

\begin{figure}
{\centering
    \labellist
    \small\hair 2pt
    \pinlabel {$D'$} at 0 330 
    \pinlabel {\textcolor{violet}{$W$}} at 80 280
		\pinlabel {(a)} at 47 185
			\pinlabel {(b)} at 164 185
		\pinlabel {(c)} at 280 185
		\pinlabel {(d)} at 392 185
			\pinlabel {(e)} at 47 -5
			\pinlabel {(f)} at 164 -5
		\pinlabel {(g)} at 280 -5
		\pinlabel {(h)} at 396 -5
		\pinlabel \rotatebox{90}{$B\cap(B^4\times \frac{1}{2})$} at 5 255
    \pinlabel \rotatebox{90}{$B\cap(B^4\times 1)$} at 360 70 
	\endlabellist
	\includegraphics[width=90mm]{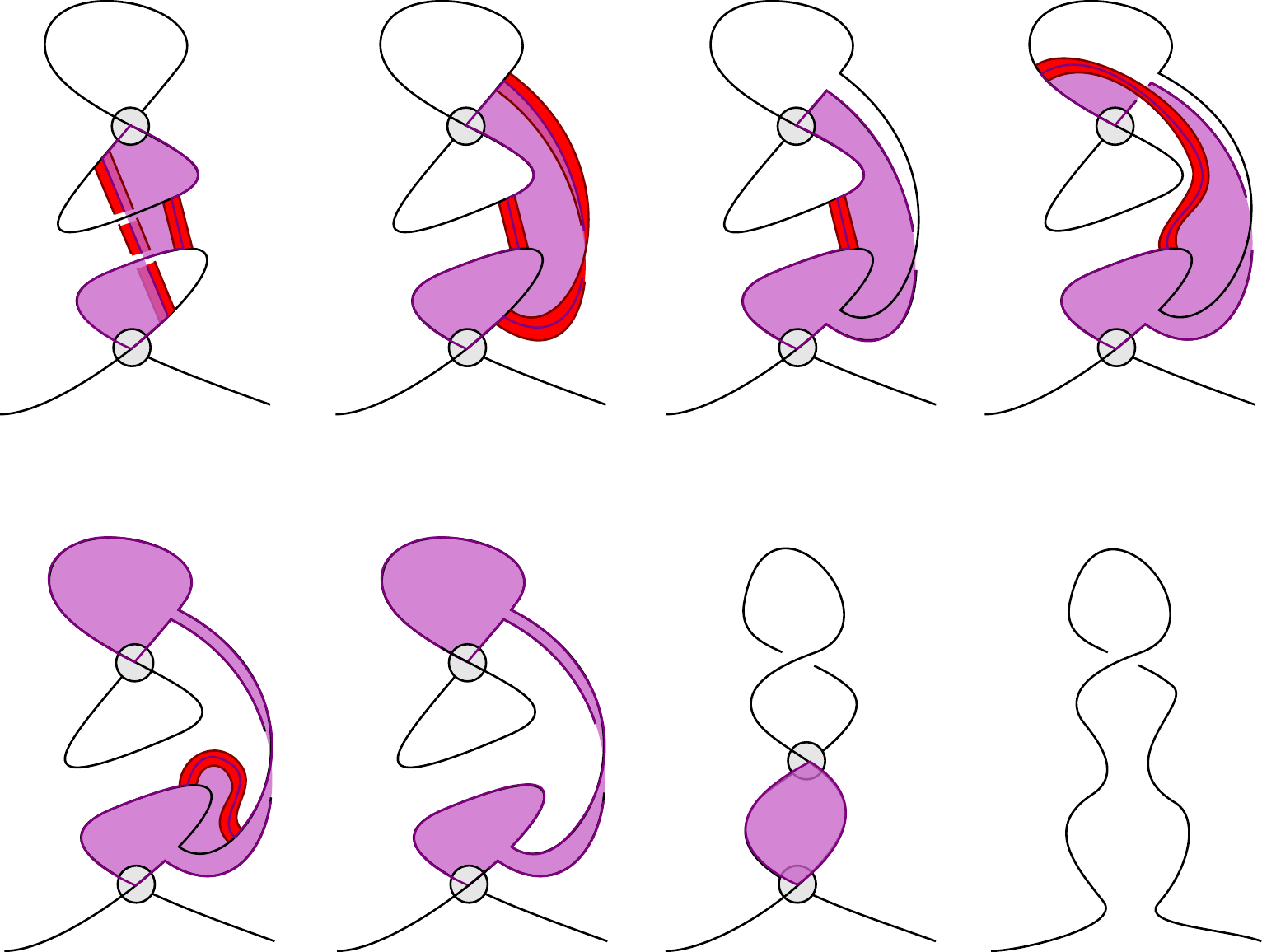}
	\vspace{.1in}
	\caption{In (a), we draw the Whitney disk $W$ with boundary on disk $D'\subset (B^4\times \frac{1}{2})$ as in Figure~\ref{fig:gettypeiicircle}. We draw $D'$ as an immersed tangle with attached bands indicating index-1 critical point. The two self-intersections of $D'$ are pushed into the tangle, so that it is singular. We indicate these with grey/shaded disks. See~\cite{hkmimmersed} for more on these diagrams of self-transverse immersed surfaces in 4-manifolds and the moves on these diagrams corresponding to isotopy or homotopy. From (a) to (g), we isotope $D'$ and $W$ to make $W$ look standard. Then from (g) to (h), we perform the Whitney move along $W$ and obtain an unknotted disk, as claimed in Figure~\ref{fig:gettypeiicircle}.}
	\label{fig:gettypeiiboundary}
}

\end{figure}

For the convenience of the reader, the isotopy that simiplifies the Whitney disk $W'$ with boundary on immersed disk $D'$ in Figure~\ref{fig:gettypeiiboundary} is broken into several steps:

\begin{itemize}
\item{(a) to (b):} We slide one index-1 point of $D'$ above a self-intersection (intersection/band pass of~\cite{hkmimmersed}).
\item{(b) to (c):} We cancel an index-0, index-1 pair of $D'$.
\item{(c) to (d):} We slide another index-1 point of $D'$ above a self-intersection.
\item{(d) to (e):} We further move the index-1 point.
\item{(e) to (f):} We cancel an index-1, index-2 pair of $D'$.
\item{(f) to (g):} We isotope the diagram to turn $W$ into the standard diagram of a Whitney disk.
 \end{itemize}

\subsection{Clasp self-intersections and finger moves}\label{sec:claspmove}

In this subsection, we reprove the following classical 3-dimensional lemma which will be applied to our singular link $L$.

\begin{lemma} \label{lem:clasp}
Let $J$ be a link of null-homotopic components in a 3-manifold $Y^3$.  Then $J$ bounds a set of immersed disks in $Y$ with only clasp intersections amongst the disks as in Figure~\ref{fig:clasp}.
\end{lemma}

\begin{figure}
    \centering
    \includegraphics[width=30mm]{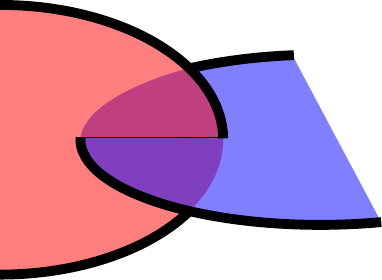}
    \caption{A clasp intersection consists on an arc, neither of whose preimages are properly embedded. (That is, each local sheet has boundary at one end of the arc.) Note that a clasp intersection may occur between two distinct disks or as a self-intersection of one disk.}
    \label{fig:clasp}
\end{figure}
\begin{proof}

Since every component of $J$ is null-homotopic, $J$ bounds an immersed collection of disks $D$. The proof proceeds by modifying these disks in the following steps.

\begin{enumerate}
    \item Remove the branch point self-intersections of the disks in $D$.
    \item Remove the circles of double points.
    \item Remove the triple points.
    \item Remove the ribbon singularities.  
\end{enumerate}

To achieve step (1), we preemptively add branch points to each disk in $D$ so that each has an even number of branch points (Figure~\ref{fig:nocusps}, top). In Figure~\ref{fig:nocusps} (middle), we illustrate how to move a branch point within a disk without introducing new branch points. We can thus move a pair of branch points (of opposite signs) in $D$ to look like Figure~\ref{fig:nocusps} (bottom left), and then remove the pair entirely. Repeat until there are no branch point self-intersections in $D$.

\begin{figure}
    \centering
    \includegraphics[width=80mm]{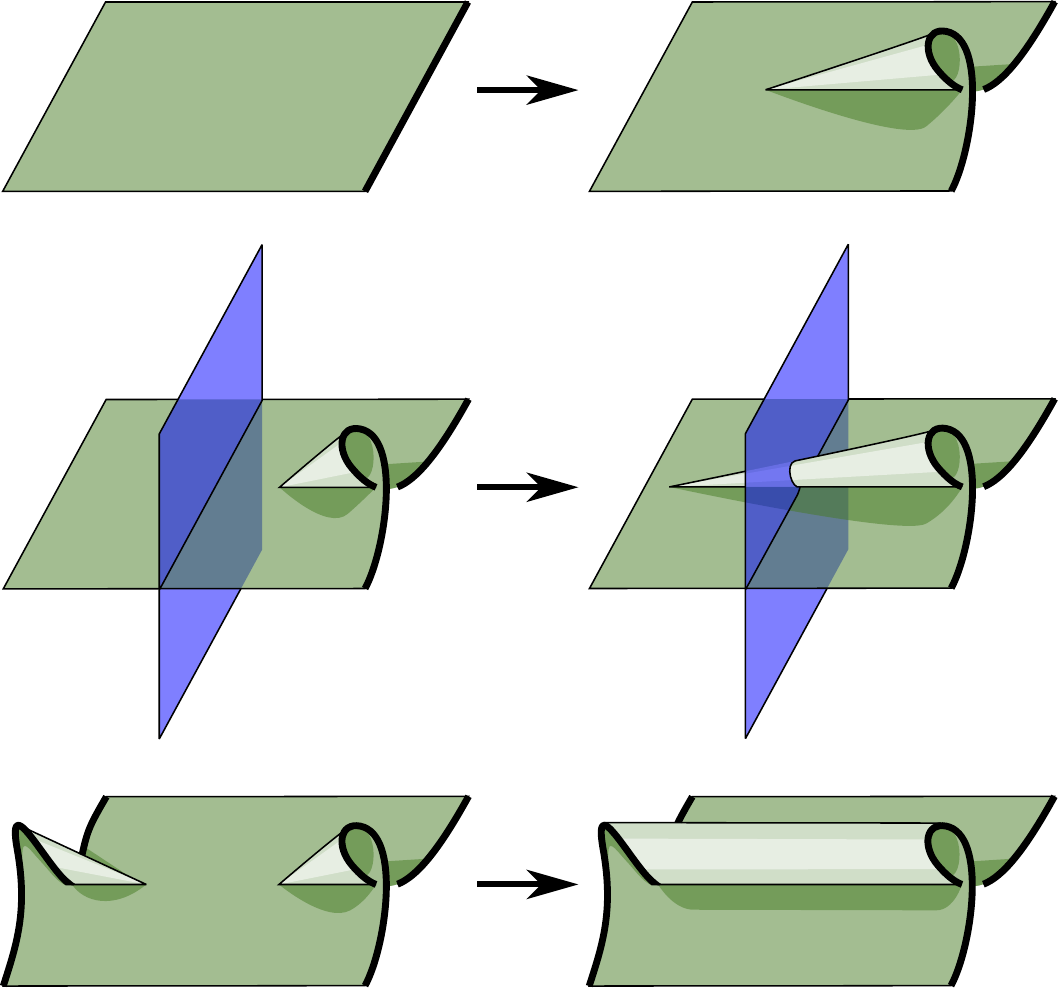}
    \caption{Step (1) of the proof of Lemma~\ref{lem:clasp} requires us to manipulate branch points in the immersed disks until eventually cancelling them all in pairs.}
    \label{fig:nocusps}
\end{figure}

For step (2), we use the finger move in Figure~\ref{fig:breakcircles}. We let $C$ be a circle of intersection in $D$ which ich outermost,% the double arc $C$ in Figure~\ref{fig:breakcircles} is outermost, 
i.e., there is an arc $\eta$ from $C$ to $J$ in $D$ that does not meet any self-intersections of $D$ in its interior. Then we may use $\eta$ as the guiding arc for a finger move on $D$ %meeting self-intersections of the disk in $k$ points. 
as in Figure~\ref{fig:breakcircles}; the effect is to break the circle $C$ and replace it with an arc. %After the finger move, the double point circle is no longer a circle.
We introduce no new branch points or closed circles of self-intersection to $D$. By performing this move to all of the double point circles starting with the outermost and working inward, we eventually eliminate all closed self-intersection circles of $D$.

\begin{figure}
    \centering
    \labellist
    \small\hair 2pt
    \pinlabel{\textcolor{blue}{$C$}} at 70 110
     \pinlabel{\textcolor{darkgreen}{$\eta$}} at 90 80
    \endlabellist
    \includegraphics[width=80mm]{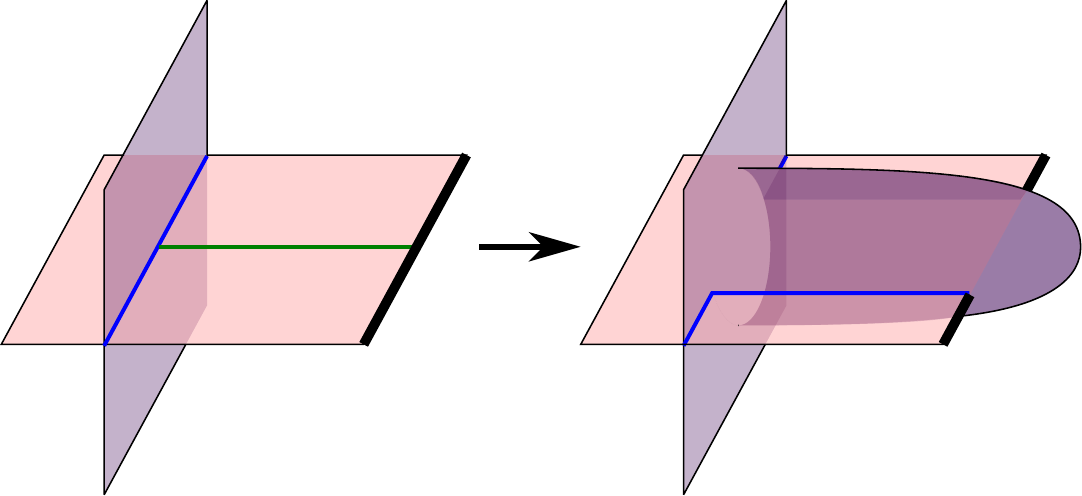}
    \caption{We can break an outermost closed self-intersection circle by performing a finger move.}
    \label{fig:breakcircles}
\end{figure}

For step (3), we use isotopy of $J=\partial D$ as in in Figure~\ref{fig:notriples}. (%We illustrate it as an isotopy of $\partial D$, but of course by 
If desired, one could change perspective and fix $\partial D$ and instead move the interior of $D$ via finger moves.) This move eliminates the triple point of $D$ without introducing any branch points nor closed circles of intersection. Repeat for each triple point.

%By performing this move for each triple point, we may eliminate all triple points.  %By moving along each clasp singularity or ``long'' ribbon singularity and using the finger move in [Fig 2], all of the triple points are eliminated.  In so doing, no new double points circles or cusps are created.

\begin{figure}
    \centering
    \includegraphics[width=80mm]{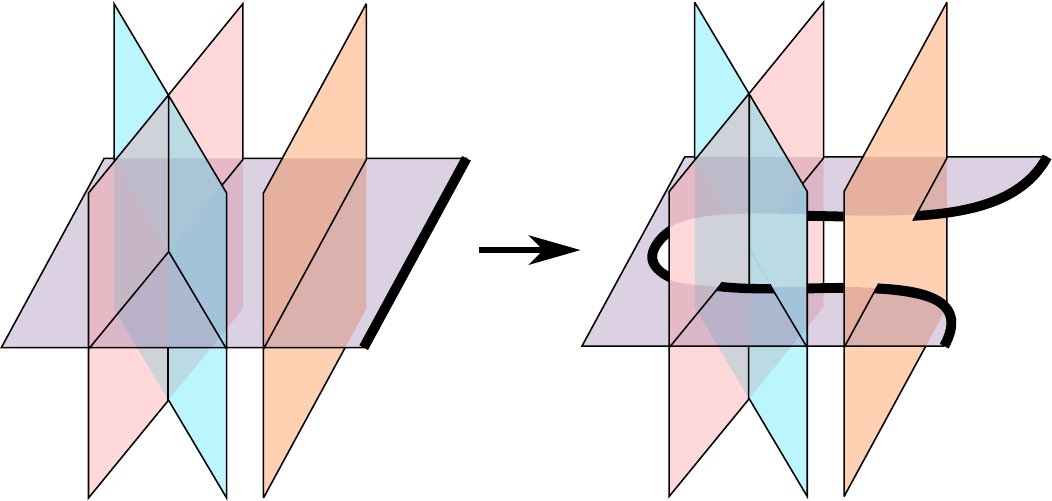}
    \caption{Given a triple point in $D$, we may isotope $J=\partial D$ to eliminate that triple point. This will not introduce any branch points, closed circles of self-intersection, or triple points to $D$.}
    \label{fig:notriples}
\end{figure}

For step (4), the finger move in Figure~\ref{fig:noribbon} is applied to each ribbon intersection. By choosing guiding arcs in $D$ whose interiors are disjoint from the self-intersections of $D$, no new triple points are created.  No branch points or double point circles are created either. Now the disks $D$ have only clasp intersections.
\begin{figure}
    \centering
    \includegraphics[width=80mm]{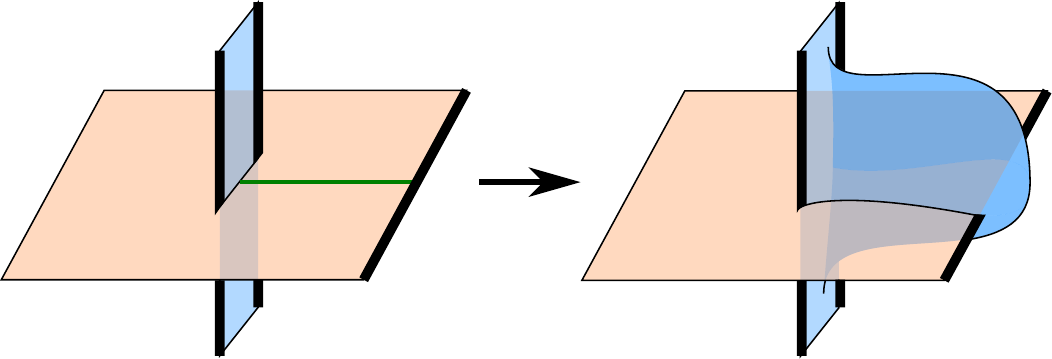}
    \caption{Given a ribbon intersection in $D$, we may perform a finger move to eliminate the ribbon intersection and add two clasp intersections.}
    \label{fig:noribbon}
\end{figure}
\end{proof}

Similar arguments to Lemma~\ref{lem:clasp} appear in~\cite{fq} and~\cite{mellor_melvin}. Lemma~\ref{lem:clasp} motivates the following move, in which we achieve crossing changes of singular circles (at the cost of introducing new singular circles).

Let $\gamma \subset Y$ be a framed arc between two singular circles $A$ and $B$ (possibly with $A=B$) such that the end points of $\gamma$ are not paired by $H$. We will homotope $H$ in a small 3-ball containing $\gamma$.  The effect on $L$ is to achieve a crossing change of $A$ and $B$ (along $\gamma$) and to add a dual pair of type I circles linking $A$ and $B$ as meridians.% that it introduces two new dual singular circles that link the old link of singular circles in a simple way - namely, as meridians. 

\begin{figure}
    \centering
    \labellist
    \small\hair 2pt
    \pinlabel {$A'$} at 110 510
    \pinlabel {$A$} at 160 510
    \pinlabel {$B$} at 215 510
    \pinlabel {$B'$} at 270 510
    \pinlabel {$A'$} at 595 510
    \pinlabel {$A$} at 645 510
    \pinlabel {$B$} at 700 510
    \pinlabel {$B'$} at 755 510
    \pinlabel {$E$} at 555 595
    \pinlabel {$E'$} at 797 595
    \endlabellist
    \includegraphics[width=120mm]{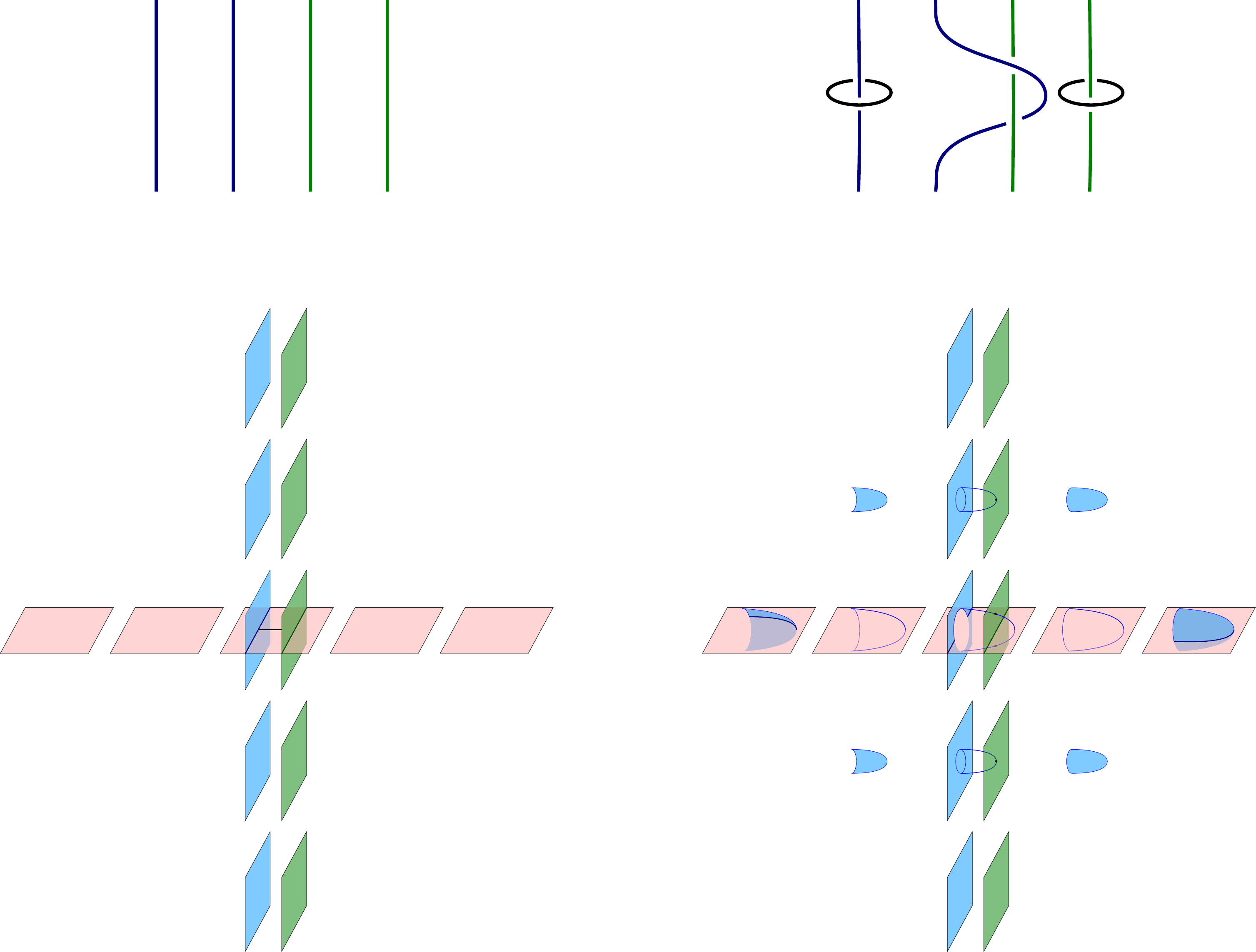}
    \caption{Top left: the singular link bounds immersed disks including a clasp between curves $A$ and $B$. (Note that $B$ might be equal to $A$ or $A'$. Top right: we perform a finger move to the image of $H$ to remove the clasp at the cost of introducing a pair of singular circles $E, E'$ that are meridians of $A'$ and $B'$. Bottom: the finger move in the image of $H$ the achieves a crossing change of the $A$ and $B$ curves.}
    \label{fig:claspmove}
\end{figure}
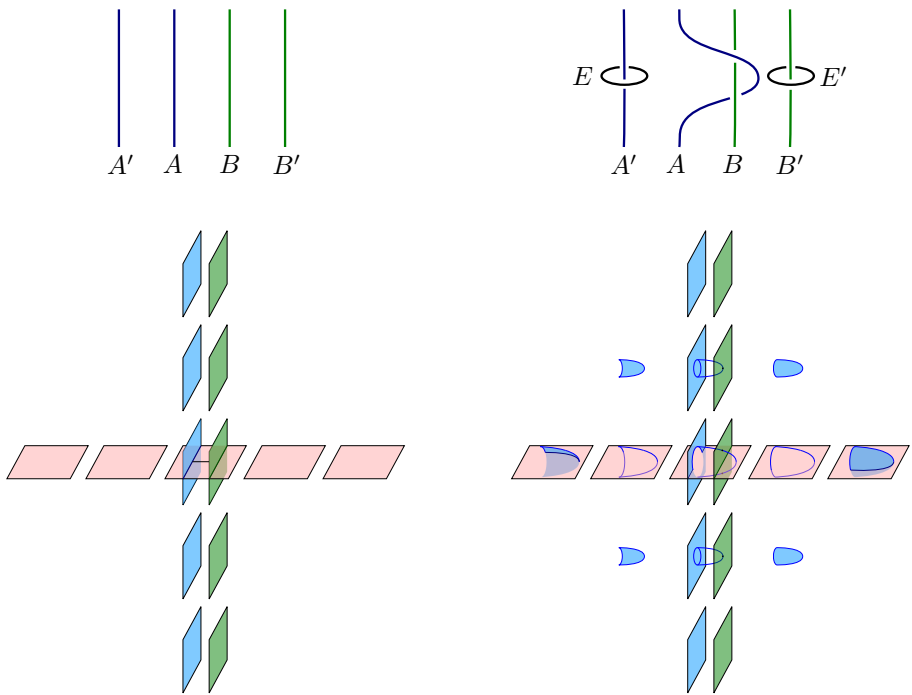

The move is shown in Figure~\ref{fig:claspmove}; here $\gamma$ is implicitly a horizontal arc between $A$ and $B$.  %We assume that one end of $\gamma$ is on a singular circle $A$ and the other is on a singular circle $B$ (we may have $A = B$). 
%The finger move is supported in the ball $B_A$ around the endpoint of $\gamma$ that is on $A$ and
We say that the finger move is a \emph{finger move along $\gamma$ for $A$ to $B$}. If $A$ and $B$ correspond to group elements $a,b$ respectively, then the introduced meridians $E$ and $E'$ have associated group element $e=ab^{-1}$. %Note that while this move is distinct from the finger moves discussed in Section~\ref{sec:fingermove} in the location of the framed arc, it is performed similarly. %If $A$ and

\begin{remark}\label{rem:fingermove}
It is worth noting the following simpler version of the {\emph{finger move}}. Let $\gamma$ be an arc with boundary on $H(Y)$ (away from $H(L)$ and with interior disjoint from $H$. Choose whiskers in $H(Y)$ (also away from $H(L)$) from each endpoint of $\gamma$ to the basepoint; orient these whiskers to form a based loop with $\gamma$ representing group element $g$. Then we may homotope $H$ by performing a finger move in which we push $\gamma(0)$ along $\gamma$ until introducing a new circle of self-intersection to the image of $H$ near $\gamma(1)$. The effect to $L$ is to add a dual pair of type I circles that are unknots unlinked with each other and all other components of $L$ and associated to the group element $g$.
\end{remark}

\subsection{Ambient Dehn surgery to remove dual singular circles}\label{sec:ambientdehn}

%Now suppose that 
%Similarly, if 
%$H$ has a dual pair of singular circles $C,C'$ that form a 2-component unlink that is split from the remaining singular circles of $H$. %, then we may reverse this procedure.
%In Figure {\blue{[FIGURE]}} we illustrate a local model of the image of $H$ near $H(C)=H(C')$. In particular, 

%to homotope $L$ and 

\label{sec:ambient_dehn}

In this subsection, we introduce a move that involves performing ambient Dehn surgery on $H(Y)$ in order to remove a dual pair of type I singular circles from $L$.  Here, we will assume that $H(Y)$ has a dual sphere $G$. %S_0$ has an immersed dual $G$. We isotope $G$ into %there is an immersed sphere $G$ in $X$ that intersects $S_0$ transe in a single point.%

%We will not work directly with $G$ but rather we can homotope $G$ into $X \times I$ so that it is an embedded sphere that intersects $H(S^2 \times I)$ in a single point.  Note that the resulting sphere has the same $w_2$ as $G$ when we identify $H^2(X \times I; \mathbb{Z}/2) = H^2(X; \mathbb{Z}/2)$.  We will again call the resulting sphere $G$.  In fact, the existence of such a $G$ in $X \times I$ that intersects $H$ in a single point is all that is needed.  

Suppose that $A$ and $A'$ are a dual pair of type I circles in $L$ with the property that both are unknots and $A'$ is unlinked from $L-A'$, i.e., $A'$ bounds an embedded disk $D_1$ in $Y$ that is disjoint from the other singular circles. 

%We perform two Dehn surgeries on $H(S^2 \times I)$ in such a way that the resulting immersed manifold in $X \times I$ is again $S^2 \times I$, and we call the new immersion again $H$ and analyse the effect on the singular circles $L$.  

%The first of these Dehn surgeries is done along the knot $A$ using the disk bounding $A'$. % as shown in \red{[Figure - showing dehn surgery along $A$ (maybe part (a) of a 2-part figure together with the next ones?)]}. 
%Specifically, we take a disk $D_1$ in $Y$ bounded by $A'$, and
Let $D_1:=H(D_1)$ and thicken this disk in the two normal directions to $Y$ (normal to the sheet containing $A'$). The result is a 4-dimensional 2-handle that may be viewed as attached to $H(Y)$ along $A$. We surger $H(Y)$ along this 2-handle to obtain an immersion $H':M\to W$ of a 3-manifold $M$. The manifold $M$ is obtained from $Y$ by integral Dehn surgery along $A$. Note that we cannot control the framing of this Dehn surgery, which we refer to as the {\emph{relative framing on $A$}}. The self-intersection link of $H'$ agrees with that of $H$, except we have removed $A$ and $A'$.%the intersection $H(A)=H(A')$.

Let $\mu$ be a meridian of $A$ in $Y$. This meridian bounds a disk $D_2$ in $Y$ meeting $A$ in one point. Let $D_2:=H(D_2)$, so that $D_2$ is now a disk in $X\times I$. The boundary of $D_2$ is naturally viewed as living in $H(Y)$, but we may also take its boundary to be in $H'(M)$. 

Perturb the interior of $D_2$ to intersect $H'(M)$ transversely; this intersection consists of one point in $H'(A')$. %This disk intersects $Y$ in one point in its interior; this point is contained in $H(A')$. If we attempt to push the interior of $D_2$ off $Y$ %(near the sheet containing $A$), then $D_2$ will intersect the sheet containing $A'$ transversely once. We remove this intersection by tubing
Tube $D_2$ to the dual sphere $G$ to obtain a disk $D_2'$ whose interior is disjoint from $H'(M)$. %; see \red{[Figure showing the disk for the meridian]}.
We then again thicken $D_2'$ to obtain a 4-dimensional 2-handle, which we use to surger $H'(M')$ along $H'(\mu)$, yielding another immersion $H'':N \to W$. The 3-manifold $N$ is obtained from $Y$ by integral Dehn surgery on a Hopf link (with components corresponding to $A$ and $\mu$). If $G$ has trivial normal bundle, then may arrange for the framing of the surgery along $\mu$ to be any integer, whereas if $G$ has nontrivial normal bundle then we may choose this framing only up to parity. Details on this framing may be found for example in  
%The possible framings for the surgeries resulting from tubing this disk to $G$ depends on $w_2(G)$.  If $w_2(G)=0$, then we may arrange for this Dehn surgery along $\mu$  %\blue{MM: maybe just link to picture in old paper?}
%to be 0-framed (or any even integer) whereas if $w_2(G)=1$, then the Dehn surgery along $\mu$ may be taken to $1$ (or any odd integer).  The specific framing depends on how we thicken $D_2'$ - details can be found for example in section 5 of 
\cite[Section 5]{km}.  The key points to notice are:
\begin{itemize}
    \item The choice of thickening of the disk $D_2$ to form a 4-dimensional 2-handle is not unique. By considering different thickenings, we may achieve all potential framings of some parity. %By appropriately thickening the disk $D_2$, either all of the even framing surgeries or all of the odd framing surgeries can be achieved by ambient Dehn surgery along $D$
    \item We may twist $D_2$ once about its boundary to introduce a new intersection points between $D_2$ and $H'(M)$, and then remove this intersection by tubing $D_2$ to a copy of the dual sphere $G$. When $G$ has trivial normal bundle, this does not effect which framings of $\partial D_2$ (as a curve in $H'(M)$) extend over $D_2$: hence, we may change the induced framing of the Dehn surgery by $\pm 1$. This fails when $G$ has nontrivial normal bundle.% Tubing a particular disk to a sphere $G$ preserves the parity of the possible ambient Dehn surgeries if $w_2(G) = 0$ and changes the parity of the possible ambient Dehn surgeries if $w_2(G) = 1$.
\end{itemize}

%Because $G$ is framed, we may take this Dehn surgery to be 0-framed, so the result is once again a copy of $Y$.

%We will not control the surgery coefficient for the Dehn surgery done along $A$ which we will refer to as the \emph{relative framing}.  However, the Dehn surgery done along the meridian will be either even or odd depending on if $w_2(G)=0$ or $w_2(G)=1$, and by choosing the 2-dimensional subbundle of the normal bundle of $G$ appropriately, we may assume that the surgery coefficient is 0 and 1 in these two respective cases.  

Thus, if $G$ has trivial normal bundle then we may ensure that $N$ is homeomorphic to $Y$. In this case, we redefine $H=H''$. The total effect
%Thus, in total we perform a pair of Dehn surgeries to our original $Y$ along a Hopf link, where one framing is an arbitrary integer and the other is 0 or 1 depending on whether $w(G)=0$ or $w(G)=1$. Thus, if $w_2(G)=0$ then $Y'\cong Y$. The 
%The pair of Dehn surgeries is along a Hopf link in $S^2 \times I$ and therefore, if either of the two surgery coefficients is $0$, then the resulting space $Y'$ is again $S^2 \times I$.  If $w_2(G) = 0$, then we can assume the framing along the meridian circle is $0$ and therefore the result of the surgeries is $S^2 \times I$.  The
 on $L$ is to remove $A$ and $A'$ while leaving the other singular circles unchanged (i.e., Figure~\ref{fig:addtwists} with $k=0$).

 In the case where $G$ has nontrivial normal bundle but the relative framing on $A$ happens to be $0$, then again $N \cong Y$. If $H(Y)$ is $s$-characteristic, the framing of the surgery along $\mu$ is forced to be odd. We can choose the framing along $\mu$ to be any odd integer $k$, and similarly redefine $H:=H''$. The effect on $L$ is to remove $A$ and $A'$ while adding $k$ full twists to the singular circles that link $A$, as in Figure~\ref{fig:addtwists}.% \red{[Figure 4 of stong but let's do it with $k$ arbitrary - we actually need that (at least +1 and -1) later]}.  

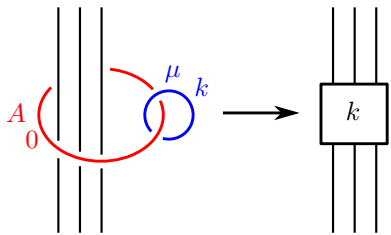
\begin{figure}
    \centering
    \labellist
    \small\hair 2pt
    \pinlabel{\textcolor{red}{$A$}} at -8 45
    \pinlabel{\textcolor{red}{$0$}} at -2 35
    \pinlabel{\textcolor{blue}{$\mu$}} at 51 60
    \pinlabel{\textcolor{blue}{$k$}} at 62 55
    \pinlabel{$k$} at 119 46
    \endlabellist
    \includegraphics{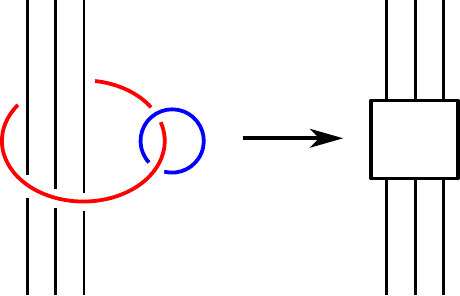}
    \caption{The dual singular circles $A$ and $A'$ are unknots and $A'$ (not pictured) is unlinked from all other circles. The circle $\mu$ is a meridian of $A$. We perform two ambient Dehn surgeries along disks bounded by $H(A)$ and $H(\mu)$, as pictured.}
    \label{fig:addtwists}
\end{figure}

\begin{remark}\label{remark:relativezero}
In fact, when $G$ has nontrivial normal bundle we may choose the relative framing of $A$ at the cost of changing $H(Y)$ (and hence $M''$) by ambient connect-sum with an embedded, nullhomologous 3-sphere. This is explained in~\cite[Page 193]{fq}. Since $G$ has nontrivial normal bundle, its normal bundle in $W$ admits a 2-dimensional subbundle whose boundary is a copy $P$ of $S^3$ with the Hopf fibration. Choose an arc $\eta$ from $A'$ to $H^{-1}(H(Y)\cap G)$ disjoint from $L$ in its interior and then connect sum $H(Y)$ to $P$ via the arc $H(\eta)$ from $H(A)$ to $P$. We redefine $H$ to be this newly obtained immersion of $Y$; the singular link is unchanged up to equivalence. (See Figure~\ref{fig:relativeframing} for a shematic of this operation.) This changes the relative framing on $A$ by $\pm 1$, with sign depending on the choice of sign of the 2-dimensional subbundle of the normal bundle of $G$ that yielded $P$. Thus, we may assume that the relative framing on $A$ is zero. 
\end{remark}

\begin{figure}
     \centering
    \labellist
    \small\hair 2pt
    \pinlabel{\textcolor{darkgreen}{$G$}} at 105 95
    \endlabellist
    \includegraphics[width=100mm]{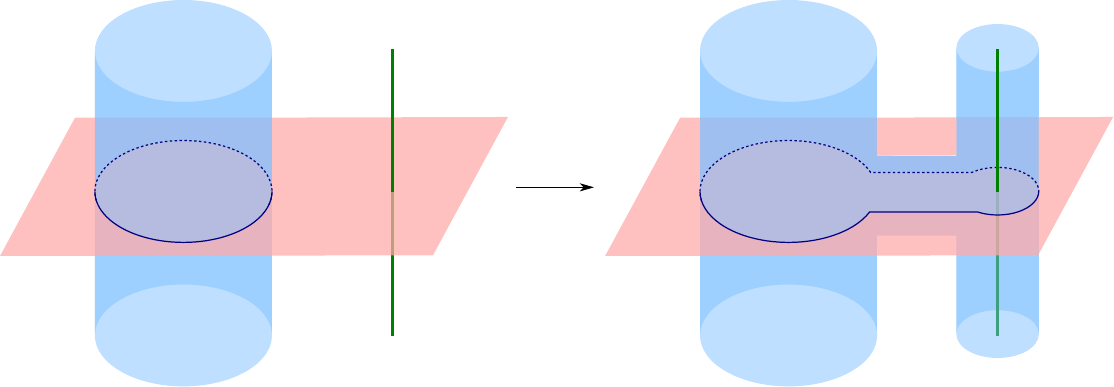}
    \caption{Left: $H(A)=H(A')$ in $W$. The horizontal and cylindrical surfaces are both cross-sections of sheets of $H(Y)$. We assume $G$ has nontrivial normal bundle. Right: We tube $H(Y)$ to a 3-sphere that bounds an Euler number $\pm1$ disk-bundle over $G$ to obtain a new immersion, which we now call $H$. This operation does not change the singular link $L$ up to equivalence, but the relative framing induced by $A'$ on $A$ changes by $\pm1$.}
    \label{fig:relativeframing}
\end{figure}

%\red{Giving $S^2$ the canonical orientation (coming from the boundary of $B^3$ with the canonical orientation), we give $S^2 \times I$ an orientation and thus a fundamental class $[S^2 \times I]$.  Going forward, given a singular concordance $H$ we write $[H]$ to be the image of this fundamental class $[S^2 \times I]$ considered in $H_3(X \times I, X \times \{0,1\}; \mathbb{Z})$.}

\begin{remark} \label{rmk:homologous}
Note that if $H'$ is obtained from $H$ by a sequence of ambient Dehn surgeries as in the section, then since $H$ is $\pi_1$-trivial, so is $H'$. This holds because every element of $\pi_1Y$ can be presented by a curve disjoint from the surgery curves used in the creation of $H'$. Moreover, we have $[H] = [H']$ in $H_3(W,H(\partial Y); \mathbb{Z}\pi_1 W)$ since $H$, $H'$ are cobordant via the trace of the 2-handle surgeries.  %This gives a $\pi_1$-trivial immersion $T:C\to W$ with $\ %the surgery carried out by attaching the handles defines an (immersed) manifold in $W \times I$ cobounding the image of $H$ considered in $W \times \{0\}$ and the image of $H'$ considered in $W \times \{1\}$.  Note that since $H$ was $\pi_1$-trivial, then so is $H'$, since every element in $\pi_1 Y$ can be made disjoint from the location of the surgery.  In addition, if $W$ has non-empty boundary then $[H] = [H']$ in $H_3(W, \partial W; \mathbb{Z})$.}
\end{remark}

%In fact, with $w_2(G) = 1$, by choosing an appropriate disk, we have full control over the surgery coefficient for $A$ (in particular we can, and will, assume that it is $0$ and thus that the resulting 3-manifold is $S^2 \times I$.  Note that by appropriately thickening $D_1$, the relevant issue is whether or nor the relative framing on $A$ is even or odd.  If the relative framing is even, then by choosing an appropriate thickening, it can be assumed to be 0.  So assume the relative framing is odd.  Perform a boundary twist along the disk $D_1$ \red{include a boundary twist pic?} so that it intersects $H(S^2 \times I)$ in a single point and call the resulting (isotopic) disk $D_1'$.  Carrying along the particular thickening of $D_1$, we see that since this is an isotopy, the potential framings of surgery on $A$ coming from using $D_1'$ are also odd.  But now, by tubing $D_1'$ into $G$ using  the fact that $D_1'$ intersects $H(S^2 \times I)$ in a single point, we obtain a new disk $D_1''$ whose interior is now disjoint from $H(S^2 \times I)$ and such that the framings of surgeries along $D_1''$ are even as desired. 

%In either of these cases, the result is a new immersion of $S^2 \times I$ and the singular set has one fewer pair of type I singular circles. 

% whereas when $w_2(g) =1$, we must add an odd number of full twists into the circles that pass through $A$.  

Making use of Remark~\ref{remark:relativezero}, we are now able to prove Lemma~\ref{lem:stong1}. In order to prove this lemma, it was essential to understand relative framings and how we may manipulate them given an unframed dual. We restate the lemma for convenience.

\begin{lemstong1} 
%Let $S_0, S_1$ be embedded 2-spheres in a 4-manifold $X^4$ and let $H : S^2 \times I \to X^4 \times I$ be an immersion with $H(S^2 \times \{0\}) = S_0$ and $H(S^2 \times \{1\}) = S_1$.
Assume that $H$ is $s$-characteristic in $X$ with a dual sphere $G$ and let $A$ and $A'$ be a dual pair of type I circles in $L$. Assume $A$ and $A'$ are both null-homotopic in $Y$. Then $$\lk(A, L - A) = \lk(A', L - A') \pmod{2}.$$%{\blue{MM: note: do we ever need type II case?}}
\end{lemstong1}

\begin{proof}
First we consider the simple case that $A$ and $A'$ are unknots that are unlinked from each other, so $A$ and $A'$ bound disjointly embedded disks $D_A, D_{A'}$ in $Y$ (although we expect both disks to intersect the other components of $L$). By Remark~\ref{remark:relativezero}, we may change $H$ by tubing it to embedded 3-spheres, preserving $L$, so that the relative framing induced by $D_{A'}$ on $A$ is zero. %Then $$
%{\blue{Let $D_A'$ be an immersed disk in $Y$ bounded by $A'$. %and inducing the 0-framing on $A'$. %and $D_{A'}$ be immersed disks in $Y$ with $\partial D_A = A$ and $\partial D_{A'} = A'$, chosen so that each induces the 0-framing on its boundary.
%By Remark~\ref{remark:relativezero}, we may change $H$ (while preserving $L$) to arrange that the relative framings induced by $D_A'$ and $D_{A'}$ on $A'$ and $A$ (respectively) are zero.
Then $H(D_A)\cup H(D_{A'})$ (perturbed to be embedded) is a 2-sphere in $W$ with trivial normal bundle. Moreover, we can isotope $H(D_A)\cup H(D_{A'})$ so that it intersects $H(Y)$ transversely exactly once for every intersection of $L$ with the interiors of $D_A$ and $D_{A'}$. We conclude that $L$ intersects $D_A$ and $D_{A'}$ in the same parity -- that is, we conclude that $\lk(A, L - A) \equiv \lk(A', L - A')\pmod{2}$, as desired.
\end{proof}

%Moreover, we may use Remark~\ref{remark:relativezero} to remove essential circles in $L$ with the aid of an unframed dual sphere.

%\begin{proposition} \label{prop:removeessential}
%Suppose $H(Y)$ has a dual sphere $G$ with nontrivial normal bundle. Then there exists another $H'$ with $H'(Y)$ homologous to $H(Y)$ so that each component of $L$ is nullhomotopic.
%\end{proposition}

\begin{remark}\label{ambientdehnremark}
To summarize, the move in this subsection allows us to remove a pair of dual singular circles $A,A'$ when we have that:
\begin{enumerate}
    \item $A$ and $A'$ are unknotted and $A'$ is unlinked from all other singular circles,
    \item There exists a 2-sphere $G$ in $W$ that intersects $H(Y)$ transversely in a single point.
    \begin{enumerate}
        \item\label{ambientdehna} If $G$ has trivial normal bundle, then the effect to $L$ is to simply remove $A$ and $A'$, leaving the other singular circles unchanged.
\item If $G$ has nontrivial normal bundle, then we must add some number $n$ of full twists to the circles that link $A$. If $H(Y)$ is not $s$-characteristic then we can find some other dual sphere with trivial normal bundle and arrange for $n$ to be zero, as in \eqref{ambientdehna}. But if $H(Y)$ is $s$-characteristic, then we may only take $n$ to be any odd integer. We will always take $n\in\{\pm 1\}$).
    \end{enumerate}
\end{enumerate}

%Thus, by performing the aboved sequence of ambient surgeries (regardless of whether $G$ has trivial or nontrivial normal bundle), we may remove the singular circles $A$ and $A'$. 

%We may choose $n$ up to parity, and if $H(Y)$ is $s$-characteristic then we may take $n$ to be any odd number (and specifically will usually take $n\in\{\pm 1\}$).

%When $w_2(G) = 0$ then we can just erase simply remove $A$ and $A'$ without changing the other singular circles, whereas when $w_2(G) = 1$ we must introduce an arbitrary odd number of full twists to any singlular circles passing through $A$.% - \red{figure - just of the twist or no twist just to drive the point home}.  

%In case (i), we may change $H$ with the effect of simply removing $A$ and $A'$ while not changing any other singular circles. In case (ii), we remove $A$ and $A'$ but add $k$ full twists to singular circles linking $A$ for $k$ a freely chosen odd integer.

\end{remark}

%\red{added this:}
%\begin{remark} \label{rmk:homology_unchanged}
%If a map $H$ is modified by one of (or a sequence of) the moves in this section to another map $H'$, then $[H] = [H']$ in $H_3(W;\mathbb{Z})$ and in the case where $W$ has nonempty boundary, $[H] = [H']$ in $H_3(W,\partial W; \mathbb{Z})$.  
%\end{remark}

\section{The Freedman--Quinn and Stong invariants}\label{sec:fqkm}

In this section, we will review the Freedman--Quinn invariant of~\cite{fq,schneiderman_teichner} and Stong invariant of a pair of 2-spheres in a 4-manifold.  We show how the vanishing of the Freedman--Quinn invariant allows us to assume that we have a singular concordance $H : S^2 \times I \to X \times I$ between the two spheres with no type II singular circles.  

\subsection{The Freedman--Quinn invariant and removing type II circles} \label{sec:removing_type_II}

We now restrict to the setting where the ambient 5-manifold is $X\times I$ for some compact, orientable 4-manifold $X$ and the 3-manifold $Y$ being immersed in $X\times I$ is $S^2\times I$.

We begin with a definition of the self-intersection invariant $\mu$ on $\pi_3 X$, and then follow with a definition of the Freedman--Quinn invariant. See~\cite{schneiderman_teichner} and~\cite{km} for more details. 

\begin{definition}[\cite{fq,schneiderman_teichner}]
Let $P\in\pi_3 X$ be represented by a based immersion of a 3-sphere $P : S^3 \to X \times I$ (here we are identifying $\pi_3 X = \pi_3(X \times I)$). Let $A_1,\ldots, A_n$ be the type II singular circles of $P$. For each $k$, %let $\epsilon_k\in\{\pm1\}$ be the sign of the self-intersection of $P(S^3)$ at $P(A_k)$, and 
let $g_k\in\pi_1 X$ be the group element associated to $A_k$. Since $A_k$ is type II, either $g_k$ is trivial or $g_k$ is contained in $T_X := \{g \in \pi_1 X  : g^2 = 1, g \neq 1 \}$.

Let $\mathbb{F}_2 T_X$ denote the vector space over the field with two elements with basis $T_X$.  %As discussed in~\cite[Section 4]{schneiderman_teichner}, there is a natural inclusion of $\mathbb{F}_2 T_X$ as a subgroup of $\mathbb{Z} \pi_1 X / \langle g+g^{-1}, 1\rangle$ %(where here $<g + g^{-1},1>$ is the subgroup of the abelian group $\mathbb{Z} \pi_1 X$ generated by the elements of the form $g+g^{-1}$ for $g \in \pi_1 X$ together with the element $1$)
%given by inclusion
%\begin{align*}
%\mathbb{F}_2 T_X &\to \mathbb{Z} \pi_1 N / \langle g+g^{-1}, 1\rangle \\
              % g &\mapsto [g].
%\end{align*}
We then define the $\mu$ self-intersection invariant to be the map
\begin{align*}
\mu : \pi_3 X &\to\mathbb{F}_2 T_X\\% \mathbb{Z} \pi_1 X / \langle g+g^{-1}, 1\rangle\\
P&\mapsto \sum_{k|g_k\neq 1}g_k.
\end{align*}
\end{definition}

%For a singular-concordance $H$, we define $\mu(H)$ by the analogous formula.  Then $\fq(S_0,S_1)$ is $\$

%We now define $\fq$, which is analogous to $\mu$ with $S^2\times I$ taking the place of $S^3$ once . We quotient by the image of $\mu$ in order to ensure that $\fq$ is well-defined given the boundary of an immersed, based (along an arc) $S^2\times I$.

\begin{definition}[see also~\cite{fq,schneiderman_teichner,km}]\label{def:basedfq}
Let $S_0$ and $S_1$ be based-homotopic %$\pi_1$-trivial surfaces 
2-spheres in a 4-manifold $X$ with basepoint $z$. Choose a based singular concordance $H : S^2 \times I \to X \times I$ from $S_0$ to $S_1$ (e.g., the trace of a based homotopy from $S_0$ to $S_1$). Let $A_1,\ldots, A_n$ be the type II singular circles of $H$. For each $k$, let $g_k\in\pi_1 X$ be the group element associated to $A_k$ (these are involutions so there is no need to distinguish between elements and their inverses). Define $\mu(H)$ to be the element of $\mathbb{F}_2 T_X$ given by $$\mu(H)=\sum_{k|g_k\neq 1} g_k.$$ %Since $A_k$ is type II, either $g_k$ is trivial or $g_k$ is contained in $T_X := \{g \in \pi_1 X  : g^2 = 1, g \neq 1 \}$.
 
 The {\emph{Freedman--Quinn invariant}} $\fq(S_0,S_1)$ of the pair $S_0,S_1$ is the element of $\mathbb{F}_2 T_X/\mu(\pi_3 X)$ represented by $\mu(H)$ for some choice of based singular concordance $H$ between $S_0$ and $S_1$.
\end{definition}

In Definition~\ref{def:basedfq}, it is essential that $H$ is based. %Consider the case that $Y\cong S^2\times I$.
If $H'$ is another based immersion of $S^2\times I$ from $S_0$ to $S_1$, then the images of $H$ and $H'$ agree at their boundaries and along the vertical arc $z\times I$. 
The remaining pieces of the images of $H$ and $H'$ are immersed 3-balls that together form an immersed $S^3$ in $X\times I$. We conclude that $\mu(H)$ and $\mu(H')$ differ by an element of $\mu(\pi_3 X)$ -- see~\cite{schneiderman_teichner} for details. 

%When $Y=\Sigma\times I$ for a positive-genus, we make use of $\pi_1$-triviality of $H$ to ``compress" $H$ to obtain an immersion of $S^2\times I$ while preserving $\mu(H)$. We conclude by the above paragraph that $\mu(H)$ is well-defined up to addition with elements in $\mu_3(\pi_3 X)$. See~\cite[Lemma 6.1]{km} for a complete proof.

%To avoid requiring the condition that $H$ contains $z \times I$ and that $N$ be a product, we consider $\fq$ up to conjugacy by $\pi_1(X)$.

%\begin{definition}[\cite{fq,st}]\label{def:basedfq}
%Assume that $S_0$ and $S_1$ are homotopic and $N=X\times I$.
%The {\emph{Freedman--Quinn invariant}} $fq(S_0,S_1)$ of the pair $(S_0,S_1)$ is the orbit of the action by conjugation of $\pi_1(X)$ on $\mathbb{F}_2 T_X/(\mu_3(\pi_3(X)))$ that contains $\mu_3(H)$.
%\end{definition}

%More generally, we want to consider 2-spheres in the boundary of a general 5-manifold.

We now deal with the issue (or rather non-issue) of basepoints in the presence of a dual sphere $G$.%one of the spheres $S_0$ or $S_1$.  %Note that %if $g \in T_X$, then $hgh^{-1} \in T_X$ for any $h \in \pi_1(X)$, so
%there is an action of $\pi_1(X)$ on $\mathbb{F}_2 T_X$ given by conjugation.  This action fixes $\mu(\pi_3(X))$ since
%$$
%\mu(h \cdot a) = h \mu(a) h^{-1}
%$$
%for any $h\in\pi_1(X),a\in\pi_3(X)$
%where $a \in \pi_3(X)$, $h \cdot a$ is the usual action of $\pi_1 X$ on $\pi_3 X$.

\begin{proposition}[{\cite[Lemma 2.1]{schneiderman_teichner}~\cite[Proposition 4.9]{km}}]\label{prop:no_base_points}
Let $S_0$ and $S_1$ be homotopic 2-spheres %$\pi_1$-trivial surfaces 
in a 4-manifold $X^4$. Assume there is an immersed 2-sphere $G$ in $X$ with $G\cdot S_i =  1\pmod{2}$. Let $H$ and $H'$ be (not necessarily based) singular concordances between $S_0$ and $S_1$. Then $\mu(H)=\mu(H')$ as elements of $\mathbb{F}_2 T_X/(\mu_3(\pi_3 X))$.%, where here the basepoints of the domains of $H$ and $H'$ are the points $H^{-1}(z_0)$ and $H'^{-1}(z_0)$, respectively.
\end{proposition}
\begin{proof}
Since $G\times I\cap H(S^2\times I)$ is a 1-manifold with an odd number of endpoints in each of $S_0\times\{0\}, S_1\times\{1\}$, there is an arc $\eta$ in $G\times I\cap H(S^2\times I)$ that runs from $X\times0$ to $X\times 1$. By an isotopy in a neighborhood of $S_1\times\{1\}$, we can arrange for $\eta$ to project to a closed loop $C$ in $X$. Since $Y$ is simply connected, $C$ represents a well-defined element of $\pi_1 X$. (Note also that since $S_1$ is simply connected, the choice of isotopy near $S_1$ that made $C$ a closed loop does not affect the group element represented by $\eta$.)

This closed loop $C$ is contained in the 2-sphere $G$ away from its self-intersections, so $C$ is nullhomotopic. We can thus isotope $H$ so that there is a vertical arc $\eta$ in $H(S^2\times I)\cap G\times I$. We repeat from $H'$ to find a vertical arc $\eta'$ in $H'(S^2\times I)\cap G\times I$, and then isotope these arcs to agree. The proof now follows from the discussion after Definition~\ref{def:basedfq}.
\end{proof}

Combining Definition~\ref{def:basedfq} and Proposition~\ref{prop:no_base_points}, we conclude that the Freedman--Quinn invariant $\fq(S_0,S_1)$ is well-defined if $S_0,S_1$ are homotopic (not necessarily based-homotopic) and there is an immersed 2-sphere $G$ in $X$ intersecting each of $S_i$ in an odd number of points. Note that if $H$ is a singular concordance between $S_0$ and $S_1$, and $H(S^2\times I)$ intersects some 2-sphere $G$ transversely once, then the projection of $G$ to $X$ is an immersed sphere intersecting each $S_i$ in an odd number of points.

\begin{definition}[\cite{fq,schneiderman_teichner}]\label{def:fq}
Let $S_0$ and $S_1$ be homotopic embedded 2-spheres in an orientable 4-manifold $X^4$ such that there exists an immersed 2-sphere $G$ in $X$ with $G\cdot S_i =  1\pmod{2}$. Let $H : S^2\times I \to X \times I$ be a singular concordance between $S_0$ and $S_1$.  The {\emph{Freedman--Quinn invariant}} $fq(S_0,S_1)$ of the pair $(S_0,S_1)$ is defined to be $\mu(H)$ viewed as an element of $\mathbb{F}_2 T_X/(\mu_3(\pi_3 X))$.
\end{definition}

Note that if $S_0,S_1$ are based-homotopic, then this definition coincides with Definition~\ref{def:basedfq}. %In particular, we need not specify that $H$ be based along an arc, explaining why this was not a hypothesis in Section~\ref{sec:moves}, since we will % -- hence, we may apply the moves of Section~\ref{sec:moves} to $H$ while preserving

\begin{proposition}\label{fq_typeII}
Suppose that $S_0$ and $S_1$ are homotopic embedded 2-spheres in a 4-manifold $X$ such that there exists an immersed 2-sphere $G$ in $X$ with $G \cdot S_i = 1 \pmod{2}$ and such that  $\fq(S_0,S_1)=0$.  Then there exists a singular concordance $H$ between $S_0$ and $S_1$ with no type II singular circles.
\end{proposition}
 
%Note that the dual sphere present in the statement of Theorem~\ref{thm:main} is present in Proposition~\ref{prop:fq_typeII}.

%In the proof of Proposition~\ref{fq_typeII}, removing essential singular circles follows from Proposition~\ref{prop:removeessential}. We include this requirement in the conclusion of Proposition~\ref{fq_typeII} simply for convenience.
The main tool in proving Proposition~\ref{fq_typeII} is the Whitney move of Section~\ref{sec:whitney}. In brief, our strategy is to use the fact that $\fq(S_0,S_1) = 0$ to initially choose $H$ to have $\mu(H)=0\in \mathbb{F}_2 T_X$, rather than $\mu(H)=0$ only after quotienting by $\mu(\pi_3 X)$.  We will then use the Whitney move to remove type II circles in pairs.  The existence of such cancelling pairs is guaranteed by the assumption that $\fq(S_0,S_1) = 0$.  A different approach to this argument for eliminating the type II singular circles using the hypothesis that $\fq(S_0,S_1) = 0$ is given in~\cite{km} (see in particular Figure 12).  

\begin{proof}[Proof of Proposition~\ref{fq_typeII}]

Let $H_1:S^2 \times I\to X\times I$ be a singular concordance between $S_0$ and $S_1$.  Since $\fq(S_0,S_1)=0$, we have $\mu(H_1)=0\in\mathbb{F}_2T_X/\mu(\pi_3 X)$. Then there exists a based immersion $P:S^3\to X\times I$ with $\mu(H_1)+\mu(P)=0\in\mathbb{F}_2 T_X$. Push the basepoint of $P$ off the basepoint of $X\times I$, interior to $H_1(S^2 \times I)$, and obtain a new immersion $H_2:S^2\times I \to X\times I$ by connect-summing $H_1$ and $P$. Note that the singular link $L_2$ of $H_2$ is a split union of those of $H_1$ and $P$, along with more type I circles corresponding to intersections of $H_1(S^2 \times I)$ with $P(S^3)$. Thus, $\mu(H_2)=\mu(H_1)+\mu(P)=0\in\mathbb{F}_2 T_X$.

%be an immersion obtained by whose image is obtained by deleting small balls from the interiors of $\Im(H_1)$ and $\Im(P)$ 
%and then connecting the resulting boundaries with an embedded tube along an arc with interior disjoint from $\Im(H_1)$ and $P$. (The singular link of $H_2$ is thus a split union of the singular links of $H_1$ and $P$.)

%Then $\mu(H_2)=\mu(H_1)+\mu(P)=0\in\mathbb{F}_2 T_X$.

%Using the interpretation of $\mu_3$ and $\fq(S_0, S_1)$ as in Propositions 4.3 and 4.4 of~\cite{km} (see also~\cite[Section 4.D]{schneiderman_teichner} and~\cite[Chapter 10, Section 9]{fq}), we have arranged that
By definition of $\mu$, every nontrivial 2-torsion element in $\pi_1 X$ is associated to an even number of type II singular circles in $L_2$. We must also arrange for that to be true for the trivial element. To that end, if $L_2$ has an odd number of type II circles associated to the trivial element of $\pi_1 X$, perform the move of Section~\ref{sec:maketrivialtypeII} to introduce a new type II circle associated to the trivial element and denote the resulting immersion by $H_3$. Otherwise, just let $H_3:=H_2$.

Now $H_3$ is an immersion of $S^2 \times I$ whose singular link $L_3$ %has an even number of type II circles, with the
has the property that the type II circles of $L_3$ can be put into disjoint pairs so that in each pair, the two type II circles correspond to the same element in $T_X \cup \{1\}$.

Let $C$ and $C'$ be one such pair of type II circles in $L_3$.  Let $x,x'$ be distinct points in $C$ with $H_3(x)=H_3(x')$. Let $y,y'$ be distinct points in $C'$ with $H_3(y)=H_3(y')$. Choose arcs $\gamma_1,\gamma_2$ in $Y$ from $x$ to $y$ and from $y'$ to $x'$, respectively. Take $\gamma_1,\gamma_2$ disjoint from $L_3$ in their interiors. Since $C$ and $C'$ correspond to the same element of $T_X\cup\{1\}$, $H_3(\gamma_1\cup\gamma_2)$ is nullhomotopic (see Section~\ref{sec:finddisks} and note that since these are type II circles the orderings of $x,x'$ and $y,y'$ are irrelevant). Therefore, there is an embedded disk $D$ in $X\times I$ with $\partial D=H_3(\gamma_1\cup\gamma_2)$. Frame $\gamma_1$ and $\gamma_2$ so as to yield framed bands $b_1,b_2$ for some (not necessarily clean) Whitney disk. Perform a Whitney move to obtain a new immersion which we denote by $H_4 : S^2 \times I \to X \times I$.

Recall from Section~\ref{sec:finddisks} that the singular link $L_4$ of $H_4$ differs from $L_3$ by surgery along the bands $b_1,b_2$, along with the potential addition of dual pairs of type I circles. %that the Whitney move introduces a pair of type I circles to the singular link $L_4$ of $H_4$ for every intersection of the interior of $D$ with the image of $H_3$. Moreover, since it is possible to perform a Whitney move along $D$ to obtain a new immersion, it must be the case that $
By Remark~\ref{rem:checkannulus}, surgering $C\sqcup C'$ along $b_1$ and $b_2$ yields exactly two circles. These two circles in $L_4$ each include one arc in $C$ and one arc in $C'$ and have the same image under $H_4$, i.e., they are a type I dual pair. The net result of the above sequence of moves is that $H_4$ has two fewer type II circles than $H_3$, and the type II circles of $H_4$ still occur in pairs associated to the same fundamental group element. We may therefore repeat this process until obtaining an immersion $H:S^2 \times I \to X \times I$ that has no type II circles.
%Finally, we apply the argument of Proposition~\ref{prop:removeessential} to remove essential singular circles from the singular link $L$ of $H$. Recall that this argument involves introducing a type I pair $C,C'$ to $L$ via a finger move, and then banding $C,C'$ to a type I pair $A,A'$ in $L$ with $A$ essential via a Whitney move (here remembering that $H$ does not have type II circles), while potentially also adding more pairs of dual type I nullhomotopic circles. The result of the band surgery on $A,A',C,C'$ is a type I dual pair, where the component arising from $A,C$ is nullhomotopic and the component arising from $A',C'$ is freely homotopic to $A'$.\mblue{MM: note to self, check that I said this right about A' in the earlier proof} We repeat this argument, never introducing type II circles to $L$, until every component of $L$ is nullhomotopic.
\end{proof}

\begin{remark}\label{rmk:hypos_fq}
Note that if $\mu(\pi_3 X)=0$, then such a map $H$ as in Proposition~\ref{fq_typeII} exists for any desired homology class in $H_3(X \times I, X \times \{0,1\};\mathbb{Z}\pi_1X)$ that can be represented by a singular concordance between $S_0$ and $S_1$.  This is because $\fq(S_0,S_1)=0$ implies that for any singular concordance $H'$ from $S_0$ to $S_1$, we have $\mu(H')=0$ (rather than just vanishing in a quotient). Then we need not add any $\pi_3 X$ elements to obtain $H$ in the proof of Proposition~\ref{fq_typeII}. The other operations used in the proof preserve the homology class of the image
%and under the hypotheses that either $\pi_1 X$ contains no 2-torsion or $\mu(\pi_3 X)$, we will automatically have that $\mu(H) = 0$ and so there is no need to add on a 3-sphere as in the proof of Proposition~\ref{fq_typeII}.  All of the other operations used in the proof of Proposition~\ref{fq_typeII} do not change the homology class $[H]$
(see Remark~\ref{rem:remzpi1class}).% and therefore, the resulting singular concordance represents the same homology class $[H]$.  
\end{remark}

\begin{figure}
\begin{centering}
\labellist
\small\hair 2pt
\pinlabel{$x$} at 26 85
\pinlabel{$x'$} at 27.5 53
\pinlabel{$y$} at 142 85
\pinlabel{$y'$} at 143.5 53
\pinlabel{$\textcolor{red}{b_1}$} at 85 115
\pinlabel{$\textcolor{red}{b_2}$} at 86.5 20
\endlabellist
\includegraphics[width=90mm]{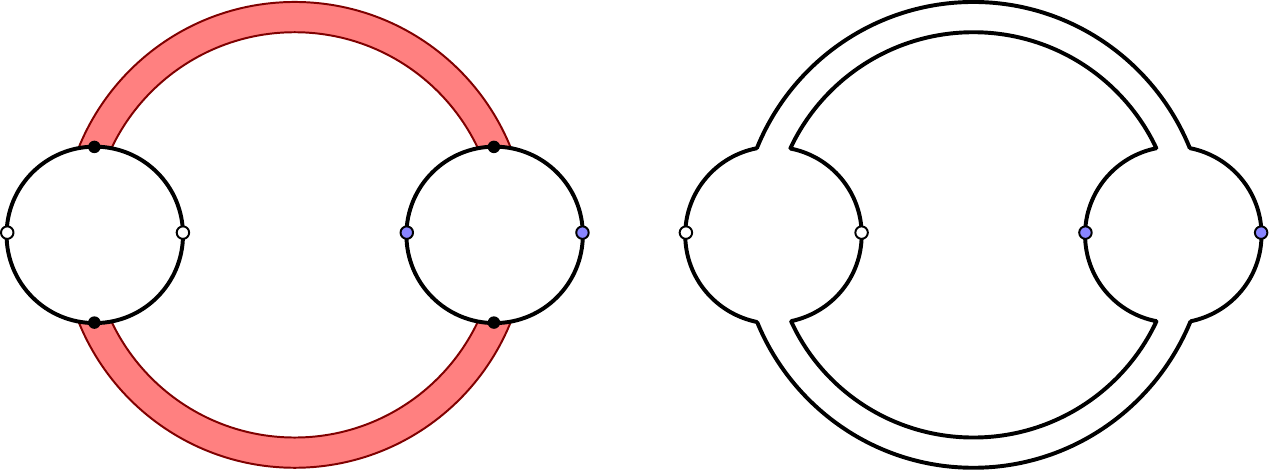}
\caption{Left: the type II curves $C$ and $C'$ in $L_3$. We mark some pairs of points with the same image under $H_3$. The images of the bands $b_1$ and $b_2$ bound a framed Whitney disk. Right: After performing the Whitney move to obtain $H_4$, the two type II circles become one pair of type I circles. (We omit newly introduced type I pairs from this drawing.)}\label{fig:whitneyremovetypeii}
\end{centering}
\end{figure}
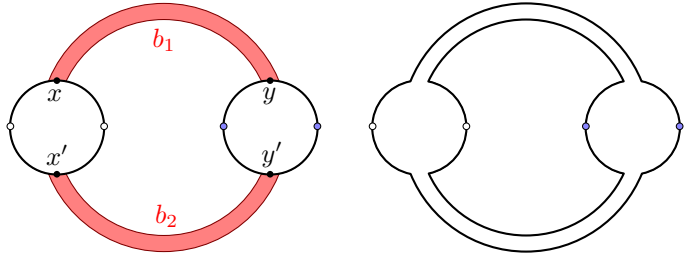

\subsection{The Stong invariant}\label{sec:km}

In this section we define the Stong invariant for certain pairs of homotopic, oriented 2-spheres $S_0,S_1$ embedded in a 4-manifold $X$. Like the Freedman--Quinn invariant, the Stong invariant is defined by studying the intersection set of a singular concordance between $S_0$ and $S_1$. In~\cite{stong}, Stong gives an argument for why this invariant is well-defined (with some additional hypotheses necessary to define the invariant at all) when the element of $H_3(X\times I,X \times \{0,1\};\mathbb{Z}\pi_1)$ represented by a concordance is specified (the theorem of~\cite{stong} is phrased for 3-spheres in a 5-manifold but we discuss later how this extends to pairs of 2-spheres in a 4-manifold -- this is also mentioned in~\cite{schneiderman_teichner}).  Later in this section, we discuss a quotient of the target in which $\km$ is well-defined without specifying the homology class represented by $H(S^2 \times I)$.

%(Or more simply, we discuss what we must quotient the target by in order to identify the value of $\km$ for any two potential homology classes of $H(S^2 \times I)$.) 

%after we quotient the target to ensure that $\km(S_0,S_1)$ is well-defined without specifying a homology class represented by $H(Y)$. 

%In the current paper, we will not give the most general definition of $\km$ since it can be greatly simplified in the presence of a dual sphere, as in the hypotheses of our main theorems {\blue{NUMBERS}}. We refer the reader to~\cite{km2} for the general setting.

%\begin{definition}\label{def:km}
%Let $S_0, S_1$ be homotopic, orientable, $\pi_1$-trivial surfaces in a closed, orientable 4-manifold $X$. Assume there is an immersed 2-sphere $G$ in $X$ intersecting $S_0$ in a single point, and that $S_i$ is $s$-characteristic, and that there is an immersion $H:\Sigma\times I\to X\times I$ mapping boundary to $S_0\times\{0\}\sqcup S_1\times\{1\}$ with the property that the singular link $L$ is a Hopf link contained in a ball, where the two components are a type I dual pair. Then the {\emph{Stong}} invariant $\km(S_0,S_1) 
%\end{definition}

Let $S_0$ and $S_1$ be embedded, $s$-characteristic spheres in an orientable 4-manifold $X$ with $\fq(S_0,S_1) = 0$ (here, either take $S_0,S_1$ to be based-homotopic or to be homotopic with an immersed 2-sphere intersecting each $S_i$ in an odd number of points as in Proposition~\ref{prop:no_base_points}).   
We will define an invariant $\km(S_0,S_1)$ that is valued in a quotient of $H_1(X; \mathbb{Z}/2\mathbb{Z})$. The invariant $\km$ is secondary to $\fq$, in the sense that $\km$ cannot be defined for a pair of spheres with nonzero Freedman--Quinn invariant. Several choices will be made in defining $\km$.  Stong~\cite{stong} shows (see also~\cite{km2}) that $\km$ is independent of these choices.   %In order to state this definition, we must define a certain quantity for type II components of a singular link known as the {\emph{twisting number}}

%If, alternatively, $S_0$ and $S_1$ are merely homotopic (rather than based-homotopic) but there is an immersed 2-sphere in $X$ intersecting each $S_i$ in an odd number of points (as in Proposition~\ref{prop:no_base_points}) and $\fq(S_0,S_1) = 0$, then also define $\km(S_0,S_1)$ analogously.

Before proceeding, we need to define the {\emph{relative twist number}} of a pair of type II components of a singular link $L$.% $A$ in a singular link $L$. Recall that type I components arise in pairs $B,B'$, so we can easily consider the linking numbers $\lk(B,L-B)$ and $\lk(B',L-B')$. The relative twist number essentially gives an analogue for type II components; we ``split $A$ in half" and consider the linking number of one ``half" with the remainder of $L$.

\begin{definition}
Let $A$ and $B$ be type II circles in $L$ that are associated to the same element of $\pi_1(X)$. Let $b_1,b_2$ be framed bands connecting $A$ and $B$. Let $\gamma,\gamma'$ be the two circles that result from surgering $A$ and $B$ along $b_1,b_2$. (Note that if $H(S^2\times I)$ does not have a dual, then we may not be able to find a clean Whitney disk bounded by $H(b_1\cup b_2)$, so we are not actually performing a Whitney move to obtain $\gamma$ and $\gamma'$.) The {\emph{relative twist number}} of $A$ and $B$ is \[\tw(A,B) := \lk(\gamma, \gamma') + \lk(\gamma, L - \{A,B\}))\pmod{2}.\]
\end{definition}

A proof that $\tw$ is well defined is given by Stong~\cite{stong} (see also~\cite{km2}).  We now define the Stong invariant.

%We show that the relative twist number of $A$ and $B$ is well-defined (i.e., does not depend on the choice of $b_1,b_2$) in {\blue{Section~\ref{sec:kmwelldefined}}.\mblue{MM: make sure we actually do this in this half of the paper, or cite appropriately} \red{I think this is rather in the other paper, although we could include it here, I feel like it fits better in with the big ol' proof that km is well defined...}} 

\begin{definition}[\cite{stong}]\label{def:km}
Assume $\fq(S_0,S_1)=0$ for $s$-characteristic, oriented 2-spheres $S_0, S_1$ in a 4-manifold $X$. Let $H:S^2\times I\to X\times I$ be a singular concordance between $S_0$ and $S_1$. Assume that either $S_0,S_1$ are based homotopic and $H$ is based or assume there is a 2-sphere $G$ in $X\times I$ intersecting $H(S^2\times I)$ in an odd number of points.

The singular link $L$ of $H$ consists of pairs of type I components, and type II components each corresponding to an element of $T_X\cup\{1\}$.
\begin{enumerate}
    \item For each pair of type I circles, choose one to be active and the other to be inactive. Denote the active circle by $A$, its associated group element in $\pi_1 X$ is named $a$ (i.e., the element we obtain by going into the sheet $A$ and leaving on the other sheet).%is $a\in\pi_1(X)$.
    \item For each $g\in T_X$, let $L_g$ denote the collection of type II components of $L$ associated to $g$.
    \item For each $h\in\pi_1 X$, let $\epsilon_h$ denote the image of $h$ in $H_1(X;\mathbb{Z}/2\mathbb{Z})$ under $\mathbb{Z}/2\mathbb{Z}$-abelianization.
    \end{enumerate}
    
    Let 
    $$
    \Delta(H) := \sum_{A\substack{\text{ type I,}\\\text{ active}}} \lk(A, L-A) \epsilon_a + \sum_{g \in T_X} \sum_{B,C\in H_g}\tw(B,C) \epsilon_g\in H_1(X;\mathbb{Z}/2\mathbb{Z})
    $$
    
    Let $\Self(S_0)$ denote the set of singular concordances between $S_0$ and $S_0$.  Note that $\Self(S_0)$ has a monoid structure given by stacking, $\Delta : \Self(S_0) \to H_1(X; \mathbb{Z}/2\mathbb{Z})$ is a monoid homomorphism, and since every element of $H_1(X; \mathbb{Z}/2\mathbb{Z})$ is order 2, the image $\Delta(\Self(S_0))$ is a subgroup of $H_1(X; \mathbb{Z}/2\mathbb{Z})$.  

    %The {\emph{Stong invariant}} of $S_0,S_1$ is the equivalence class of $\Delta(H)$ in $H_1(X; \mathbb{Z}/2\mathbb{Z}) / \Delta(\Self(S_0))$ for some singular concordance $H$ between $S_0$ and $S_1$.  
    
    Given a fixed homology class $\alpha\in H_3(X\times I,X\times\{0,1\};\mathbb{Z}\pi_1X)$, the \emph{Stong invariant} of $S_0, S_1$ relative to $\alpha$ is
    $$
    \km(S_0,S_1;\alpha) := \Delta(H) \in H_1(X; \mathbb{Z}/2\mathbb{Z})
    $$
    for some singular concordance $H$ between $S_0$ and $S_1$ with $[H] = \alpha$.
    
    The {\emph{Stong invariant}} of $S_0, S_1$ is \[\Delta(H)\in H_1(X; \mathbb{Z}/2\mathbb{Z}) / (\Delta(\Self(S_0))).\]% for some singular concordance $H$ between $S_0$ and $S_1$. 
\end{definition}
%\begin{proposition}
%The relative twist number $\tq(A,B)$ is well defined, i.e., does not depend on the choice of the framed bands $b_1,b_2$.
%\end{proposition}

We leave it as an easy exercise that for homotopic 2-spheres $S_0,S_1$, we have $\Delta(\Self(S_0))=\Delta(\Self(S_1))$ and thus Definition~\ref{def:km} is symmetric in $S_0, S_1$.

The notation $\Delta$ comes from Stong~\cite{stong}. The invariant $\Delta$ serves a similar role in the definition of $\km$ that $\mu$ does for $\fq$. 
It should not be obvious why $\Delta(H)$ and $\km(S_0,S_1;\alpha)$ are well defined, as this is a theorem of Stong (stated in a slightly different context). We plan to explain this in significant detail for concordances in~\cite{km2}.

Note that, if $S_0$ and $S_1$ are concordant, then $\km(S_0,S_1) = 0$.  Similarly, if $S_0$ and $S_1$ are concordant via a concordance representing $\alpha$ we have $\km(S_0,S_1;\alpha)=0$.

While the set $\Self(S_0)$ is unwieldy, $\Delta(\Self(S_0))$ can often be computed in practice.  For example, when $X$ is a 2-handlebody, then we have $H_3(X\times I,X\times\{0,1\}; \mathbb{Z}\pi_1X) = 0$ -- that is, every element of $\Self(S_0)$ is homologous to $S_0\times I$, so $\Delta(\Self(S_0))=0$. Then $\km(S_0,S_1)$ may be computed using any singular concordance.

More generally, %as long as $H_2(X\times I,X\times\{0,1\};\mathbb{Z})$ is torsion-free we have from universal coefficients that \[H_3(X\times I,X\times\{0,1\};\mathbb{Z}\pi_1 X)=H_3(X\times I,X\times\{0,1\};\mathbb{Z})\otimes \mathbb{Z}\pi_1 X.\]
if we can find a %It typically is not too difficult to consider one
representative in $\Self(S_0)$ from each relevant $H_3(X\times I,X\times\{0,1\};\mathbb{Z})$ class (which is not too daunting when this group is $0$ or $\mathbb{Z}$, for example) then 
and modify it to represent different $H_3(X\times I,X\times\{0,1\};\mathbb{Z}\pi_1)$ classes and thus compute $\Delta(\Self(S_0))$. We discuss this again in Example~\ref{km_example}.

%More generally, the following property will be needed (see~\cite{stong} and~\cite{km2}):

%\begin{remark} \label{rmk:delta_homology_inv}
%if $H$ and $H'$ are singular concordances with $[H]=[H']$ in $H_3(X \times I, X \times \{0,1\}; \mathbb{Z}\pi_1 X)$ then $\Delta(H)=\Delta(H')$.  In particular, this is true for $H,H' \in \Self(S_0)$ which helps with the computation of $\Delta(\Self(S_0))$ -- see~\ref{km_example} for an example with non-trivial $\km$.  
%\end{remark}

In Section~\ref{sec:hopf}, we will see that when $H(S^2\times I)$ has a dual sphere, then in fact we can modify $H$ (via the moves in Section~\ref{sec:moves} and thus preserving $[H]$) and arrange for $L$ to be a Hopf link consisting of a type I pair $A,A'$. In this case, $\km(S_0,S_1;\alpha)=\epsilon_a$, and when $\epsilon_a=1$ then it is possible to replace $H$ with an embedding. %Then $\km(S_0,S_1;\alpha)$ measures our failure to remove this final self-intersection circle and obtain a concordance.

\subsection{Problems with positive genus} \label{sec:problems}
In this section we discuss the issue of extending the invariants $\fq$ and $\km$ to positive genus surfaces and raise a question regarding the definition of $\km$.  

In~\cite{km}, we discussed the case what we called ``$\pi_1$-negligible," homotopic, orientable genus-$g$ surfaces $F_0$ and $F_1$ embedded in a 4-manifold $X$. This was an abuse of terminology; in the setting of this paper we would say $\pi_1$-trivial rather than $\pi_1$-negligible. Unfortunately, there is an error in the proof of Lemma 6.1 and Proposition 6.2 in~\cite{km}, drawn to our attention by Mark Powell.
The strategy of the proof of~\cite[Lemma 6.1]{km} (which~\cite[Proposition 6.2]{km} relies on) was to compress an immersion $H:\Sigma_g\times I\to X\times I\times\mathbb{R}$ with $H(\Sigma_g\times0)=H(\Sigma_g\times 1)$ along a 3-dimensional 2-handle to obtain an immersion $H':\Sigma_{g-1}\times I\to X\times I\times\mathbb{R}$ with $H'(\Sigma_{g-1}\times0)=H(\Sigma_{g-1}\times 1)$, eventually reducing to the understood setting of an immersion of $S^2\times I$. Unfortunately, we can only ensure that $H(\Sigma_g\times0)=H(\Sigma_g\times 1)$ as submanifolds, not that $H|_{\Sigma_g\times 0}=H|_{\Sigma_g\times 1}$ as a map, so there is no guarantee that our compression will take place along the same circles in $\Sigma_g\times 0$ and $\Sigma_g\times 1$. Even if these circles are the same, there is again no guarantee that the compressing disks themselves will agree, so again we cannot ensure that $H'(\Sigma_{g-1}\times0)=H(\Sigma_{g-1}\times 1)$.

%, the disks that we are compressing along cannot necessarily be chosen so as to be the same in $X \times \{0\}$ and $X \times \{1\}$.  Therefore, the lower genus surfaces that we obtain in $X \times \{0\}$ and $X \times \{1\}$ after the compression will not necessarily be the same.  

%Even with a definition of $\fq$ for higher-genus surfaces, it is not clear at all if $\km$ could be defined in this case.  We leave then as a question:

\begin{question}
Can the Freedman--Quinn invariant be defined for (some) arbitrary genus based-homotopic surfaces?  Similarly, can the Stong obstruction be extended to include suitable higher-genus surfaces?
\end{question}

The Stong invariant seems even more difficult to extend to positive-genus surfaces due to its constrained definition. The Freedman--Quinn invariant has an alternative formulation in terms of immersions into the 6-manifold $X\times I\times\mathbb{R}$ (see~\cite{schneiderman_teichner}) and lifts of immersions to the universal cover of $X\times I$ (see~\cite{km}). It seems possible that this flexibility could be useful in an attempt to study positive-genus surfaces.  We are motivated to ask the following question.

\begin{question}
Is there a 6-dimensional definition of $\km$? If so, does this make it easier to prove that $\km$ is well defined $($as compared to~\cite{stong}$)$?
\end{question}

\section{Reducing the singular link when \texorpdfstring{$H(S^2\times I)$}{the image of H} has a dual sphere} \label{sec:reduce_circles}

In this section, we will use the moves presented in Section~\ref{sec:moves} to simplify $H$ when $H(S^2\times I)$ admits a dual sphere, eventually proving Theorem~\ref{thm:main1}. As in Section~\ref{sec:fqkm}, we only consider immersions from $S^2\times I$ to $X\times I$, rather than a general 3-manifold into a general 5-manifold (as was the case in Section~\ref{sec:moves}). The arguments in this section follow those of Stong (see in particular Figures 8--16 of~\cite{stong} and compare with our Figures~\ref{fig:cutoffcircles}--\ref{fig:removetwohopfs}). %However, to save on notation we may sometimes write $Y:=S^2\times I$. %restrict the 3-manifold $Y$ being $S^2 \times I$ and the ambient 5-manifold $W$ to a being a product $X\times I$ for some orientable 4-manifold $X$.  Let $H : S^2 \times I \to X \times I$ be an immersion with $H(S^2 \times \{0\}) = S_0$ and $H(S^2 \times \{0\}) = S_1$ for embedded spheres $S_0, S_1 \subset X$.  Further, we assume that $S_0$ has a dual sphere $G$.   

%In the case where $\fq(S_0, S_1) = 0$, we are able to choose $H$ so that there are no type II circles in $L$. In Section~\ref{sec:framed}, we show that when additionally $H(S^2\times I)$ has a dual sphere and is not $s$-characteristic, we can completely remove all singular circles and obtain a concordance in $X^2\times I$ between $S_0$ and $S_1$.  In Section~\ref{sec:hopf}, under the assumption that $\fq(S_0, S_1) = 0$ and $H(S^2\times I)$ has a dual sphere and is $s$-characteristic,  we see that we can modify $H$ so that $L$ is a Hopf link.

\subsection{Eliminating all intersections when \texorpdfstring{$H(S^2\times I)$}{the image of H} is not \texorpdfstring{$s$}{s}-characteristic}\label{sec:framed}

The key difference between the situations in Theorem~\ref{thm:main1} and Theorem~\ref{thm:main2} is that when $H(S^2 \times  I)$ has a dual and is not $s$-characteristic, we can find a dual sphere with trivial normal bundle.

\begin{lemma}\label{lem:notschar}
Let $H:Y^3\to W^5$ be an immersion so that there is a 2-sphere $G$ embedded in $W$ intersecting $H(Y)$ transversely once. Suppose $H(Y)$ is not $s$-characteristic. Then there is a 2-sphere $G'$ embedded in $W$ with trivial normal bundle intersecting $H(Y)$ transversely once.
\end{lemma}
\begin{proof}
If $G$ has trivial normal bundle, then the lemma is obviously trivial. So assume $G$ has nontrivial normal bundle.

Since $H(Y)$ is not $s$-characteristic, there is some 2-sphere $R$ embedded in $X\times I$ with $R\cdot H(Y)\not\equiv [R]\cdot [R]\pmod{2}$. Let $n$ be the geometric intersection number of $R$ and $H(Y)$. % Push $R$ into $X\times I$. Letting $n=\#(R\cap H(S^2\times I))$,
Let $G_1,\ldots, G_{n-1}$ be parallel copies of $G$. Choose disjoint arcs in $H(Y)$ connecting the intersection points of each $G_i$ with distinct points in $R\cap H(Y)$. % but one of the inters with $R$ to distinct points of intersection between $H(S^2\times I)$ and $G_i$.
Use these arcs to surger $R\sqcup G_1\sqcup\cdots\sqcup G_{n-1}$ to obtain one 2-sphere $G'$ that intersects $H(Y)$ transversely once.

We have 
\begin{align*}
    [G']\cdot[G'] &= ((n-1)[G] +  [R])\cdot((n-1)[G] +  [R]) \\
           &= (n-1)^2[G]\cdot[G] + [R]\cdot[R]\pmod{2} \\
           &= n - 1 + [R]\cdot[R]\pmod{2} \\
           &= 0 \pmod{2},
\end{align*}
since $[R]\cdot[R] \not\equiv n\pmod{2}$ by assumption.%  Thus by redefining $G:=G'$ if needed, we may assume that $G$ has trivial normal bundle.
\end{proof}

Now we can prove Theorem~\ref{thm:main1}, which we restate here for convenience.
%This allows us to modify $H$ via the moves of Section~\ref{sec:moves} to obtain an embedding.

\begin{mainthm1}
Suppose that $S_0$ and $S_1$ are embedded, oriented, homotopic 2-spheres in an orientable 4-manifold $X$ such that $S_0$ has an immersed dual sphere $G$ in $X$ (i.e., $G$ and $S_0$ intersect in a single point). Assume that $S_0$ is not $s$-characteristic. Then $S_0$ and $S_1$ are concordant if and only if $\fq(S_0,S_1)=0$.
\end{mainthm1}

\begin{proof}
If $S_0$ and $S_1$ are concordant, then $H$ can be chosen to be an embedding. Then $\fq(S_0,S_1)=\mu(H)=0$. Thus, the ``only if'' portion of the statement is trivial.

Now suppose $\fq(S_0,S_1)=0$. By Proposition~\ref{fq_typeII} we can choose a singular concordance between $S_0$ and $S_1$ whose singular link $L$ consists of only type I circles. Push the 2-sphere $G$ into $X\times I$, so that it is an embedded 2-sphere intersecting the image of $H$ transversely once. By Lemma~\ref{lem:notschar}, take $G$ to have trivial normal bundle. 

We use the ambient surgery technique of Section~\ref{sec:ambient_dehn} to remove dual pairs of type I singular circles. To do this, choose immersed disks bounding the singular circles in $S^2 \times I$ that only have clasp intersections as in Lemma~\ref{lem:clasp}.  Suppose that $A$ and $B$ are two singular circles (that are potentially dual or not distinct) such that the disks that they bound have a clasp intersection. In Figure~\ref{fig:frameddual} we illustrate the proceeding set of moves.

\begin{figure}
    \centering
    \labellist
    \small\hair 2pt
    \pinlabel{\textcolor{red}{$A'$}} at 0 55
    \pinlabel{\textcolor{red}{$A$}} at 13 55
    \pinlabel{\textcolor{blue}{$B$}} at 27 55
    \pinlabel{\textcolor{blue}{$B'$}} at 40 55
    
     \pinlabel{\textcolor{red}{$A'$}} at 82 55
    \pinlabel{\textcolor{red}{$A$}} at 95 55
    \pinlabel{\textcolor{blue}{$B$}} at 109 55
    \pinlabel{\textcolor{blue}{$B'$}} at 122 55
    \pinlabel{\textcolor{green}{$E$}} at 76 75
    \pinlabel{\textcolor{green}{$E'$}} at 129 75
    
       \pinlabel{\textcolor{red}{$A'$}} at 164 55
    \pinlabel{\textcolor{red}{$A$}} at 177 55
    \pinlabel{\textcolor{blue}{$B$}} at 191 55
    \pinlabel{\textcolor{blue}{$B'$}} at 204 55
    \pinlabel{\textcolor{green}{$E$}} at 155 72
    \pinlabel{\textcolor{green}{$E'$}} at 210 73
      \pinlabel{\textcolor{brown}{$F'$}} at 172 73
    \pinlabel{\textcolor{brown}{$F$}} at 222 82
    
          \pinlabel{\textcolor{red}{$A'$}} at 122 -6
    \pinlabel{\textcolor{red}{$A$}} at 135 -6
    \pinlabel{\textcolor{blue}{$B$}} at 149 -6
    \pinlabel{\textcolor{blue}{$B'$}} at 162 -6
      \pinlabel{\textcolor{brown}{$F'$}} at 130 12
    \pinlabel{\textcolor{brown}{$F$}} at 170 11
    
          \pinlabel{\textcolor{red}{$A'$}} at 40 -6
    \pinlabel{\textcolor{red}{$A$}} at 53 -6
    \pinlabel{\textcolor{blue}{$B$}} at 67 -6
    \pinlabel{\textcolor{blue}{$B'$}} at 80 -6
    
    \endlabellist
    \includegraphics[width=120mm]{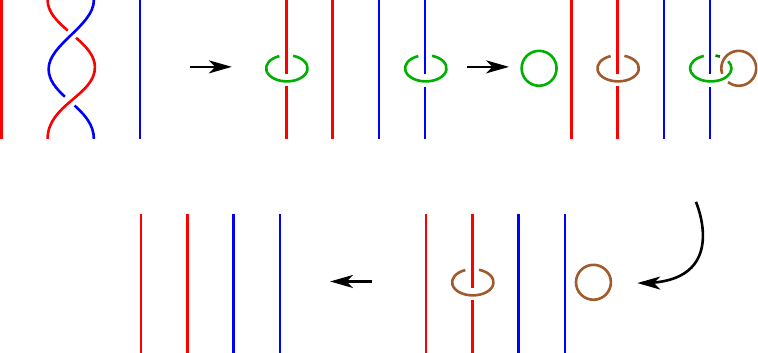}
    \vspace{.1in}
    \caption{Eliminating a clasp from $L$ when $H(S^2 \times I)$ has a dual sphere with trivial normal bundle.}
    \label{fig:frameddual}
\end{figure}
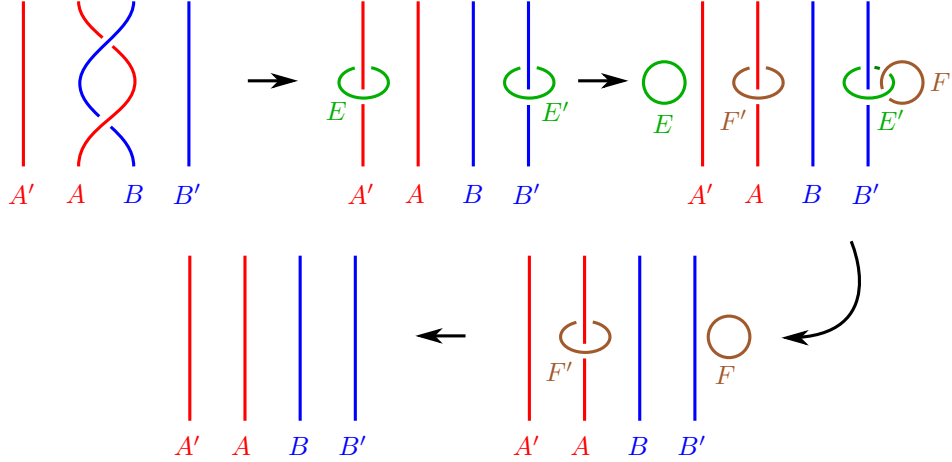

% as in \red{[Figure - this proof references one fun sequence of links - this reference here is to the first part of that picture]}. 
We perform a finger move as in Section~\ref{sec:claspmove} to remove the clasp at the cost of creating one pair of dual circles (see Section~\ref{sec:claspmove}) $E$ and $E'$, with $E$ a meridian of $A'$ (the dual of $A$) and $E'$ a meridian of $B'$ (the dual of $B$). The circles $E$ and $E'$ bound meridian disks that clasp those bounded by $A'$ and $B'$. We perform another finger move to remove the clasp between $A'$ and $E$, pulling $E$ off $A'$ entirely. This introduces yet another pair of dual circles $F$ and $F'$ (again, see Section~\ref{sec:claspmove}) with $F$ a meridian of $E'$ and $F'$ a meridian of $A$.
 
 Now because $G$ is a framed dual to $H(S^2 \times I)$, we may use the ambient Dehn surgery technique of Section~\ref{sec:ambient_dehn} (refer to Remark~\ref{ambientdehnremark}) to alter $H$ and remove first $E$ and $E'$, then $F$ and $F'$, completely from $L$ without otherwise altering $L$.
 
 %We now combine the previous moves to see that, in this situation, we may modify $H$ to remove this clasp intersection of the disks while otherwise preserving the singular circles.  This sequence of moves is shown in \red{[Figure for this section - I'll send it to you]}.  

The total result is that we have removed a clasp from the disks bounded by $A$ and $B$ without otherwise altering the singular link $L$. By repeating this at every clasp, we can thus arrange for $L$ to be an unlink. We may again use the ambient Dehn surgery technique of Section~\ref{sec:ambient_dehn} to remove each pair of singular circles. The resulting map from $S^2 \times I$ to $X\times I$ is thus an embedding, so $S_0$ and $S_1$ are concordant.
%By performing this move on all of the clasps we obtain a new map which we again call $H$ whose singular circles are an unlink of type I circles.  Then by doing the reverse of the finger move in subSection~\ref{sec:finger_around_an_element}, we can remove all of these circles and obtain an embedding.  
%Now assume that $w_2(G) = 1$, we will reduce this case to the previous case by finding a sphere $G_1$ that intersects $H(S^2 \times I)$ in a single point with $w_2(G_1) = 0$ and then continuing as in the previous case.   
%To find this sphere $G_1$, note that since $S_0$ is not $s$-characteristic, there exists a sphere $S \subset X$ with either:
%\begin{enumerate}
 %   \item $S_0 \cdot S$ even and $S \cdot S $ odd, or
  %  \item $S_0 \cdot S$ odd and $S \cdot S $ even.
%\end{enumerate}
%Either way, we can tube the intersections of $S$ with $S_0$ into the dual sphere $G$ until there is just one point of intersection left.  This produces a sphere $G'$ that intersects $S_0$ in a single point with $w_2(G') = 0$.  Pushing this sphere into $X \times I$ so that it becomes embedded and continuing as in the case where $w_2(G) = 0$ finishes the argument.    
\end{proof}

\subsection{Reduction to a Hopf link of singular circles} \label{sec:hopf}

In this section, we begin the proof of of Theorem~\ref{thm:main2}. We now take $S_0$ and $S_1$ to be $s$-characteristic, so any dual sphere to $H(S^2\times I)$ must have nontrivial normal bundle.  We will see how in this situation, we can still modify $H$ to reduce the number of components in its singular link $L$, eventually obtaining a type I dual pair forming a Hopf link.  In the next section, we will show that when the element of $\pi_1 X$ corresponding to this last pair represents the trivial element of $H_1(X; \mathbb{Z}/2\mathbb{Z})$, then we can further remove this Hopf link to obtain a concordance from $S_0$ to $S_1$.

\begin{proposition}\label{prop:LHopflink}
Let $S_0$, $S_1$ be homotopic 2-spheres in an orientable 4-manifold $X$ with $\fq(S_0,S_1)=0$. Let $H'$ be a singular concordance from $S_0$ to $S_1$ with $\mu(H)=0$. Suppose $S_0$ is $s$-characteristic and there exists an immersed 2-sphere in $G$ intersecting $S_0$ transversely in a single point. Then there is a singular concordance $H$ between $S_0$ and $S_1$ with $[H]=[H']\in H_3(X\times I, X\times\{0,1\};\mathbb{Z}\pi_1X)$ such that the singular link $L$ of $H$ is a type I dual pair forming a Hopf link.
%\red{Furthermore, such a map $H$ exists for any homology class in $H_3(X \times I, X \times \{0,1\});\mathbb{Z})$ that can be represented by any map $S^2 \times I\to X\times I$  mapping boundary to $S_0\times\{0\}\sqcup S_1\times \{1\}$.  }
\end{proposition}

\begin{proof}
As in the proof of Theorem~\ref{thm:main1} (contained in Section~\ref{sec:framed}), we use Proposition~\ref{fq_typeII} to choose an initial $H$ whose singular link $L$ consists of only type I circles, and we push $G$ into the interior of $X\times I$ so that it is embedded and intersects $H(S^2 \times I)$ transversely once. Since $H(S^2 \times I)$ is $s$-characteristic, $G$ must have nontrivial normal bundle.

%ince $\fq(S_0,S_1) = 0$, by Proposition~\ref{fq_typeII} we have an immersion $H : S^2 \times I \to X \times I$ with $H(S^2 \times \{0\}) = S_0 \times \{0\}$ and $H(S^2 \times \{1\}) = S_1 \times \{1\}$ such that $H$ has no type II singular circles.  We will modify this $H$ in stages to obtain the desired Hopf link of self-intersections.  

Via Lemma~\ref{lem:clasp}, we choose a collection $D$ of immersed disks in $S^2 \times I$ bounded by each component of $L$ %. % in. of the singular link $D$ bounded by each co the sin $L$ of $H$ so 
with the property that the intersections of $D$ are all clasp intersections.  As in subSection~\ref{sec:framed}, these clasps suggest finger moves.  We modify $H$ and then refer to the result again as $H$ throughout.  This modification will proceed in a number of steps; we summarize them here and will then give details about how to achieve each.%:

\begin{enumerate}
    \item\label{hopfstep1} %Suppose there is a clasp intersection between the disks bounded by $A$ and $B$ (which may be non-distinct). Perform a clasp move to $H$ to remove this clasp at the cost of introducing a new type I pair that are meridians to $A'$ and $B'$. Repeat at every clasp. The
    We peform clasp finger moves to $H$ until the singular link $L$ consists of an unlink of type I dual pairs, %
   % Remove the clasps amongst the disks (at the expense of introducing small type I singular circles that are meridians to the original singular circles) so that we obtain a new $H$ whose singular circles consist of an unlink $U$ of type I circles that is closed under taking duals, 
   together with a collection of small linking meridians to the unlink that are also all type I.  
    
    \item\label{hopfstep2} %Let $A,A'$ be a dual pair in $U$. %There may be any number of meridians of $A$ and $A'$ in $M$, although by 
  %  by {\blue{PARITY LEMMA}} the number of meridians of $A$ in $M$ has the same parity as the number of meridians of $A'$. If both have no meridians in $M$, then perform ambient surgery to $H$ to remove $A$ and $A'$. Now assume without loss of generality that $A$ has a meridian in $M$, and has at least as many as $A'$.
   % If there is more than one meridian of $A$ in $M$, consider the bands $b_1, b_2$ as in Figure {\blue{FIGURE}}. One can easily check that they satisfy the conditions of Remarks {\blue{CITE REMARKS}}, i.e., that $H(b_1)\cup H(b_2)$ is an annulus and the bundles $l_1$ and $l_2$ are orientable. Then by \blue{{CLEAN DISK SECTION and OMEGA2 SECTION}}, either $b_1, b_2$ are framed for some clean Whitney disk, or they would be if a full twist were added to $b_1$. Lemma {\blue{PARITY LEMMA}}
   Through a sequence of finger moves, ambient surgeries, and Whitney moves, we modify $H$ so that $L$ is a split union of Hopf links in disjoint 3-balls in $S^2 \times I$, with all components type I. %singular link whose singular link consists only of Hopf links that are contained in disjoint 3-balls, where all of the singular circles are of type I.
   We call these Hopf links $h_1,....h_n$ and note that they are partitioned into disjoint cycles by considering the equivalence classes generated by the relation, ``$h_i\sim h_j$ if one of the components of $h_i$ is dual to one of the components of $h_j$". 
    
    \item\label{hopfstep3} We shorten the cycles of Hopf links, %resulting from the previous move to obtain
    obtaining a new $H$ whose singular link is a split union % are all of type I and are all
    of Hopf links %in disjoint 3-balls such that each component is dual to the component that it links.
    such that each Hopf link consists of one type I dual pair.
      \item\label{hopfstep4} Finally, we modify $H$ to remove all but one of these Hopf links from $L$ -- or perhaps more accurately, we merge all of the Hopf links in $L$ to form a single Hopf link. % Combine two Hopf links that result from the previous step.  Doing this multiple times allows us to combine all of the Hopf links to obtain a new $H$ with a single Hopf link of type I singular circles, as desired.
\end{enumerate}

Now we give more details on how to perform each of the above steps. %Furthermore, we keep track of the fundamental group elements corresponding to the Hopf links change in Step~\ref{hopfstep4} (they multiply together), as this will be needed later.  

\subsubsection{Step \eqref{hopfstep1}}

We perform a finger move (as in Section~\ref{sec:claspmove}) on each clasp of $D$. The result is to transform the components of $L$ into an unlink, at the cost of adding many pairs of dual type I circles that are meridians of the original components. We call the original components of $L$ ``long circles" and the new meridian circles ``short circles," so now $L$ consists of an unlink of long circles and an unlink of short circles, so that every short circle is a meridian of a long circle. Every component of $L$ is type I, and the dual of any long circle is long while the dual of any short circle is short.%By construction, the set of long circles is closed under taking duals.  
%By Lemma~\ref{lem:clasp} we can assume that the components of $L$ each bound disks $D_i$ such that the intersections of the disks are all clasp intersections.  Then using the finger move in Section~\ref{sec:claspmove}, on all of these clasps, the result is achieved.  

\subsubsection{Step \eqref{hopfstep2}}
Let $A,A'$ be a dual pair of long circles. If there are no short circles linking $A$ and $A'$, then perform ambient surgery to remove both circles. Thus, without loss of generality, assume that there is at least one short circle linking $A$, and that there are at least as many short circles linking $A$ as there are linking $A'$ (by potentially switching the roles of $A$ and $A'$).

%The desired result will be achieved by considering each pair of long circles in turn  If the pair of long circles does not have any linking short circles, they we may simply remove it using the ambient Dehn surgery move.  

%We will use Whitney moves along pre-specified bands as in section 2.1 to change $H$ and $L$.  Let $A$ and $A'$ be two dual type I long circles in the unlink $U$ of large circles from step (1) such that there are small circles linking $A$ or $A'$.  By finger moves and ambient 2-handle surgery (see Figure~\ref{fig:movecircle},
In Figure~\ref{fig:movecircle}, we show how we may perform a finger move and then ambient surgery to decrease the number of short circles linking $A$ by one while increasing the number of short circles linking $A'$ by one. Recall from Lemma~\ref{lem:stong1} that the number of short circles linking $A$ has the same parity as the number of short circles linking $A'$. Therefore, by repeatedly performing this move we may arrange that $A$ and $A'$ link equal numbers of short circles.

%may move the small circles linking $A$ to $A'$ (and vice versa) and thereby, using Lemma~\ref{lem:stong1}, assume that $A$ and $A'$ have an equal number of short circles linking them.  

\begin{figure}
    \centering
    \labellist
    \small\hair 2pt
    \pinlabel {$A$} at 7 130
    \pinlabel {$A'$} at 77 130
    \pinlabel {\textcolor{red}{$E$}} at 15 175
    \pinlabel {\textcolor{red}{$E'$}} at 124 155
       \pinlabel {$A$} at 7 65
    \pinlabel {$A'$} at 77 65
    \pinlabel {\textcolor{red}{$E$}} at 11.5 110
    \pinlabel {\textcolor{red}{$E'$}} at 124 90
    \pinlabel {\textcolor{blue}{$F'$}} at 165 109.5
    \pinlabel {\textcolor{blue}{$F$}} at 84.5 110
     \pinlabel {$A$} at 7 0
    \pinlabel {$A'$} at 77 0
    \pinlabel {\textcolor{blue}{$F'$}} at 123.5 25
    \pinlabel {\textcolor{blue}{$F$}} at 84.5 45
    \endlabellist
    \includegraphics[width=70mm]{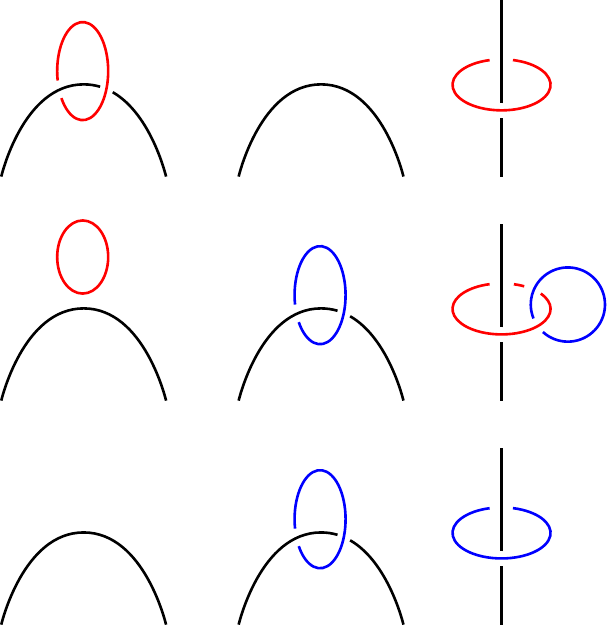}
    \caption{From top to bottom, we remove from the singular link $L$ one short circle $E$ that is a meridian of $A$ while adding another $F$ that is a meridian of $A'$. In other words, we ``move" one short circle from $A$ to $A'$. Top to middle: finger move to unclasp $E$ from $A$. Middle to bottom: ambient surgery removing $E$ and $E'$.}
    \label{fig:movecircle}
\end{figure}
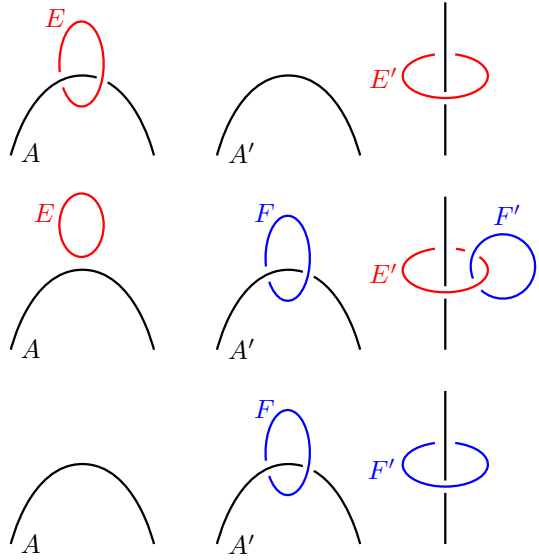

Pick points $x,y \in A$ and $x',y' \in A'$ so that $H(x) = H(x')$ and $H(y)= H(y')$.  Furthermore, pick arcs $\alpha$ from $x$ to $y$ and $\alpha'$ from $x'$ to $y'$ that are contained in the disks that demonstrate that the link of long circles is an unlink, so that the arcs cut off one short circle from each of $A$,$A'$, centered about points with the same image under $H$.

%n one side an even total number of short circles linking $A$ and $A'$ as in Figure~\ref{fig:cutoffcircles}. -- in particular, take the arcs $\alpha$ and $\alpha'$ so that they each cuts off a single small circle.  Then b
By Section~\ref{sec:cleanwhitney}, we know that $\alpha,\alpha'$ can be framed to yield bands that are framed for some clean Whitney disk. consider the %exists a Whitney disk $\Delta$ with boundary $H(\alpha \cup \alpha'
band-thickenings $b$ and $b'$ of $\alpha$ and $\alpha'$ (respectively) as in Figure~\ref{fig:cutoffcircles}. These thickening are chosen to lie inside disjoint disks $D_1$ and $D_2$ bounded by $A$ and $A'$. It is easy to verify that the conditions of Remarks~\ref{rem:checkannulus} and~\ref{rem:l1l2} are satisfied by $b,b'$, i.e., that $H(b\cup b')$ is an annulus and the bundles $l_1,l_2$ are orientable. Then by Section~\ref{sec:w2obstruction}, either $b,b'$ are framed, or they would be framed if a whole twist were added to $b$. But if we were to add a twist to $b$ and perform band surgery, we would obtain a singular link violating the conclusion of Lemma~\ref{lem:stong1}, so we conclude such a Whitney move is not possible. Thus, $b,b'$ are framed for some clean Whitney disk. We perform the Whitney move, splitting $A$ into two components $A,B$ and $A'$ into two components $A',B'$, all of which we consider long circles. The circles $B$ and $B'$ each link one short circle, forming two Hopf links, and we have decreased the number of short circles linking $A$ and $A'$ by one. We repeat until each of $A$ and $A'$ link only one short circle, forming Hopf links.

%e must argue that it is possible to perform a Whitney move along $\Delta$ using the bands $b$ and $b'$ (recall from Section~\ref{sec:whitney} that there is a $\mathbb{Z}/2\oplus\mathbb{Z}/4$ obstruction).  Upon doing this repeatedly for each pair of long singular circles, we complete this step.   

\begin{figure}
    \centering
    \includegraphics[width=70mm]{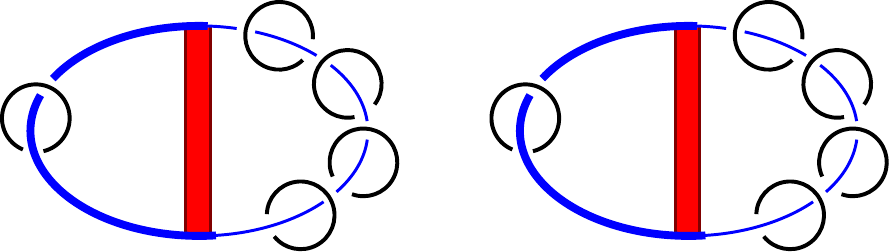}
    \caption{A dual pair $A$ and $A'$ of long circles in $L$. The bold arcs in $A$ and $A'$ have the same image under $H$ and each link one short circle. Bands contained in disks for $A$ and $A'$ disjoint from the other long circles give a framed Whitney disk exactly when the resulting pairs of long circles would satisfy Lemma~\ref{lem:stong1}.}
    \label{fig:cutoffcircles}
\end{figure}
\subsubsection{Step \eqref{hopfstep3}}
At the end of Step \eqref{hopfstep2}, the singular link $L$ consists of many Hopf links of type I circles. A {\emph{cycle}} in $L$ consists of $m$ Hopf links whose components are $(A_1,A_2')$, $(A_2,A_3')$, $\ldots$, $(A_m,A_1')$ where $A_i,A_i'$ are a dual pair. We say that such a cycle has length $m$. 

We shorten a length $m>1$ cycle as follows -- see Figure~\ref{fig:cyclereduction}. Perform a finger move to unclasp $A_1$ from $A_2'$. (If $n=2$, then $A_2=A_m'$.) This has the cost of adding two singular circles $B,B'$ that are respectively meridians of $A_2$ and $A_1'$. We perform ambient Dehn surgery to remove the pair $A_1,A_1'$, and another ambient Dehn surgery to remove the pair $A_2,A_2'$. Because $G$ has nontrivial normal bundle, in addition to removing $A_1,A_1',A_2,A_2'$, these ambient surgeries have the effect of adding a whole twist to the remaining components of $L$ that link $A_1'$ and $A_2$. Then $A_m\cup B'$ and $B\cup A_2'$ are Hopf links. Thus, $L$ now includes Hopf links of components $(B,A_3'),(A_3,A_4'),\ldots,(A_{m-1},A_m'),(A_m,B')$. This is a length $(m-1)$ cycle. By repeating this argument, we can arrange for $L$ to only consist of length 1 cycles, i.e., Hopf links each consisting of one dual pair. Moreover, as we indicate in Figure~\ref{fig:cyclereduction}, the group element associated to $B$ is $g_2g_1$, where $g_i$ is the group element associated to $A_i$. Then inductively, the group element associated to the length one cycle obtained from $(A_1,A_2')$, $(A_2,A_3')$, $\ldots$, $(A_m,A_1')$ is $g_mg_{m-1}\cdots g_3g_2g_1$.

%That is, $L$ is a split union of Hopf links in which every Hopf link is a dual pair of type I circles.  The result of each cycle shortening is a single Hopf link of dual circles whsoe corresponding fundamental group element is the product of the group elements seen by going around the cycle.  

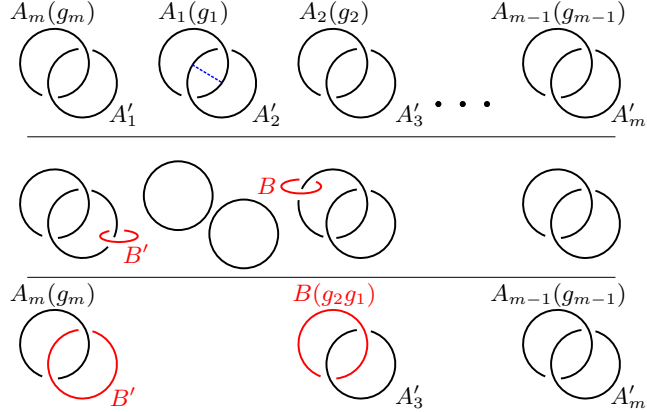
\begin{figure}
    \centering
    \labellist
    \small\hair 2pt
    \pinlabel{\footnotesize{$A_m (g_m)$}} at 13 148.5
    \pinlabel{\footnotesize{$A_1'$}} at 40 109.5
    \pinlabel{\footnotesize{$A_1 (g_1)$}} at 67 148.5
    \pinlabel{\footnotesize{$A_2'$}} at 95 109.5
    \pinlabel{\footnotesize{$A_2 (g_2)$}} at 121 148.5
    \pinlabel{\footnotesize{$A_3'$}} at 150 109.5
    \pinlabel{\huge{$\cdots$}} at 172 113
    \pinlabel{\footnotesize{$A_{m-1} (g_{m-1})$}} at 207 148.5
    \pinlabel{\footnotesize{$A_m'$}} at 234 109.5
    
    \pinlabel{\footnotesize{\textcolor{red}{$B$}}} at 95 82
    \pinlabel{\footnotesize{\textcolor{red}{$B'$}}} at 45 56
    
     \pinlabel{\footnotesize{$A_m (g_m)$}} at 13 40
    \pinlabel{\footnotesize{\textcolor{red}{$B'$}}} at 40 1
    \pinlabel{\footnotesize{\textcolor{red}{$B (g_2g_1)$}}} at 121 40
    \pinlabel{\footnotesize{$A_3'$}} at 150 1
    \pinlabel{\huge{$\cdots$}} at 172 113
    \pinlabel{\footnotesize{$A_{m-1} (g_{m-1})$}} at 207 40
    \pinlabel{\footnotesize{$A_m'$}} at 234 1
    \endlabellist
    \vspace{.05in}
    \includegraphics[width=80mm]{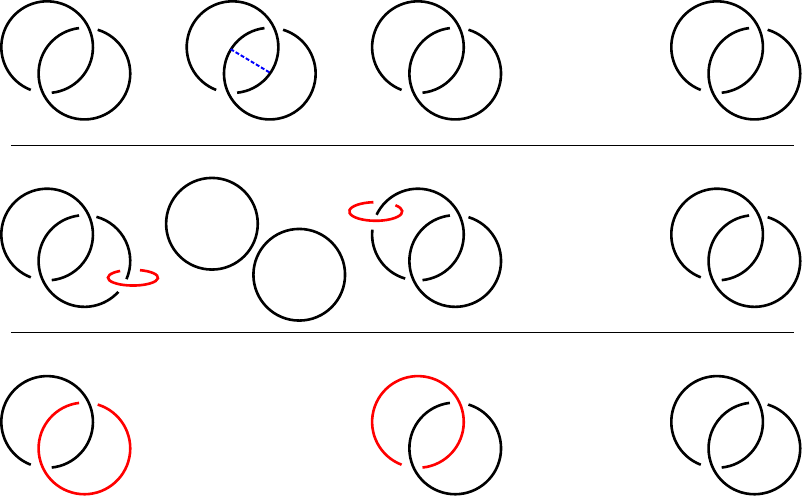}
    \caption{Given a cycle of length $m>1$, we replace it with a cycle of length $m-1$. Top row: the length $m$ cycle. From top to bottom, we perform a finger move to unclasp $A_1$ and $A_2'$, then we perform ambient Dehn surgery to remove $A_1,A_1'$ and $A_2,A_2'$. In parentheses next to each active circle, we indicate the associated fundamental group element.}
    \label{fig:cyclereduction}
\end{figure}

%This is achieved using the finger move from Section~\ref{sec:claspmove}, together with the surgery move in Section~\ref{sec:ambient_dehn}.  Note that in using the ambient Dehn Surgery move in~\ref{sec:ambient_dehn}, we have $w_2(G) = 1$ and therefore we will always have to add full twists as mentioned at the end of Section~\ref{sec:ambient_dehn}.  By putting together these two moves as in \red{Figure ii}, we obtain the desired cycle shortening.  See \red{Figure iii} for an example of a shortening of a length 3 cycle.  

\subsubsection{Step \eqref{hopfstep4}} 
%\begin{figure}
%\labellist
%\small\hair 2pt
%\pinlabel{\textcolor{red}{$A (g_a)$}} [tl] at 20 362
%\pinlabel{\textcolor{red}{$A'$}} [tl] at 67 362
%\pinlabel{\textcolor{blue}{$B'$}} [tl] at 140 362
%\pinlabel{\textcolor{blue}{$B (g_b)$}} [tl] at 170 362
%\pinlabel{2$\times$ clasp} at 255 325
%\pinlabel{\textcolor{red}{$A$}} [tl] at 343 362
%\pinlabel{\textcolor{red}{$A'$}} [tl] at 382 362
%\pinlabel{\textcolor{blue}{$B'$}} [tl] at 452 362
%\pinlabel{\textcolor{blue}{$B$}} [tl] at 487 362
%\endlabellist
%    \centering
%    \vspace{.1in}
%    \includegraphics[width=130mm]{combinehopfs}
%    \caption{When $L$ includes two split Hopf links, each of which formed from a dual type I pair, we may homotope and surger $H$ to replace the two Hopf links with a length-6 cycle.{\blue{MM: didn't finish adding all the labels sorry}}}
%    \label{fig:combinehopfs}
%\end{figure}
\begin{figure}
\centering
   \scalebox{.9}{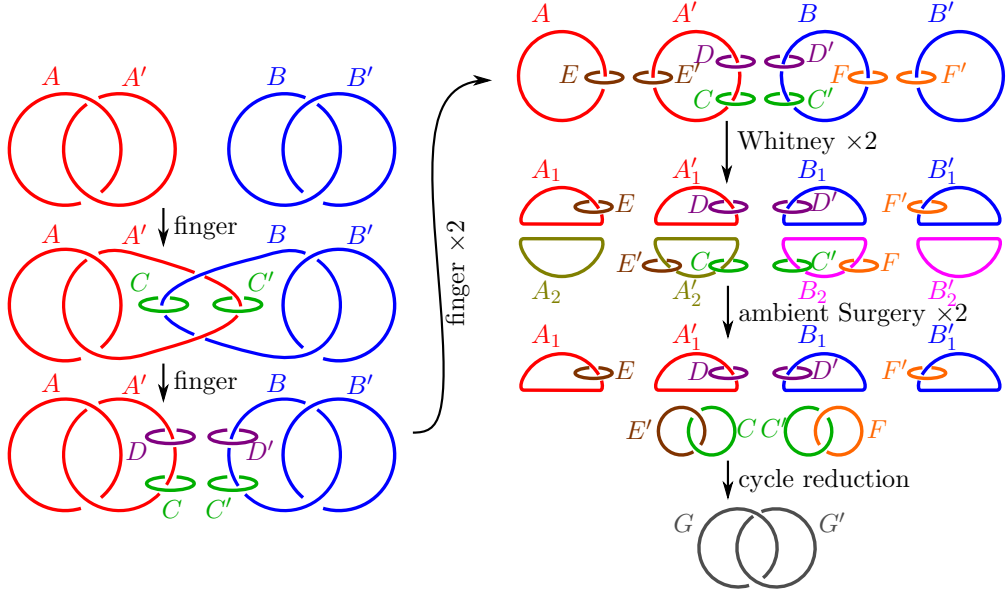}
    \caption{In the top left, we show two dual pairs of type I circles in $L$, each of which forms a Hopf link. Following the arrows, we perform moves to $H$ to replace these two Hopf links with a single Hopf link.}
    \label{fig:combinehopf}
\end{figure}
In Figure~\ref{fig:combinehopf}, we illustrate how to perform a sequence of finger moves, a Whitney move, ambient Dehn surgery, and the cycle shortening of step \eqref{hopfstep3} to replace two split Hopf links in $L$ with one Hopf link.

\begin{remark}\label{rem:combinehopf}
The Hopf link obtained in Figure~\ref{fig:combinehopf} may be taken to have associated group element $ba$, where $a$ and $b$ are the group elements associated to $A$ and $B$, respectively.  Suppose $A$ and $B$ are associated to elements $a,b\in\pi_1(X)$. Then $C$ and $D$ are each naturally associated to $c=ba$, but by exchanging the roles of $D$ and $D'$ the group element $d$ associated to $D$ is $a^{-1}b^{-1}$. (Recall that the group element associated to a type I self-intersection is defined only up to inverse due to the choice of active preimage.) The circles $E$ and $F$ are both associated to $e=f=1$. Then the final active circle $G$ is associated to the product $f\cdot e\cdot d\cdot c\cdot b\cdot a=1\cdot 1\cdot a^{-1}b^{-1}\cdot ba\cdot b\cdot a =ba$.
\end{remark}

%This is achieved through a sequence of finger moves, Whitney moves as in step (2), and finally cycle shortening as in Step (3). 
We repeat until $L$ consists of a single Hopf link, completing the proof of Proposition~\ref{prop:LHopflink}.
\end{proof}

\begin{remark}\label{rem:hopfpi1}
In the proof of Proposition~\ref{prop:LHopflink}, the element of $\pi_1 X$ associated to the final Hopf link is a product of all the group elements associated to the active singular circles at the end of Step \eqref{hopfstep2} (at which time $L$ is a split union of Hopf links). The terms of this product may be taken to be in any order, according to a choice of in which order we merge the various Hopf links together in Steps \eqref{hopfstep3} and \eqref{hopfstep4}.%in Figure~\ref{fig:combinehopf}.%That is, some word in $w$ in all the group elements associated to the long circles $A_i, A_i'$ at the end of Step \eqref{hopfstep1} with each element $g_i$ appearing in $w$ exactly half the total number of short circles linking the corresponding $A_i$ and $A_i'$. While discussion of group elements is not necessary to conclude Proposition~\ref{prop:LHopflink}, we will make use of this observation later in the Proof of Theorem~\ref{thm:main} part~\ref{part2}.
\end{remark}

\begin{remark}\label{rmk:reduce_to_hopf_no_homology_change}
Note that as in Remark~\ref{rmk:hypos_fq}, if $\mu(\pi_3 X) = 0$, then such a map $H$ as in Proposition~\ref{prop:LHopflink} exists in any homology class in $H_3(X \times I, X \times \{0,1\});\mathbb{Z}\pi_1X)$ that can be represented by a singular concordance from $S_0$ to $S_1$ (since we automatically have $\mu(H)=0$). %This is because, we can start with any such such singular concordance $H$, and under the hypotheses that either $\pi_1 X$ contains no 2-torsion, or $\mu(\pi_3 X) = 0$, as in Remark~\ref{rmk:hypos_fq}, the use of Proposition~\ref{fq_typeII} does not change the homology class $[H]$ and similarly by Remark~\ref{rmk:homology_unchanged}, the rest of the proof of Proposition~\ref{prop:LHopflink} also does not alter $[H]$.  
\end{remark}

\begin{lemma}\label{lem:addHopfnullhomologous}
Let $H$ be a singular concordance between the spheres $S_0$ and $S_1$ such that $H(S^2 \times I)$ has a dual sphere $G$. Let $g$ be any element of $\pi_1 X$ representing the trivial element of $H_1(X;\mathbb{Z}/2\mathbb{Z})$.  We can modify $H$ to add a Hopf link to to the link of singular circles $L$ of $H$, split from the rest of $L$, consisting of a dual type I pair with active circle associated to $g$.
\end{lemma}

In order to prove Lemma~\ref{lem:addHopfnullhomologous}, we will make use of the following group-theoretic lemma.

\begin{lemma} \label{lem:Z/2_abelianization}
Let $V$ be a path-connected topological space and $\mathcal{G}$ a finite set of generators for $\pi_1 V$.  The kernel of the map $\pi_1 V\to H_1(V; \mathbb{Z}/2\mathbb{Z})$ contains exactly the elements of $\pi_1 V$ that can be represented by a word $w$ in the alphabet $\mathcal{G} \cup \mathcal{G}^{-1}$ (where $\mathcal{G}^{-1}$ denotes the set of inverses of elements of $\mathcal{G}$) where the total number of times $a$ appears in $w$ is equal to the number of times $a^{-1}$ appears in $w$ modulo 2 for each $a\in\mathcal{G}$.% where the total number of times that an element $a$ together with $a^{-1}$ appear in $g$ is even, for every element of $A$.
\end{lemma}

\begin{proof}
%Let $h\in\Ker(\pi_1(X)\xrightarrow{\blue{MM name??}}H_1(X;\mathbb{Z}/2\mathbb{Z})$.
Since $H_1(V; \mathbb{Z}/2\mathbb{Z})$ is abelian with every nontrivial element 2-torsion, any word in $\mathcal{G}\cup\mathcal{G}^{-1}$ with the property that every letter and its inverse appear with the same parity must be contained in $\ker(\pi_1(V)\to  H_1(V;\mathbb{Z}/2\mathbb{Z}))$.%epresent the trivial elemen

Now suppose $h\in\ker(\pi_1(V)\to  H_1(V;\mathbb{Z}/2\mathbb{Z})$. 
%his kernel % and $x^2 = 0$ for all $x \in H_1(X; \mathbb{Z}/2\mathbb{Z})$, the kernel 
%contains all of the elements represented by words in the alphabet $\mathcal{G} \cup \mathcal{G}^$ where the total number of times an element $a$ together with $a^{-1}$ appear is even.  
%Conversely, suppose $g$ is in the kernel of the map $\pi_1 X \to H_1(X; \mathbb{Z}/2\mathbb{Z})$.
The coefficient exact sequence 
$$
0 \to \mathbb{Z} \xrightarrow{2} \mathbb{Z} \to \mathbb{Z}/2\mathbb{Z} \to 0
$$
gives rise to the Bockstein exact sequence which in particular shows that the kernel of the map $H_1(V; \mathbb{Z}) \to H_1(V; \mathbb{Z}/2\mathbb{Z})$ is equal to the image of the map $H_1(V; \mathbb{Z}) \xrightarrow{2} H_1(V; \mathbb{Z}/2\mathbb{Z})$.  Since $h$ is in the kernel of 
$\pi_1 V \to H_1(V; \mathbb{Z}) \to H_1(V; \mathbb{Z}/2\mathbb{Z})$, 
we have
\begin{align*}
    \pi_1 V &\to H_1(V; \mathbb{Z}) \\
          h &\mapsto 2x
\end{align*}
for some $x \in H_1(V; \mathbb{Z})$.  Let $h_0\in \pi_1(V)$ with $h_0$ abelianizing to $x$. Then $hh_0^{-2}$ is the kernel of the abelianization map.  Therefore, we can write $hh_0^{-2}$ as the product of commutators, each of which contains can be written as a word $w'$ in $\mathcal{G}\cup\mathcal{G}^{-1}$ where $a$ and $a^{-1}$ appear an equal number of times for each $a\in\mathcal{G}$. Let $w_0$ be a word in $\mathcal{G}\cup\mathcal{G}^{-1}$ representing $h_0$. Then $w:=w'w_0^2$ is the desired word representing $h$.%th a even total number of appearances of the element $a$ together with $a^{-1}$.  Since $g_0^2$ can also be written in the form, we obtain the desired result for $g$.  
\end{proof}

We now prove Lemma~\ref{lem:addHopfnullhomologous}.

\begin{proof}[Proof of Lemma~\ref{lem:addHopfnullhomologous}]
%et $g\in\pi_1(X)$ be an element abelianizing t
Let $g_1,\ldots,g_n$ be generators for $\pi_1 X$. 
By Lemma~\ref{lem:Z/2_abelianization}, there is a word $w$ in $\{g_i,g_i^{-1}|1\le i\le n\}$ representing $g$ such that for each $i$, $g_i$ and $g_i^{-1}$ appear in $w$ with the same parity. Let $n_i^+$ and $n_i^-$ be the number of times $g_i$ and $g_i^{-1}$ (respectively) appear in $w$, so $n_i^++n_i^-$ is even for each $i$.

\begin{figure}
    \centering
    \includegraphics[width=30mm]{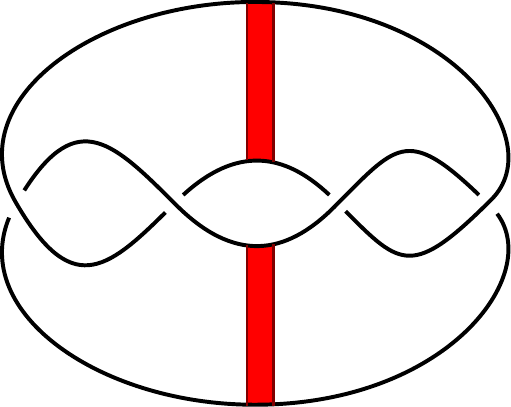}
    \caption{We perform a finger move along an arc representing group element $g_i$ to introduce an unlinked pair of type I circles in $L$ representing $g_i$ Then we perform a finger move along bands as shown to replace this unlink with two Hopf links, each of which is a dual pair corresponding to $g_i$.}
    \label{fig:addtwohopfs}
\end{figure}
In Figure~\ref{fig:addtwohopfs}, we show how to add two split Hopf links of dual type I pairs to $L$ with both active circles associated to $g_i$. 
%e illustrate a procedure to add two Hopf links to $L$ in Figure~\ref{fig:addtwohopfs}.
In words, we introduce a split 2-component unlink $B\sqcup B'$ to $L$ consisting of a type I pair associated to $g_i$ via a finger move. We draw a pair of bands $b_1,b_2$ attached to $B$ and $B'$ in Figure~\ref{fig:addtwohopfs} with the property that band surgery to $B\sqcup B'$ along $b_1$ and $b_2$ yields two Hopf links split from the rest of $L$. These bands are framed for some clean Whitney disk by the same argument we used in part~\ref{hopfstep2} of the proof of Proposition~\ref{prop:LHopflink}: we easily check that $b_1,b_2$ satisfy the orientability conditions of Remarks~\ref{rem:checkannulus} and~\ref{rem:l1l2}, implying that either these bands are framed or they would be framed if a whole twist were added to $b_1$, but adding a whole twist to $b_1$ and performing the band surgery would yield a link violating the conclusion of Lemma~\ref{lem:stong1}. We then modify $H$ by performing a clean Whitney move along $b_1$ and $b_2$, yielding two Hopf links both associated to $g_i$. %Now $L$ is a split union of three Hopf links. One active circle is associated to $g$ and the other two to $g^{-1}$.

Perform the above move $|w|/2$ times, where $|w|$ represents the length of $w$ and we vary $i$ so that for each $i$ we introduce $n_i^++n_i^-$ Hopf links with active circle corresponding to $g_i$. In $n_i^-$ of these dual pairs, reverse the roles of the active and inactive circles so now these Hopf links have active circle associated to $g_i^{-1}$.

Finally, merge these $|w|$ Hopf links together to form one Hopf link split from the rest of $L$ %rther to merge these two Hopf links
as in the proof of Proposition~\ref{prop:LHopflink} (part~\ref{hopfstep4}). Choose the order in which the Hopf links are merged to ensure that the active circle of the resulting Hopf link is associated to the element of $\pi_1 X$ represented by $w$, i.e., $g$ %to , resulting in a single Hopf link associated to $g^{-1}g=1$ 
(see Remarks~\ref{rem:combinehopf} and~\ref{rem:hopfpi1}).
\end{proof}

Lemma~\ref{lem:addHopfnullhomologous} lets us further simplify $H$ in Proposition~\ref{prop:LHopflink} in the case that the single active circle in $L$ is associated to an element of $\pi_1 X$ representing the trivial element of $H_1(X;\mathbb{Z}/2\mathbb{Z})$.

\begin{proposition}\label{prop:trivialremove}
Let $S_0, S_1, G, H$ be as in Proposition~\ref{prop:LHopflink} where the singular link $L$ is a Hopf link of type I circles with corresponding group element $g\in\pi_1 X$. Suppose $g$ represents the trivial element of $H_1(X;\mathbb{Z}/2\mathbb{Z})$. Then $S_0$ and $S_1$ are concordant.%Then we can modify $H$ to be an embedding, concluding that $S$
\end{proposition}

\begin{proof}
%In this section, we assume that there is an immersed 2-sphere $G$ in $X$ intersecting $S_0$ transvsersely once, that $\fq(S_0,S_1)=0$, and that $w_2(G)=1$.  We will further assume that we have an immersion $H: S^2 \times I \to X \times I$, as before, and that, as in the previous section, $H$ only has a Hopf link of singular circles.  Furthermore, we will assume that the group element of this Hopf link is in the kernel of the map $\pi_1 X \to H_1(X; \mathbb{Z}/\mathbb{Z})$.  In this section, we show how under these conditions, we can alter $H$ to remove the Hopf link of self-intersections and obtain a cobordism.  

Using Lemma~\ref{lem:addHopfnullhomologous}, we modify $H$ to add a second Hopf link to $L$ associated to $g^{-1}$. Merge these two Hopf links as in Figure~\ref{fig:combinehopf} so that $L$ is a single Hopf link consisting of a dual type I pair with active circle associated to $g^{-1}g=1$.

Again apply Lemma~\ref{lem:addHopfnullhomologous} to further modify $H$ to add a second Hopf link to $L$ with active circle associated to the trivial element of $\pi_1 X$. Now $L$ is a split union of two Hopf links, each consisting of a dual type I pair with active circle associated to $1\in\pi_1 X$.
%pply the operation in Figure~\ref{fig:addtwohopfs} and the operation in Figure~\ref{fig:combinehopf}, we modify $H$ by adding a second Hopf link to $L$ corresponding to the trivial element of $\pi_1 X$. Namely, we first introduce two Hopf links with the trivial group element as in Figure~\ref{fig:addtwohopfs}. Then we  combine these two Hopf links in to one Hopf link as in Figure~\ref{fig:combinehopf} to obtain the desired single additional Hopf link with trivial group element (this is where it is important that we have the trivial group element).  

%Now that our singular circle link $L$ is a pair of Hopf links of type I singular circles with trivial group elements, we are at the starting point of Figure~\ref{fig:removetwohopfs}.  By
Consider the framed bands $b_1$ and $b_2$ attached to $L$ pictured in the top left of Figure~\ref{fig:removetwohopfs}. We indicate numbers $m,n$ of half-twists in $b_1, b_2$ (we can take both $m,n\in\{0,1\}$ but this is not important for this discussion). In the top and middle row of Figure~\ref{fig:removetwohopfs}, we show how to perform finger moves, clean Whitney moves, and ambient surgeries to $H$ to reduce $n$ by one. Using symmetry, we can thus performing this operation until $m=n=0$. Now perform a clean Whitney move along $b_1,b_2$, so that the singular link $L$ becomes a 2-component unlink which we can remove by ambient surgery as illustrated in the bottom row of Figure~\ref{fig:removetwohopfs}. The map $H$ is now an embedding, so $S_0$ and $S_1$ are concordant.% arrange for $m$ and $n$ to be 
%Now performing the sequence of operations in Figure~\ref{fig:removetwohopfs}, we obtain an $H$ with two dual singular circles that form an unlink, which we then remove via an ambient Dehn surgery to obtain the desired concordance.  
\end{proof}

\begin{figure}
\labellist
\small\hair 2pt
\pinlabel{$m$} at 9 535
\pinlabel{$n$} at 93 535
\pinlabel{$m$} at 212 535
\pinlabel{$n$} at 296 535
\pinlabel{$m$} at 416 535
\pinlabel{$n$} at 500 535
\pinlabel{$m$} at 9 304
\pinlabel{$n$} at 93 304
\pinlabel{$m$} at 212 304
\pinlabel{$n$} at 296 304
\pinlabel{$m$} at 416 304
\pinlabel{$n-1$} at 500 304
\pinlabel{finger} at 145 545
\pinlabel{add} at 372 560
\pinlabel{bands} at 372 545
\pinlabel{Whitney} at 250 427
\pinlabel{surgery} at 172 330
\pinlabel{2$\times$} at 172 345
\pinlabel{isotopy} at 355 330
\pinlabel{Whitney} at 152 100
\pinlabel{isotopy} at 355 100
\endlabellist
    \centering
    \includegraphics[width=\textwidth]{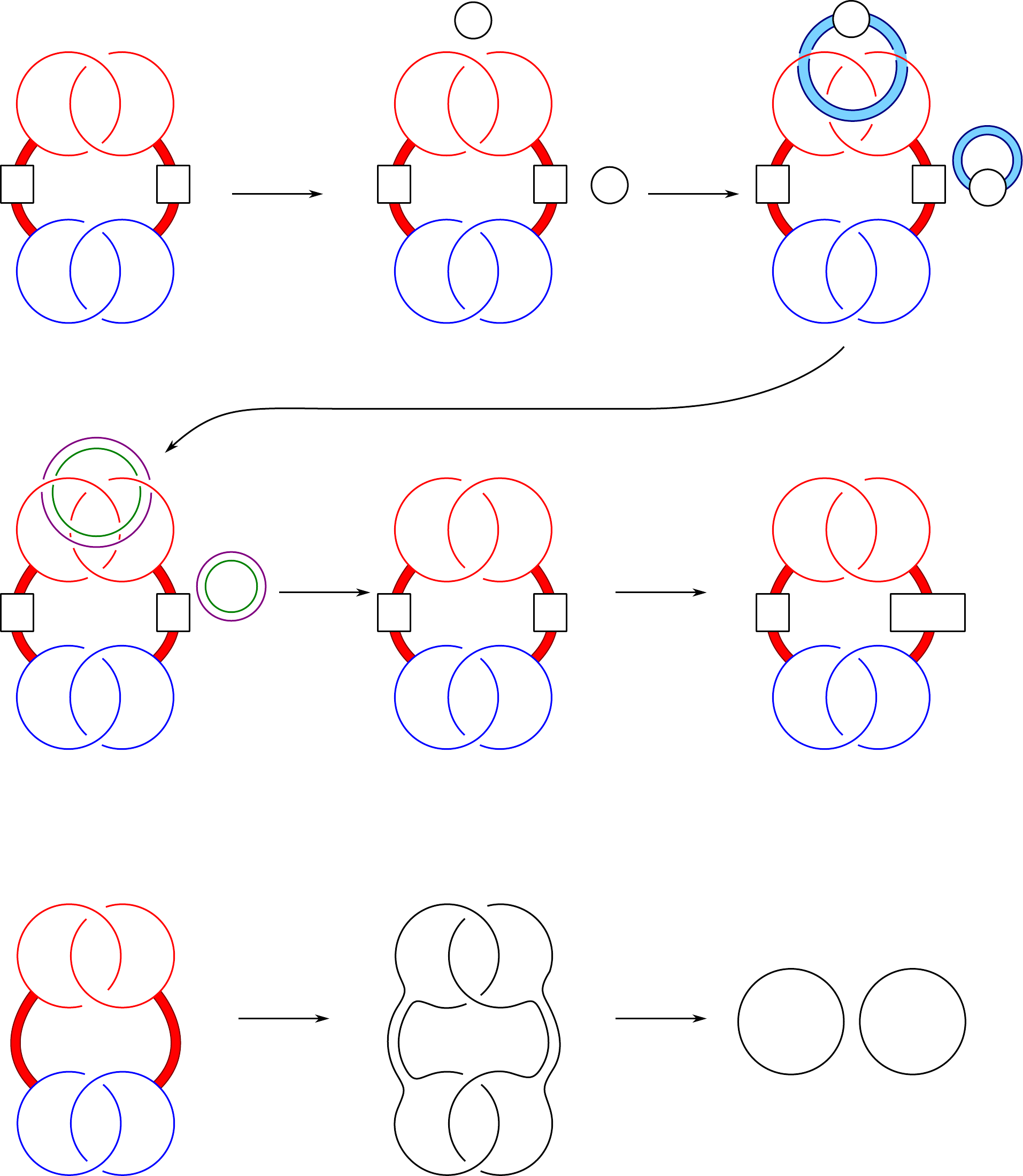}
    \caption{In the top left we show two valid bands for a Whitney move. Here $n$ and $m$ denote numbers of half-twists. Following the arrows, we add a negative half-twist to one of the bands. By repeating this argument (perhaps mirroring or reversing signs), we obtain the bottom left diagram. Performing this Whitney move transforms the original two Hopf links into a 2-component unlink of dual type I circles, which we can remove by performing one ambient surgery.}%We begin by using the operation in Figure~\ref{fig:addtwohopfs} to add two Hopf links with trivial group element.  In the top left part of this figure we show two valid bands for a Whitney move which we perform (we continue to draw the bands in the remaining frames as bands).  In the second frame, we have introduced an unlink by using a finger move around a group element (say the identity).  In the third frame, we have done another Whitney move and in the fourth frame we perform two surgeries (each of which results in a single twist of the components of the red Hopf link).  Finally, an isotopy gives the final result.}
    \label{fig:removetwohopfs}
\end{figure}
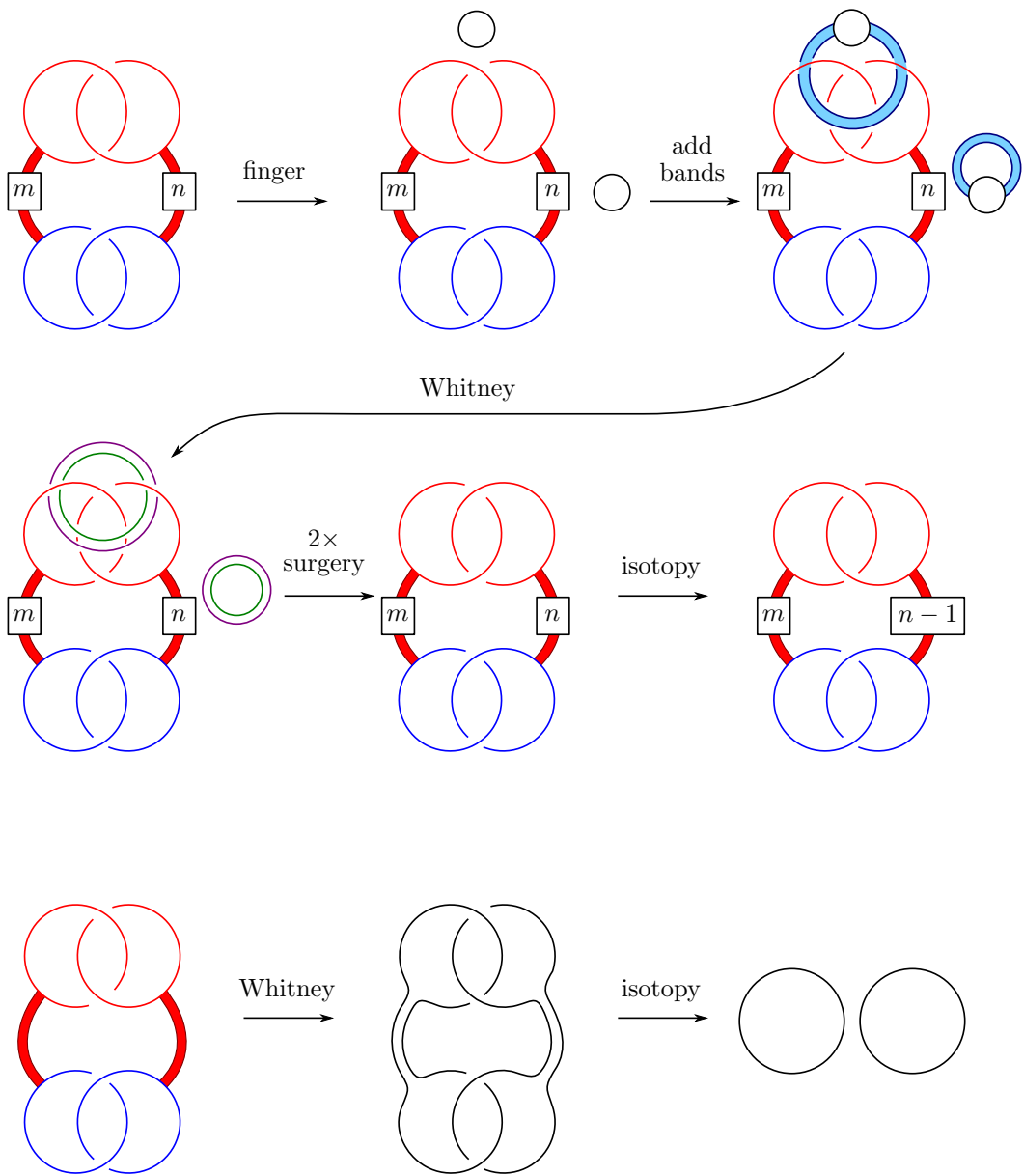

We now are in position to prove Theorem~\ref{thm:main3} (and then will conclude Theorem~\ref{thm:main2} as a corollary).

\begin{mainthm3}
Suppose that $S_0$ and $S_1$ are embedded, oriented, homotopic 2-spheres in an orientable 4-manifold $X$ such that $S_0$ has an immersed dual sphere $G$ in $X$ and $S_0$ is $s$-characteristic. Then $S_0$ and $S_1$ are concordant if and only if there exists a singular concordance $H$ from $S_0$ to $S_1$ so that $\mu(H)=0$ and $\Delta(H)=0$. 
\end{mainthm3}

%\begin{mainthmpart2}[\cite{stong}]
%Suppose that $S_0$ and $S_1$ are embedded homotopic spheres in an orientable 4-manifold $X$ such that $S_0$ has an immersed dual sphere $G$ in $X$ (i.e., $G$ and $S_0$ intersect in a single point). Assume that $S_0$ is $s$-characteristic. Then $S_0$ and $S_1$ are concordant if and only if $\fq(S_0,S_1)=0$ and $\km(S_0,S_1)=0$.
%\end{mainthmpart2}

\begin{proof}
If $S_0$ and $S_1$ are concordant, then a concordance $H$ between them satisfies $\mu(H)=0$ and $\Delta(H)=0$, so the ``only if" direction is trivial.

Now suppose $H$ is a singular concordance from $S_0$ to $S_1$ with $\mu(H)=0$ and $\Delta(H)=0$.
%\red{Now suppose that $\fq(S_0,S_1)=0$ and $\km(S_0,S_1)=0$.  Let $H$ be a singular concordance between $S_0$ and $S_1$.  Since $\km(S_0,S_1) = 0$ we can, by stacking some $H$ with a suitable self singular concordance  $S : S^2 \times I \to X \times I$ of $S_0$, we obtain a new such map (which we again call $H$) with $\Delta(H) = 0$.}
%\red{If there is no 2-torsion in $\pi_1 X$, then it is immediate that $\mu(H) = 0$.  If $\mu(\pi_3 X)=0$ then again we have $\mu(H) = 0$ since $\fq(S_0,S_1) = 0$.  Therefore, by
By Proposition~\ref{prop:LHopflink}, there is a singular concordance $H'$ from $S_0$ to $S_1$ whose singular link $L$ is a Hopf link consisting of a dual type I pair and $[H']=[H]$ in $H_3(X \times I, X \times \{0,1\}, \mathbb{Z}\pi_1X)$. Then $\Delta(H')=\Delta(H)=0$. Therefore, the element of $\pi_1(X)$ associated to the active singular circle of the singular link of $H'$ represents $0$ in $H_1(X;\mathbb{Z}/2\mathbb{Z})$. Applying Proposition~\ref{prop:trivialremove}, we find that $S_0$ and $S_1$ are concordant.
\end{proof}

Theorem~\ref{thm:main2} follows directly from Theorem~\ref{thm:main3} because if $\km(S_0,S_1)=0$, then $\km(S_0,S_1;[H])=0$ for some $H$, and hence $\Delta(H)=0$. % (say, by stacking with an appropriate self-concordance of $S_0$).
If $\mu(\pi_3(X))=0$, then $\fq(S_0,S_1)=0$ implies $\mu(H)=0$ automatically.%\pi_1X$ has no 2-torsion, then $\mu(H)=0$ automatically. On the other hand, if $\mu(\pi_3 X)=0$, then $\fq(S_0,S_1) = 0$ implies $\mu(H) = 0$ for any singular concordance $H$ between $S_0$ and $S_1$.  

\section{Additional discussion} \label{sec:example}

In this section, we give an example from our earlier work in~\cite{km} of a pair of 2-spheres $S_0,S_1$ with $\km(S_0,S_1)\neq 0$ and discuss the relevance of realizing the Stong invariant with obstructing sliceness of spherical links. In addition, we remark on the finiteness of certain concordance classes of knotted spheres.

\begin{example} \label{km_example}
In~\cite{km}, we show how to produce 2-spheres $S_0,S_1$ with nontrivial Stong invariant. In~\cite{km} we omit discussion of the quotient of the target by $\Delta(\Self(S_0))$, so we discuss it in greater detail here.

Let $X$ be a 4-manifold, $S$ be an $s$-characteristic 2-sphere in $X$, and $\alpha\in H_1(X;\mathbb{Z}/2)$ nontrivial and not in the image of $\Self(S)$ under $\Delta$. In practice, it is often not difficult to check that $\alpha$ is not in $\Delta(\Self(S))$; note that the relative long exact sequence in homology includes the following.
\[H_3(X; \mathbb{Z}\pi_1 X) \xrightarrow{\pi} H_3(X \times I, X \times \{0,1\}; \mathbb{Z} \pi_1 X)\xrightarrow{\partial}H_2(X \times \{0,1\};\mathbb{Z}\pi_1 X).\]

Therefore, given any two $H, H' \in \Self(S)$, since we have $\partial(H) = \partial(H')$ we must have $[H] = [H'] + \pi(x)$ for some $x \in H_3(X;\mathbb{Z}\pi_1 X)$. We could now for example consider one of the following situations.
\begin{itemize}
    \item $X=B^3\times S^1$, $S$ any 2-sphere, $\alpha$ the generator of $H_1(B^3\times S^1;\mathbb{Z}/2\mathbb{Z})$. Since $H_3(X;\mathbb{Z}\pi_1 X)=0$, we have $\Delta(\Self(S))=0$.
    \item $X=S^3\times S^1$, $S$ any 2-sphere, $\alpha$ the generator of $H_1(S^3\times S^1;\mathbb{Z}/2\mathbb{Z})$. We have $H_3(X;\mathbb{Z}\pi_1 X)\cong\mathbb{Z}$, with generator represented by an embedded 3-sphere. We conclude that every element of $\Self(S)$ is homologous (with $\mathbb{Z}\pi_1 X$ coefficients) to an embedded self-concordance of $S$, and hence $\Delta(\Self(S))=0$.
\end{itemize}

%\mathbb{Z}[\mathbb{Z}]$.  $S_0$ is isotopic to one of the $\mathbb{C}P^1$'s and we can tube on the the straight concordance $[S_0 \times I]$ any number of disjoint $S^3$ factors from the $S^3 \times S^1$, suitably weighted by tubes going around the $S^1$ direction an arbitrary number of times.   These singular self-concordances will have trivial $\Delta$ since the singular link is an unlink.  This allows us to obtain an singular concordance for every homology class in the image of 

In~\cite[Example 7.2]{km}, we show how to produce a 2-sphere $S_1$ homotopic to $S_0:=S\#\mathbb{CP}^1\#\overline{\mathbb{CP}}^1$ embedded in $X\#\mathbb{CP}^2\#\overline{\mathbb{CP}}^2$ with $\km(S_0,S_1)=\alpha$. (As usual, if $S_0$ is smooth, then we may take $S_1$ to be smooth.) The 2-sphere $S_1$ is obtained from $S_0$ by performing a finger move about $\alpha$ followed by one Whitney move, chosen so that the trace of the described homotopy has singular link a Hopf link. The Whitney disk can almost be found in $X$: in~\cite{km} we perform a finger move to $S$ in $X$ to obtain an immersed sphere $S'$ and exhibit a framed Whitney disk $W'$ that intersects $S'$ in two points. We remove these two intersection points by blowing up twice (with opposite signs to ensure the resulting neat Whitney disk is still framed).

 %with $\km(S_0,S_1)=\alpha$. 
 Since \[H_3(X\#\mathbb{CP}^2\#\overline{\mathbb{CP}}^2;\mathbb{Z}\pi_1 (X\#\mathbb{CP}^2\#\overline{\mathbb{CP}}^2))= H_3(X;\mathbb{Z}\pi_1 X)\] and $\alpha\not\in\Delta(\Self(S))$, we also have $\alpha\not\in\Delta(\Self(S_0))$. Thus, $\km(S_0,S_1)$ is nontrivial in the quotient \[H_1(X\#\mathbb{CP}^2\#\overline{\mathbb{CP}}^2;\mathbb{Z}/2\mathbb{Z})/\Delta(\Self(S_0))\cong H_1(X;\mathbb{Z}/2\mathbb{Z})/\Delta(\Self(S)).\] Thus, $S_0$ and $S_1$ are not concordant.

\end{example}

\begin{remark}
In Example~\ref{km_example}, if $S$ already has two embedded dual spheres with Euler numbers $2n+1$ and $-2n-1$, we can avoid blowing up. We instead have $S_0=S$ and $S_1$ obtained from $S$ by a finger and Whitney move, with $\km(S_0,S_1)=\alpha$.
\end{remark}

The Stong invariant is particularly interesting in light of its potential relevance to the study of concordance of links. While Kervaire~\cite{kervaire} showed that every 2-sphere in $S^4$ is concordant to the unknotted 2-sphere, it remains open whether every link of 2-spheres in $S^4$ is concordant to the unlink.

\begin{slicelink}
%Let $S_0$ be the unknotted 2-sphere in $B^3\times S^1$.
If there are 2-spheres $S_0,S_1$ in $B^3\times S^1$ with $\km(S_0,S_1)=1$ in $H_1(B^3\times S^1;\mathbb{Z}/2\mathbb{Z})\cong\mathbb{Z}/2\mathbb{Z}$ then there is a 2-component link of spheres in $S^4$ that is not concordant to the unlink.
\end{slicelink}

\begin{proof}
Note that we have not strictly defined the Stong invariant for a pair of 2-spheres that are not based-homotopic nor have a dual sphere. However, there is a self-concordance $H$ of $S_0$ so that $H($basepoint$\times I)\subset B^3\times S^1\times I$ projects to a loop in $B^3\times S^1$ representing the generator of $\pi_1(B^3\times S^1)$. (In words, take $H$ to be an isotopy that moves each $B^3$ factor once about the $S^1$ factor of $B^3\times S^1$.)

Take $S_0$ and $S_1$ to have a common basepoint. Given any singular concordance $H'$ from $S_0$ to $S_1$, we can preconcatenate $H'$ with some number of copies of $H$ and obtain a singular concordance isotopic rel.\ boundary to a based singular concordance with the same singular link as $H'$. Therefore, if $\km(S_0,S_1)=1$, implying that there is no based concordance from $S_0$ to $S_1$, then there is no concordance from $S_0$ to $S_1$ even without reference to basepoints.

Let $S$ be the unknotted 2-sphere in $B^3\times S^1$. Since $\km(S_0,S_1)\neq 0$, we cannot have both $\km(S,S_0)=0$ and $\km(S,S_1)=0$. Without loss of generality, take $\km(S,S_1)\neq 0$ and redefine $S_0:=S$.

Now let $U$ be the unknotted 2-sphere in $S^4$ and identify $S^4\setminus\nu(U)$ with $B^3\times S^1$. Then $U\sqcup S_0$ is the 2-component unlink. Suppose there are concordances $H_U,H:S^2\times I\to S^4\times I$ with disjoint images so that $H_U$ goes from $U$ to $U$ and $H$ goes from $S_0$ to $S_1$. Let $W=S^4\times I\setminus\nu(\Im(H_U))$. Let $W'$ be the intersection of $W$ with $S^4\times[1-\epsilon,1]$ with $\epsilon$ small so that $W'$ can be naturally identified with $B^3\times S^1\times I$. Refer to Figure~\ref{fig:slicelinks} for a schematic.

Since $S_1\subset\partial W'$ is nullhomotopic in $W'$, there is an immersed ball $B$ in $W'$ with boundary $S_1$. We could tube $B$ to a copy of $S_0$ in $\partial W'$ to obtain a concordance from $S_0$ to $S_1$, so we conclude $\Delta(B)=1$. On the other hand, by including $B$ into $W$ and tubing $B$ to $S_0$ in $\partial W$ we also find the Stong invariant of the pair $(S_0,S_1)$ in the boundary of the 5-manifold $W$ is still $\Delta(B)=1\in H_1(W';\mathbb{Z}/2\mathbb{Z})=H_1(W;\mathbb{Z}/2\mathbb{Z})$. This contradicts the existence of $H$.

We plan to discuss the Stong invariant in a general 5-manifold (rather than a concordance invariant defined using a product) in~\cite{km2}. This is already the main focus of Stong's work~\cite{stong}, so this technical point can be sidestepped by translating the above argument into his setting. In~\cite{stong}, one would instead define $\km(S_1)$ to be $\Delta(B)$ for $B$ an immersed 3-ball bounded by $S_1$, rather than by considering any concordance. We conclude that because $\km(S_1)=1$ when viewed as a 2-sphere in $\partial W'$, we also have $\km(S_1)=1$ when viewed as a 2-sphere in $\partial W$ since $H_1(W;\mathbb{Z}/2\mathbb{Z})$ and $H_1(W';\mathbb{Z}/2\mathbb{Z})$ are identified.
\end{proof}

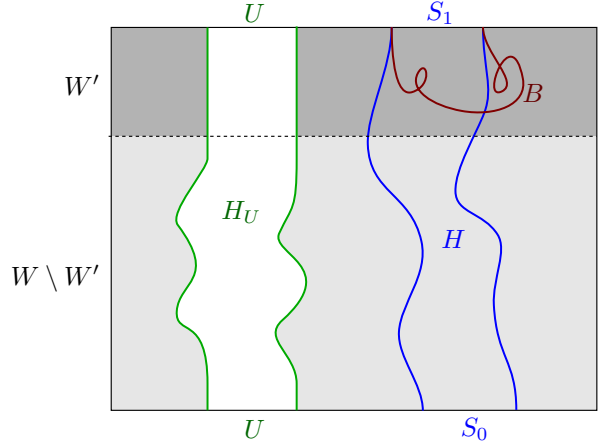
\begin{figure}
\labellist
\small\hair 2pt
\pinlabel{$\textcolor{darkgreen}U$} at 55 -7
\pinlabel{$\textcolor{darkgreen}U$} at 55 150
\pinlabel{$\textcolor{blue}{S_0}$} at 138 -7
\pinlabel{$\textcolor{blue}{S_1}$} at 125 150
\pinlabel{$\textcolor{darkgreen}{H_U}$} at 50 75
\pinlabel{$\textcolor{blue}{H}$} at 130 65
\pinlabel{$\textcolor{darkerred}{B}$} at 160 120
\pinlabel{$W'$} at -10 123
\pinlabel{$W\setminus W'$} at -20 50
\endlabellist
\vspace{.1in}
    \centering
    \includegraphics{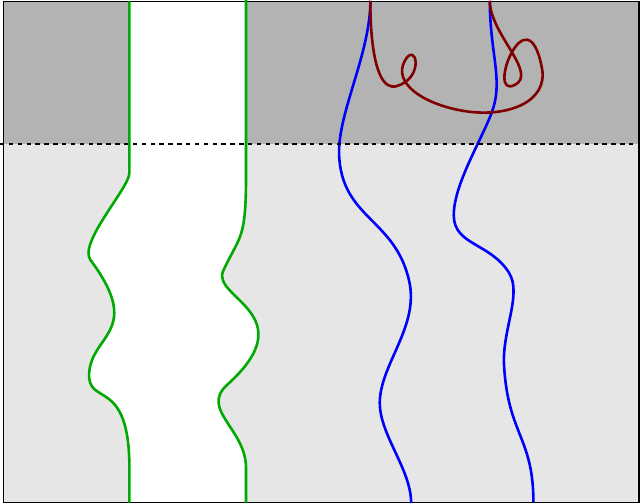}
    \caption{A schematic of the proof of Theorem~\ref{thm:slicelink}.}
    \label{fig:slicelinks}
\end{figure}

It is tempting to think that Theorem~\ref{thm:slicelink} can be restated for concordance of links whose components represent nontrivial homology elements in a general simply connected 4-manifold. For instance, via Example~\ref{km_example} we can obtain 2-spheres $S_0,S_1$ in $X:=B^3\times S^1\#\mathbb{CP}^2\#\overline{\mathbb{CP}}^2$ in the homology class $[\mathbb{CP}^1\#\overline{\mathbb{CP}}^1]$ with $\km(S_0,S_1)$ nontrivial. Letting $U$ denote the unknotted 2-sphere in $\mathbb{CP}^2\#\overline{\mathbb{CP}}^2$ and identifying $X$ with $\mathbb{CP}^2\#\overline{\mathbb{CP}}^2\setminus\nu(U)$, one can ask whether the link $U\sqcup S_0$ is concordant to $U\sqcup S_1$. We know from work of Sunukjian~\cite{sunukjian} that $S_0$ and $S_1$ are concordant inside $\mathbb{CP}^2\#\overline{\mathbb{CP}}^2$, so this is truly a question about link concordance. Unfortunately, it is not clear how to make use of the fact that $\km(S_0,S_1)\neq 0$ inside $X\times I$ -- this just informs us that any concordance from $S_0$ to $S_1$ in $(\mathbb{CP}^2\#\overline{\mathbb{CP}}^2)\times I$ must intersect $U\times I$. In theory there could be a concordance from $U\sqcup S_0$ to $U\sqcup S_1$ in which the self-concordance $H$ of $U$ is nontrivial. In order to apply an argument analogous to that of Theorem~\ref{thm:slicelink}, it would be necessary to find a singular concordance from $S_0$ to $S_1$ inside $(
\mathbb{CP}^2\#\overline{\mathbb{CP}}^2)\times I$ that is disjoint from both the image of $H$ and $U\times I$, which in general seems to be a nontrivial task.

We also give the following application to concordance classes of 2-spheres in 4-manifolds.

\begin{finitecor}
Let $S_0$ be an oriented embedded 2-sphere in an orientable 4-manifold $X$ with an immersed dual sphere $G$.  Let $\Concordance(S_0)$ be the set of concordance classes of embedded spheres in $X$ that are homotopic to $S_0$. Suppose that $\pi_1 X$ has a finite number of 2-torsion elements.
\begin{enumerate}
    \item Suppose $S_0$ is not $s$-characteristic.  Then $\Concordance(S_0)$ is finite of size at most $2^{|T_X|}$.
    \item Suppose $S_0$ is $s$-characteristic and that $\mu(\pi_3 X) = 0$.  Then $\Concordance(S_0)$ is finite of size at most $2^{|T_X|} \cdot |H_1(X; \mathbb{Z}/2\mathbb{Z})|$.
\end{enumerate}
\end{finitecor}

\begin{proof}
Let $S_1$ and $S_2$ be embedded 2-spheres in $X$ with $\fq(S_0,S_1) = \fq(S_0,S_2)$.  By concatenating singular concordances $H_1$ from $S_1$ to $S_0$ and $H_2$ from $S_2$ to $S_0$, we find $\fq(S_1,S_2) = \fq(S_1,S_0) + \fq(S_0,S_2) = 0$. Note that $G$ is a dual sphere for the singular concordance from $S_1$ to $S_2$ obtained by stacking $S_0$ to $S_1$.

Thus when $S_0$ is not $s$-characteristic, by the proof of Theorem~\ref{thm:main1} we find that $S_1$ and $S_2$ are concordant and obtain $|\Concordance(S_0)| \leq |\mathbb{F}_2T_X|=2^{|T_X|}\cdot$.

Similarly, if $S_0$ is $s$-characteristic and $\mu(\pi_3 X)=0$, we see again by apply the proof of Theorem~\ref{thm:main2} (using the fact that $G$ is a dual sphere to some singular concordance $H$ from $S_1$ to $S_2$ with $\mu(H)=0$) that $S_1$ and $S_2$ are concordant exactly when $\km(S_1,S_2)=0$. Thus $|\Concordance(S_0)|\le|\mathbb{F}_2T_X|\cdot |H_1(X;\mathbb{Z}/2\mathbb{Z})|=2^{|T_X|} \cdot |H_1(X; \mathbb{Z}/2\mathbb{Z})|$.
%that $\fq(S_1, S_2) = 0$ and $\km(S_1,S_2) = 0$.  Furthermore, there is a dual sphere to the image of $H : S^2 \times I \to X \times I$ that results from this stacking (namely, in the middle where $S_0$ is).  The existence of such a dual to the image of $H$ is all that was used in the proof of Theorem~\ref{thm:main2}, and therefore in this case we can simplify $H$ to obtain a concordance.  Therefore, in this case we have $|\Concordance(S_0)| \leq 2^{|T_X|} \cdot |H_1(X; \mathbb{Z}/2\mathbb{Z})|$, as desired.  
\end{proof}

\bibliographystyle{plain} 
\bibliography{biblio}

\end{document}